%% file: iclr2026_conference.tex
\documentclass{article} 
\usepackage{iclr2026_conference_arxiv_style,times}

\input{math_commands.tex}

\usepackage{hyperref}
\usepackage{url}

\usepackage{mathtools}
\usepackage{amsthm}
\usepackage{amssymb}
\usepackage{subcaption}

\newtheorem{definition}{Definition}
\newtheorem{theorem}{Theorem}
\newtheorem{lemma}[theorem]{Lemma}
\newtheorem{remark}{Remark}

\newtheorem{corollary}[theorem]{Corollary}

\title{Global Convergence of Four-Layer Matrix Factorization under Random Initialization}

\author{Minrui Luo \thanks{Work done while Minrui Luo was visiting the University of Washington. } \\
Institute for Interdisciplinary Information Sciences\\
Tsinghua University\\
Beijing, China \\
\texttt{luomr22@mails.tsinghua.edu.cn} \\
\And
Weihang Xu \\
Paul G. Allen School of Computer Science and Engineering \\
University of Washington \\
Seattle, WA 98105, USA \\
\texttt{xuwh@cs.washington.edu} \\ 
\And
Xiang Gao \\
Institute for Interdisciplinary Information Sciences\\
Tsinghua University\\
Beijing, China \\
\texttt{x-gao22@mails.tsinghua.edu.cn} \\ 
\And
Maryam Fazel \\
Electrical and Computer Engineering\\
University of Washington \\
Seattle, WA 98105, USA \\
\texttt{mfazel@uw.edu} \\
\And
Simon Shaolei Du \\
Paul G. Allen School of Computer Science and Engineering \\
University of Washington \\
Seattle, WA 98105, USA \\
\texttt{ssdu@cs.washington.edu} 
}

\begin{document}

\maketitle

\begin{abstract}

Gradient descent dynamics on the deep matrix factorization problem is extensively studied as a simplified theoretical model for deep neural networks. Although the convergence theory for two-layer matrix factorization is well-established, no global convergence guarantee for general deep matrix factorization under random initialization has been established to date. 
To 
address this gap, we provide a polynomial-time global convergence guarantee for randomly initialized gradient descent on four-layer matrix factorization, given certain conditions on the target matrix and a standard balanced regularization term. 
Our analysis employs new techniques to show saddle-avoidance properties of gradient decent dynamics, and extends previous theories to characterize the change in eigenvalues of layer weights. 

\end{abstract}

\input{subfiles/1.intro}

\input{subfiles/2.related_work}

\input{subfiles/3.problem_formulation}

\input{subfiles/4,5.proof_sketch}

\input{subfiles/6.conclusions}

\input{subfiles/7.Reproducibility_Statement}


\subsubsection*{Acknowledgments}

We thank Haoxin Li from University of Science and Technology of China for valuable discussions on Gaussian initializations, especially on the circular ensembles part. We thank Yiran Zhang from Tsinghua University for valuable discussions on the regularization terms. 

\bibliography{iclr2026_conference}
\bibliographystyle{iclr2026_conference}

\newpage

\appendix

\input{subfiles/appendix/appendix_1_reduction_to_diagonal}

\input{subfiles/appendix/appendix_2_initialization}

\input{subfiles/appendix/appendix_3_basic_lemmas}

\input{subfiles/appendix/appendix_4_balanced}

\input{subfiles/appendix/appendix_5_notations_and_preliminaries}

\input{subfiles/appendix/appendix_6_skew_and_main}

\input{subfiles/appendix/appendix_7_staged,gf}

\input{subfiles/appendix/appendix_8_staged,gd}

\input{subfiles/appendix/appendix_9_numerical_simulations}

\input{subfiles/appendix/appendix999.LLM_usage_declaration}

\end{document}

%% file: math_commands.tex
\usepackage{amsmath,amsfonts,bm}

\def\eqref#1{equation~\ref{#1}}

\def\1{\bm{1}}

\DeclareMathAlphabet{\mathsfit}{\encodingdefault}{\sfdefault}{m}{sl}
\SetMathAlphabet{\mathsfit}{bold}{\encodingdefault}{\sfdefault}{bx}{n}



%% file: subfiles/1.intro.tex
\section{Introduction}

This paper investigates matrix factorization, a fundamental non-convex optimization problem, which in its canonical form seeks to optimize the following objective:
\begin{equation}\label{eq: loss}
    \mathcal{L}(W_1,\ldots, W_N)\coloneqq \frac{1}{2}\left\|W_N\cdots W_1 -\Sigma\right\|_{F}^2+\mathcal{L}_{\rm reg}(W_1,\ldots, W_N),
\end{equation}

where $W_{j} \in \mathbb{F}^{d \times d}$ denotes the $j^{\text{th}}$ layer weight matrix, $\Sigma \in \mathbb{F}^{d \times d}$ denotes the target matrix and $\mathcal{L}_{\rm reg}$ is a (optional) regularizer. Here $\mathbb{F}\in \{\mathbb{C}, \mathbb{R}\}$ as we consider both real and complex matrices in this paper. Following a long line of works \citep{arora2019convergenceanalysisgradientdescent, jiang2023algorithmic, ye2021globalconvergencegradientdescent, chou2024gradient}, we aim to understand the dynamics of gradient descent (GD) on this problem:
\begin{equation}\label{eq: GD dynamics}
    j=1,\ldots, N: W_j(t+1)=W_j(t)-\eta\nabla_{W_j} \mathcal{L}(W_1(t),\ldots, W_N(t)),
\end{equation}

where $\eta\in\mathbb{R}^+$ is the learning rate.

While global convergence guarantee for the case of two-layer matrix factorization ($N=2$) is well studied \citep{du2018algorithmicregularizationlearningdeep, ye2021globalconvergencegradientdescent, jiang2023algorithmic}, the deep matrix factorization problem, $i.e.$, the $N>2$ case is less explored. While the model representation power is independent of depth $N$, the deep matrix factorization problem is naturally motivated by the goal of understanding benefits of depth in deep learning (see, $e.g.$, \cite{arora2019implicitregularizationdeepmatrix}). A long line of previous works  \citep{hardt2016identity, arora2019implicitregularizationdeepmatrix, arora2019convergenceanalysisgradientdescent, wang2023implicit} studies this regime as it directly captures Deep Linear Networks (DLN), the simplest type of deep neural networks.  However, a general global convergence guarantee is still missing. Therefore, the following open research question can be naturally asked:
\begin{center}
    \emph{Can we prove global convergence of GD for matrix factorization problem (\ref{eq: loss}) with $N>2$ layers?}
\end{center}

In this paper, we provide a positive answer to the question above. Specifically, we consider $4$-layer matrix factorization $(N=4)$ with the standard balancing regularization term (see \cite{park2017non, ge2017spuriouslocalminimanonconvex, zheng2016convergence}) as
\[\mathcal{L}(W_1, W_2, W_3, W_4)\coloneqq \frac{1}{2}\left\|W_4W_3W_2W_1-\Sigma\right\|_{F}^2+ \frac{1}{4} a  \left( \sum_{j=1}^{3} \left\| W_j W_j^H - W_{j+1}^H W_{j+1} \right\|_F^2 \right),\]
where $W_j^H$ denotes the Hermitian transpose of $W_j$ and $a\in \mathrm{R}^+$ is a hyperparameter. We consider both real $(\mathbb{F}=\mathbb{R})$ and complex $(\mathbb{F}=\mathbb{C})$ setting with random Gaussian initialization and prove global convergence of gradient descent. Our main result can be summarized as follows:

\begin{theorem}[Main theorem, informal]\label{Total convergence bound, gd, informal}  
Consider four-layer matrix factorization under gradient descent, random Gaussian initialization with scaling factor $\epsilon \le \sigma_1^{1/4}(\Sigma) / {\rm poly}( 1 / \delta, d)$, regularization factor $a \ge \sigma_1(\Sigma) \cdot {\rm poly}\left( 1 / \delta, d, \ln\left(\sigma_1^{1/4}(\Sigma)/\epsilon\right)\right)$, where $\sigma_1(\Sigma)$ denotes the largest singular value of target matrix $\Sigma$. Then for $\Sigma$ with identical singular values, there exists learning rate $\eta = O\left( 1 / \left[\sigma_1^{3/2}(\Sigma) \cdot {\rm poly}\left(a/\sigma_1(\Sigma), 1 / \delta, d, \sigma_1^{1/4}(\Sigma) / \epsilon\right)\right]\right)$ and convergence time $T(\epsilon_{\rm conv},\eta) = \eta^{-1} \sigma_1^{-3/2}(\Sigma) \cdot {\rm poly}\left(1 / \delta, d, \sigma_1^{1/4}(\Sigma) / \epsilon, \ln\left(d \sigma_1^2(\Sigma) / \epsilon_{\rm conv}\right)\right)$, such that for any $\epsilon_{\rm conv} > 0$, (1) with high probability $1 - \delta$ over the complex initialization, or (2)  with probability close to $\frac{1}{2}(1-\delta)$ over the real initialization, when $t > T(\epsilon_{\rm conv},\eta)$, $\mathcal{L}(t) < \epsilon_{\rm conv}$. 
    
\end{theorem}

The formal version of Theorem \ref{Total convergence bound, gd, informal} is stated in Theorem \ref{Total convergence bound, gd} in Appendix. 

\begin{remark}

A natural question is why the convergence guarantee in the real case holds only with probability close to $\frac{1}{2}$, but not $1$. For the other $\frac{1}{2}$ probability, Theorem \ref{Total convergence bound, balanced, gf, informal} presents a special case - considering gradient flow under the strict balance condition (which can be viewed as the limit as $a \to +\infty$), showing that the optimization process does not converge to a global minimum in finite time (and hence converges to a saddle point). 

\end{remark}

\noindent\textbf{Main contributions. }
Our major contributions can summarized as follows:
\begin{itemize}
    \item We prove global convergence of GD for $4$-layer matrix factorization under random Gaussian initialization. To the best of our knowledge, this is the first global convergence result for general deep linear networks under random initialization beyond the NTK regime in \cite{du2019width}. This result helps provide new insights towards understanding the training dynamics of general deep neural networks.

    \item We construct a novel three-stage convergence analysis of gradient descent dynamics, consisting of an alignment stage, a saddle-avoidance stage, and a local convergence stage. We also develop new techniques to show GD dynamics avoids saddle points and to characterize layer matrix eigenvalue changes, which we believe are of independent interest for deep linear networks analysis.
\end{itemize}

\noindent\textbf{Challenges and techniques. }
Our analysis employs the following key techniques:

\begin{itemize}
    \item Initialization analysis.  To guarantee that gradient descent makes progress, it is necessary to establish a monotonically increasing lower bound for the singular values of the weight matrices. This, in turn, requires analyzing the smallest singular value of a newly introduced term (namely $W+WW^H$, where $W = W_4 W_3 W_2 W_1$), at initialization. This analysis utilizes tools from random matrix theory, particularly the concept of  Circular Ensembles. The detailed proof is given in Appendix \ref{section: initialization}.

    \item Regularity condition of each layer. To bridge the initialization with the subsequent training dynamics, we need to ensure that key matrix properties evolve in a controlled manner even during the rapid changes in the alignment stage. We prove that despite significant updates, the weight matrices retain certain spectral properties from their initial state. A delicate analysis of the smooth evolution of the extreme singular values and the limiting behavior of the Hermitian term after the regularization term converges is provided in Section \ref{subsubsection: Convergence of Regularization term} and \ref{subsubsection: The Limit Behavior of the Hermitian main term}. 
    
    \item Saddle avoidance. To avoid convergence to a saddle point, it is essential to prevent the smallest singular values of the weight matrices from decaying to zero, as such decay would cause the gradient norm to vanish. To this end, we construct a hermitian term providing lower-bounds for these singular values, along with a skew-hermitian error. During the optimization, the skew-hermitian error is approximately non-increasing, which in turn ensures that the minimum singular value of the hermitian term is non-decreasing. This mechanism provides a persistent lower bound, thereby effectively avoiding saddle points. 

    \item Bound of eigenvalue change. Finally, to translate the continuous-time intuition into rigorous guarantees for the discrete gradient descent algorithm, we develop new perturbation bounds for eigenvalues. In continuous time, the time derivatives of eigenvalues are directly characterized by the derivatives of the matrix. In discrete time, however, eigenvalue changes depend on the spectral gap in general, requiring a fine-grained, problem-specific analysis. Similar challenge are noted in Lemma 3.2 of \cite{ye2021globalconvergencegradientdescent}. We address this issue in Lemma \ref{maximum and minimum singular values, general, discrete} and \ref{minimum singular values lower bound, general, discrete} in Appendix \ref{subsection: Lemmas on Eigenvalue Change under Discrete Time}. 
\end{itemize} 

These techniques form a cohesive proof strategy: the initialization analysis provides a favorable starting point; the regularity analysis ensures controlled dynamics throughout training; the saddle avoidance mechanism guarantees persistent progress; and the discrete-time perturbation bounds rigorously translate these insights into a full global convergence proof. 

%% file: subfiles/2.related_work.tex
\section{Related works}

For two-layer matrix factorization, the global convergence of symmetric case has been established under various settings \citep{jain2017globalconvergencenonconvexgradient, li2019algorithmicregularizationoverparameterizedmatrix, Chen_2019}. For asymmetric matrix factorization case with objective $\mathcal{L} = \frac{1}{2} \|UV^\top - \Sigma\|_F^2$, the following homogeneity issue occurs: the prediction result remains the same if one layer is multiplied by a positive constant while the other is divided by the same, introducing significant challenges in convergence analyzing (\cite{lee2016gradientdescentconvergesminimizers}, Proposition 4.11). \cite{tu2016lowranksolutionslinearmatrix} and \cite{ge2017spuriouslocalminimanonconvex} tackles this problem by manually adding a regularization term on the objective function. \cite{du2018algorithmicregularizationlearningdeep} discovers that gradient descent automatically balances the magnitudes of layers under small initialization, providing analysis of global convergence with polynomial time under decayed learning rate, while removing the regularization term. \citet{ye2021globalconvergencegradientdescent} extends the convergence analysis to constant learning rate. 


\cite{kawaguchi2016deeplearningpoorlocal} analyzes landscape for general DLN, showing there exists saddle points with no negative eigenvalues of Hessian for depth over three. \cite{bartlett2018gradientdescentidentityinitialization} analyzes the dynamic under identity initialization, proving polynomial convergence with target matrix near initialization or symmetric positive definite, but such initialization fails to converge when target matrix is symmetric and has a negative eigenvalue. 
\cite{arora2019convergenceanalysisgradientdescent} provides global convergence proof under specific deep linear neural network structures and initialization scheme , requiring the initial loss to be smaller than the loss of any rank-deficient solution. \cite{ji2019gradientdescentalignslayers} conducted the proof of convergence on general deep neural networks with similar requirements on the initial loss. \cite{arora2019implicitregularizationdeepmatrix} simplifies the training dynamics of deep linear neural network into the dynamic of singular values and singular vectors of product matrix under balanced initialization, providing theoretical illustration of local convergence when singular vectors are stationary. \citet{du2019width} proves global convergence for wide linear networks under the neural tangent kernel (NTK) regime. More recent works focus on GD dynamics under (approximately) balanced initialization schemes \citep{min2023convergence} or the $2$-layer case \citep{min2021explicit, xiong2023over, tarmoun2021understanding}. 
\cite{chizat2024infinite} studies the infinite-width limit of DLN in the mean field regime. However, none of these results imply a global convergence guarantee for general DLN with $N>2$ under random initialization.

%% file: subfiles/3.problem_formulation.tex
\section{Preliminaries}\label{section: problem formulation}

\noindent\textbf{Notation. }
Denote the complex conjugate of $M$ as $\bar M$ and adjoint of $M$ as $M^{H}$, $\mathbb{N}$ as the set of non-negative integers, and $\mathbb{N}^*$ as the set of positive integers. For $k_1 < k_2 \in \mathbb{N}$, $\prod_{j=k_2}^{k_1} M_{j} = M_{k_2} M_{k_2-1}\cdots M_{k_1}$. $x \sim \mathcal{N}(0,1)_{\mathbb{C}}$ means that the real and imaginary parts are independently sampled from Gaussian distribution with variance $\frac{1}{2}$: $\Re x, \Im x \overset{\text{i.i.d.}}{\sim} \mathcal{N}(0,1/2)$. $Q \sim U(d,\mathbb{C})$ or $O(d,\mathbb{R})$ means $Q$ is drawn from the unique uniform distribution (Haar measure) on the unitary or orthogonal group, implying its distribution is unitarily/orthogonally invariant. Consider general $N$-layer matrix factorization, for simplicity we define the following notations: 

    \begin{equation}
     \begin{aligned}
      W_{\prod_L , j} &\coloneqq \prod_{k=N}^{j} W_k ,\, W_{\prod_R , j} \coloneqq \prod_{k=j}^{1} W_k,\, W \coloneqq \prod_{k=N}^{1} W_k = W_{\prod_L , 1} = W_{\prod_R , N}, 
      \end{aligned}
    \end{equation}

$W$ is referred to as \textit{product matrix}. The loss is written by $\mathcal{L}(W_1, \cdots,W_N) = \mathcal{L}_{\rm ori} + \mathcal{L}_{\rm reg}$, where $\mathcal{L}_{\rm ori} = \frac{1}{2} \left\| \Sigma - W \right\|_F^2 $, $   \mathcal{L}_{\rm reg} = \frac{1}{4} a  \left( \sum_{j=1}^{N-1} \left\| \Delta_{j,j+1} \right\|_F^2 \right)$. 

\noindent\textbf{Algorithmic setup. } For the real case $(W_{j} \in \mathbb{R}^{d \times d})$, GD dynamics is canonical and described by \eqref{eq: GD dynamics}. Under complex field $(W_{j} \in \mathbb{C}^{d \times d})$, for simplicity and coherence we define $\nabla_{M} = \frac{\partial}{\partial \Re M} + i\frac{\partial}{\partial \Im M}$, which is two times of Wirtinger derivative with $\bar M$: $\frac{\partial}{\partial \bar M} = \frac{1}{2} \left( \frac{\partial}{\partial \Re M} + i \frac{\partial}{\partial \Im M} \right)$. By following the updating rule of complex neural networks (see \cite{guberman2016complexvaluedconvolutionalneural}), the gradient can be uniformly represented by 

    \begin{equation}
     \begin{aligned}
      \nabla_{W_j} \mathcal{L} &= \nabla_{W_j} \mathcal{L}_{\rm ori} + \nabla_{W_j} \mathcal{L}_{\rm reg} \\
      \nabla_{W_j} \mathcal{L}_{\rm ori} &= -W_{\prod_L , j+1}^H \left(\Sigma - W \right) W_{\prod_R , j-1}^H ,\, \nabla_{W_j} \mathcal{L}_{\rm reg} = -a W_j \Delta_{j-1,j} + a \Delta_{j,j+1} W_j ,
     \end{aligned}
    \end{equation}
    
    Under gradient flow, $\frac{\mathrm{d} W_j}{\mathrm{d} t} = - \nabla_{W_j} \mathcal{L}$; under gradient descent, $W_{j}(t+1) = W_j(t) - \eta \nabla_{W_j} \mathcal{L}(t)$. 

\noindent\textbf{Reduction to diagonal target. }Following the simplification process of Section 2.1 in \cite{ye2021globalconvergencegradientdescent}, suppose the singular value decomposition of $\Sigma$ is $\Sigma = U_\Sigma \Sigma^\prime V_\Sigma^H$, by applying the following transformation $W_1\leftarrow W_1 V_\Sigma$ and $W_N \leftarrow U_\Sigma^H W_N$, the dynamics remain the same form, while the distributions of $W_j$ under our initialization schemes remain the same. Hence without loss of generality, we assume the target matrix is \textit{diagonal with real and non-negative entries} throughout our analysis. Detailed analysis is presented in Appendix \ref{Appendix: Reduction To Diagonal (Identical) Target}. 

    For some of the results, we further require target matrix to be \textit{an identity matrix scaled by a positive constant} $\Sigma = \sigma_1(\Sigma) I$, which is equivalent to \textit{requiring the singular values of target matrix are identical}.
    
\noindent\textbf{Balancedness. } Following a long line of works \citep{arora2019convergenceanalysisgradientdescent, arora2019implicitregularizationdeepmatrix, du2018algorithmicregularizationlearningdeep}, we define the balance error between layer $j$ and $j+1$ as 

    \begin{equation}
     \begin{aligned}
      \Delta_{j,j+1} &\coloneqq\begin{cases}
          W_j W_j^H - W_{j+1}^H W_{j+1} &,\, j\in [1,N-1] \cap \mathbb{N}^* \\
          O^{d\times d} &,\, j\in\{0,N\}
      \end{cases}  \quad . 
      \end{aligned}
    \end{equation}

    As discussed in Definition 1 of \cite{arora2019convergenceanalysisgradientdescent}, the weights are approximately balanced (namely $\|\Delta_{j,j+1}\|_F$ are small) throughout the iterations of gradient descent under approximate balancedness at initialization and small learning rate. Notice that approximate balancedness holds for small initialization near origin (small variance for Gaussian initialization). 

    Specifically, under \textit{gradient flow} the balanced condition (defined as $\|\Delta_{j,j+1}\|_F \equiv 0$ or equivalently $\Delta_{j,j+1}\equiv O$, $\forall j \in [1,N-1]\cap\mathbb{N}^*$) \textit{holds strictly at arbitrary time under balanced initialization}, which is defined as $\Delta_{j,j+1}(t=0)\equiv O$, $\forall j \in [1,N-1]\cap\mathbb{N}^*$. 

    \begin{remark}
    
    As previously discussed, balance condition holds approximately under small initialization, so such regularization's affect on the training process is relatively weak, especially when weight matrices grow larger and be away from origin. 
    
    \end{remark}

%% file: subfiles/4,5.proof_sketch.tex
\section{Training Dynamics under Balanced Gaussian Initialization}\label{section: balanced init, convergence, informal}

To exhibit the convergence dynamics clearly, we present the global convergence under the simplified scenario of balanced Gaussian initialization (formally defined in Section \ref{subsection: Balanced Gaussian Initialization}) and gradient flow. Notice that the adjacent matrices remain balanced due to the non-increasing property of regularization term (Lemma \ref{regularization, total}). 

\begin{theorem}\label{Total convergence bound, balanced, gf, informal} (Informal) Global convergence bound under balanced Gaussian initialization, gradient flow. 
For four-layer matrix factorization under gradient flow, balanced Gaussian initialization with scaling factor $\epsilon \le \sigma_1^{1/4}(\Sigma) / {\rm poly}( 1 / \delta, d)$, then for target matrix with identical singular values, 

1. For $\mathbb{F} = \mathbb{R}$, with probability at least $\frac{1}{2}$ the loss does not converge to zero. 

2. For $\mathbb{F} = \mathbb{C}$ with high probability at least $1-\delta$ and for $\mathbb{F} = \mathbb{R}$ with probability at least $\frac{1}{2}(1-\delta)$, there exists $T(\epsilon_{\rm conv}) =  \sigma_1^{-3/2}(\Sigma) \cdot {\rm poly}\left(1 / \delta, d, \sigma_1^{1/4}(\Sigma) / \epsilon, \ln\left(d \sigma_1^2(\Sigma) / \epsilon_{\rm conv}\right)\right)$, such that for any $\epsilon_{\rm conv} > 0$, when $t > T(\epsilon_{\rm conv})$, $\mathcal{L}(t) < \epsilon_{\rm conv}$. 

\end{theorem}

The formal version is stated in Theorem \ref{Total convergence bound, balanced, gf} in the Appendix. 

\subsection{Balanced Gaussian Initialization}\label{subsection: Balanced Gaussian Initialization}

Generally, random Gaussian initialization does not satisfy strict balancedness. To adapt the random Gaussian initialization to ensure balanced condition, we introduce a \textit{balanced Gaussian initialization} scheme for the analysis below. The procedure is defined as follows: 

(1) Sample $G$ with entries $G_{ij}\overset{\text{i.i.d.}}{\sim} \mathcal{N}(0,1)_{\mathbb{F}}$, $Q_{k,k+1;k\in[0,N]\cap\mathbb{N}} \overset{\text{i.i.d.}}{\sim}\mathrm{Haar}$ on $U(d,\mathbb{C})$ for $\mathbb{F}=\mathbb{C}$ (or $O(d,\mathbb{R})$ for $\mathbb{F}=\mathbb{R}$). $s_{j,j\in[1,N]\cap\mathbb{N}^*} \in \mathbb{F}$ are arbitrary constants with modulus/absolute value $1$. 

(2) For scaling factor $\epsilon \in\mathbb{R}^+$, which is a small positive constant, set the weight matrices by: 

\begin{equation}\label{init, balanced}
 \begin{aligned}
  W_j = \begin{cases}
   s_{j} \epsilon Q_{j,j+1} G Q_{j-1,j}^H &, 2 \nmid j \\
   s_{j} \epsilon Q_{j,j+1} G^H Q_{j-1,j}^H &, 2 \mid j 
  \end{cases} \quad . 
 \end{aligned}
\end{equation}

Intuitively, $Q_{k,k+1;k\in[0,N]\cap\mathbb{N}}$ are i.i.d. uniformly distributed unitary/orthogonal matrices. By Corollary \ref{Balanced Initialization: each matrix is a Gaussian random matrix ensemble} in the Appendix, each matrix is a $\epsilon$-scaled Gaussian random matrix ensemble (but not independent of the others), while satisfying balanced condition $\Delta_{j,j+1}(0)=O$, $\forall j\in [1,N-1]\cap\mathbb{N}^*$. 

\begin{theorem}\label{Balanced initialization, final} 

Under $\epsilon$-scaled balanced Gaussian initialization with even number of depth $2 \mid N$, suppose $W$ is $W = U \Sigma_w^N V^H$, where $U$, $V$ are unitary/orthogonal matrices, $\Sigma_w$ is positive semi-definite and diagonal, denote $s\coloneqq \prod_{j=1}^{N} s_j$, then for some $f_1 = O\left(\frac{1}{\delta}\right)$, $f_2^\prime = O\left( \frac{1}{\delta^2} \right)$: 

1. If $\mathbb{F} = \mathbb{C}$, with probability $1-\delta$ such that 

\begin{equation}
 \begin{aligned}
  \|\Sigma_w\|_{op} \le f_1(\delta) \sqrt{d} \epsilon &,\,
  \|(U-V)\Sigma_w\|_F|_{t=0} \le 2f_1(\delta) d \epsilon \\
  \sigma_{\min}((U+V)\Sigma_w)|_{t=0} &\ge f_2^\prime(\delta)^{-1}d^{-3/2} \epsilon .
 \end{aligned}
\end{equation}

2. If $\mathbb{F} = \mathbb{R}$, $\Pr( s \det(Q_{N,N+1}) \det(Q_{01})=1) = \Pr( s \det(Q_{N,N+1}) \det(Q_{01})=-1) = \frac{1}{2}$. Under $\Pr( s \det(Q_{N,N+1})\det(Q_{01})=-1)$, $\sigma_{\min}((U+V)\Sigma_w)|_{t=0}$; under $\Pr( s \det(Q_{N,N+1})\det(Q_{01})=1)$, with probability $1-\delta$ such that 

\begin{equation}
 \begin{aligned}
  \|\Sigma_w\|_{op} \le f_1(\delta) \sqrt{d} \epsilon &,\,
  \|(U-V)\Sigma_w\|_F|_{t=0} \le 2f_1(\delta) d \epsilon \\
  \sigma_{\min}((U+V)\Sigma_w)|_{t=0} &\ge f_2^\prime(\delta)^{-1}d^{-3/2} \epsilon .
 \end{aligned}
\end{equation}

\end{theorem}

Proof is presented in Appendix \ref{Balanced initialization, final, proof}. 

\subsection{Non-increasing Skew-Hermitian Error}

As presented in Lemma \ref{ASVD, non-negative diagonal} in the Appendix, the product matrix can be factorized in to the form of $W(t) = U(t) \Sigma_w(t)^N V(t)^H$, where $\Sigma_w(t)$ is positive semi-definite and diagonal (consequently real-valued), $U$ and $V$ are unitary/orthogonal matrices, $U$, $V$ and $\Sigma_w$ are analytic. For simplicity, we denote $\sigma_{w,j}$ as the $j^{th}$ diagonal entry of $\Sigma_{w}$, and $u_j$, $v_j$ as the $j^{th}$ column of $U$, $V$. 
Under this representation of product matrix, we obtain a \textit{non-increasing skew-hermitian/symmetric term}: 

\begin{theorem}\label{antisym-error, general, informal}(Informal) Skew-hermitian error term is non-increasing. 

Under balanced initialization with product matrix $W(t) = U(t) \Sigma_w(t)^N V(t)^H$, for depth $N\ge 2$, if 
the singular values of the product matrix at initial $W(0)$ are non-zero and distinct, then the following skew-hermitian error $\left\|\Sigma^{1/2} (U - V) \Sigma_w \right\|_F^2$ is non-increasing: 

\begin{equation}
  \frac{\mathrm{d}}{\mathrm{d} t} \left\|\Sigma^{1/2} (U - V) \Sigma_w \right\|_F^2 \le 0 .
\end{equation}

\end{theorem}

\textit{Proof sketch. Proof of the Theorem \ref{antisym-error, general, informal} involves technical and lengthy calculations. 
The formal version is stated in Theorem \ref{antisym-error, general balanced}, while a special version for even $N$ is separately discussed in Theorem \ref{antisym-error for 2|N}. 
For the proof of Theorem \ref{antisym-error, general balanced}, the idea is to decompose the derivative of this term into the derivative of $\sigma_{w,j}$ and $u_j$, $v_j$, which have been characterized by Theorem 3 and Lemma 2 in \cite{arora2019implicitregularizationdeepmatrix} respectively. This method is hard to generalize into unbalanced setting. For Theorem \ref{antisym-error for 2|N}, this term is directly derived from derivative of $W_N W_N^H$, $W_1^H W_1$ and $W$. This approach is straight forward and can be extended to unbalanced initialization, but encounters difficulty under odd depth $2\nmid N$. }

\begin{remark}
    This result is under the reduction of target matrix. For general target matrix, suppose its SVD is $\Sigma = U_\Sigma \Sigma^\prime V_\Sigma^H$, then Theorem \ref{antisym-error, general, informal} becomes: 

\begin{equation}
  \frac{\mathrm{d}}{\mathrm{d} t} \left\|{\Sigma^\prime}^{1/2} (U_\Sigma^H U - V_\Sigma^H V) \Sigma_w \right\|_F^2 \le 0.
\end{equation}
\end{remark}

\textbf{Explanation of the result. } This theorem provides an intrinsic non-increasing term (under initialization close to origin, this term is already small at initial) of the system. Though the result is accurately derived under strictly balanced initialization and gradient flow, one may expect similar property to hold under small initialization and gradient descent. 

Moreover, this theorem characterizes when $U$ and $V$ become aligned. The product matrix can be expressed as $W = \sum_{i=1}^{d} \sigma_{w,j}^N u_j v_j^H$, while the error can be rewritten as $\sum_{j=1}^{d} \sigma_{w,j}^2 \left\| \Sigma^{1/2} (u_j - v_j) \right\|_F^2$. Each term $\sigma_{w,j}^N u_j v_j^H$ of the product matrix can be interpreted as \textit{a ``feature" of the linear neural network}, containing one ``value" $\sigma_{w,j}^N$ and two ``directions" $u_j$, $v_j$. When the loss converges, each feature converges to $\sigma_{j} u_{\Sigma,j} u_{\Sigma,j}^H$, where $\Sigma = \sum_{j=1}^{d} \sigma_{j} u_{\Sigma,j} u_{\Sigma,j}^H$ is a SVD of $\Sigma$. This shows that under initialization near origin, once a ``value" of the $j^{th}$ feature \textit{increases to a relatively large value} (comparing to initialization), the directions of this feature \textit{automatically align with each other} (i.e. $\langle u_j, v_j \rangle \approx 1$). Followed by Theoretical illustration part of \cite{arora2019implicitregularizationdeepmatrix}, Section 3, generally the alignment of $U$, $V$ leads to convergence. 

As shown in the proof sketch, the analysis for odd $N$ encounters difficulty when generalized to the unbalanced case, thus this intrinsic non-increasing term becomes considerably more challenging to characterize. This is why we have developed the convergence proof for the four-layer case rather than the three-layer architecture.

\subsection{Non-Decreasing Hermitian Main Term}

This section shows the dynamics of the minimum singular value of hermitian main term $ (U + V) \Sigma_w $. 

The motivation of studying this specific term is that it provides a bound for $\sigma_{k}(\Sigma_{w})$, $k\in[1,N-1]\cap\mathbb{N}^*$, especially a tight bound for $\sigma_{\min}(\Sigma_{w})$ (refer to Lemma \ref{bound of eigenvalues under perturbation}): 

\begin{equation}
 \begin{aligned}
  \frac{1}{2}\sigma_{k} \left( (U + V) \Sigma_w \right) &\le \sigma_{k}(\Sigma_{w}) \le \frac{\sqrt{2}}{2} \sqrt{ \sigma_{k}^{ 2}\left( (U + V) \Sigma_w \right)  + \left\| \left(U-V\right) \Sigma_w \right\|_{op}^2 }\\
  \frac{1}{2}\sigma_{\min}\left( (U + V) \Sigma_w \right) &\le \sigma_{\min}(\Sigma_{w}) \le \frac{1}{2} \sqrt{ \sigma_{\min}^2\left( (U + V) \Sigma_w \right)  + \left\| \left(U-V\right) \Sigma_w \right\|_{op}^2 } . 
 \end{aligned}
\end{equation}

Notice that the extra term in the upper bound is bounded by the skew-hermitian error term discussed in the previous section. 

Although the evolution of $\sigma_{k}((U+V)\Sigma_{w})$ is generally difficult to characterize, we find that in the special case of $\Sigma = \sigma_1(\Sigma) I$ and $N=4$, it exhibits a monotonically increasing pattern before local convergence: 

\begin{theorem}\label{dynamics of minimum singular value of hermitian term} Dynamics of minimum singular value of hermitian term. 

Under balanced initialization with product matrix $W(t) = U(t) \Sigma_w(t)^N V(t)^H$, for target matrix with identical singular values (reduces to $\Sigma = \sigma_1(\Sigma) I$) and depth $N=4$, the time derivative of the $k^{th}$ singular value of the hermitian term $x_k \coloneqq \frac{1}{2}\sigma_k((U+V)\Sigma_w)$ is bounded by: 

\begin{equation}
 \begin{aligned}
  \left(2\sigma_1(\Sigma) - x_k^4 - \frac{1}{2} \|\Sigma_w\|_{op}^2 \|((U-V)\Sigma_w)|_{t=0} \|_{F}^2 \right) x_k^4 - \frac{1}{16} x_k^2 \|\Sigma_w\|_{op}^2 \|((U-V)\Sigma_w)|_{t=0}\|_{F}^4 \\
  \le \frac{\mathrm{d}}{\mathrm{d} t} x_k^2 
  \le \sigma_1(\Sigma) \left(2 \|\Sigma_w\|_{op}^{2} + \|((U-V)\Sigma_w)|_{t=0} \|_{F}^2 \right) x_k^2 .
 \end{aligned}
\end{equation}

\end{theorem}

Detailed proof is presented in \ref{dynamics of minimum singular value of hermitian term, proof}.

This theorem implies that under small initialization, if all singular values $\sigma_k((U+V)\Sigma_w)$ are initially non-zero, they increase monotonically to relatively large values, leading to subsequent local convergence. However, if any singular value is initialized to zero (which occurs with probability at least $1/2$ for $\mathbb{F} = \mathbb{R}$, as shown in Theorem \ref{Balanced initialization, final}), it \textit{remains zero throughout the optimization} (see Corollary \ref{never converge to optimum, balanced}), thereby explaining the $1/2$ convergence probability in Theorem \ref{Total convergence bound, balanced, gf, informal}. Numerical simulations under the identity target setting are provided in Figure \ref{fig: balanced init, log singular values, identity target}, with additional results and discussions for non-identity targets shown in Figure \ref{fig: balanced init, log singular values, non-identity target}.

\section{Convergence under Random Gaussian Initialization}

This section presents the proof sketch for Theorem \ref{Total convergence bound, gd, informal}, extending our analytical framework in the previous section to accommodate random Gaussian initialization. 

For random Gaussian Initialization with balance regularization term, the balanced condition holds approximately. Following the methodology in balanced initialization scheme, Section \ref{section: balanced init, convergence, informal}, we then characterize the skew-hermitian error term and hermitian main term by $\|W_1 - W_2^{-1} W_3^H W_4^H\|_F^2$ and $\lambda_{\min}\left(\left(W_1 + W_2^{-1} W_3^H W_4^H\right)^H \left(W_1 + W_2^{-1} W_3^H W_4^H\right)\right)$ respectively. 


\subsection{random Gaussian Initialization}

We consider the canonical setting of random Gaussian initialization near origin: 

\begin{equation}\label{init, random gaussian}
 \begin{aligned}
  (W_{1,2,\cdots, N})_{i j} \overset{\text{i.i.d.}}{\sim} \epsilon \cdot \mathcal{N}(0,1)_\mathbb{F} . 
 \end{aligned}
\end{equation}

Specifically, we apply Gaussian distribution to generate $W_{1,2,\cdots, N} \in \mathbb{F}^{d\times d}$, $F=\mathbb{R}$ or $\mathbb{C}$ element-wisely and independently. Then the initialization is scaled by a small positive constant $\epsilon\in\mathbb{R}^+$. The scale of $\epsilon$ is determined in the main convergence Theorem \ref{Total convergence bound, gd, informal}. 

\begin{theorem}\label{Gaussian random matrix ensemble product, eigenvalues}

For $\epsilon$-scaled random Gaussian initialization on $W_{k,k=[1,N]\cap\mathbb{N}^*}$ over $\mathbb{F} = \mathbb{R}$ or $\mathbb{C}$, $N\in\mathbb{N}^*$, the initial product matrix $W = \prod_{k=N}^{1} W_k$ satisfy the following properties: 

1. If $\mathbb{F} = \mathbb{C}$, with probability at least $1-\delta$, 

\begin{equation}\label{eq, Gaussian random matrix ensemble product, eigenvalues, complex}
 \begin{aligned}
  \max_{j,k} \sigma_k(W_j) \le f_1(\delta,N) \sqrt{d} \epsilon &,\,
  \min_{j,k} \sigma_k(W_j) \le \frac{\epsilon}{f_1(\delta,N) \sqrt{d}} \\ 
  \sigma_{\min}\left(W + \left(WW^H\right)^{1/2}\right) &\ge f_2(\delta,N)^{-1} \cdot d^{-(N/2+1)} \epsilon^N . 
 \end{aligned}
\end{equation}

2. If $\mathbb{F} = \mathbb{R}$, the determinants $\det(W) > 0$ and $\det(W) < 0$ occur each with probability $1/2$. If $\det(W) < 0$, then $\sigma_{\min}\left(W + \left(WW^\top\right)^{1/2}\right) = 0$; if $\det(W) > 0$, then with conditional probability at least $1-\delta$, 

\begin{equation}\label{eq, Gaussian random matrix ensemble product, eigenvalues, real}
 \begin{aligned}
  \max_{j,k} \sigma_k(W_j) \le f_1(\delta,N) \sqrt{d} \epsilon &,\, 
  \min_{j,k} \sigma_k(W_j) \le \frac{\epsilon}{f_1(\delta,N) \sqrt{d}} \\ 
  \sigma_{\min}\left(W + \left(WW^\top\right)^{1/2} \right) &\ge f_2(\delta,N)^{-1} \cdot d^{-(N/2+1)} \epsilon^N ,
 \end{aligned}
\end{equation}

where $f_1(\delta,N) = O\left( \frac{N}{\delta} \right)$, $f_2(\delta,N) = O\left( \frac{N^N}{\delta^{N+1}} \right)$. 

\end{theorem}

Proof is provided in Appendix \ref{Gaussian random matrix ensemble product, eigenvalues, proof}. For $N=4$, $f_1 = O\left( \frac{1}{\delta} \right)$, $f_2 = O\left( \frac{1}{\delta^{5}} \right)$. 

In the convergence proof below, we consider the initialization where (\ref{eq, Gaussian random matrix ensemble product, eigenvalues, complex}) and (\ref{eq, Gaussian random matrix ensemble product, eigenvalues, real}) holds. We divide the training dynamics into three stages consisting of an alignment stage $t\in [0,T_1]$, a saddle-avoidance stage $t\in [T_1, T_1 + T_2]$,  and a local convergence stage $t\in [T_2, +\infty)$, to analyze the convergence process clearly. Here $T_1 = O\left(\frac{1}{\eta \sigma_1^{3/2}(\Sigma) \cdot {\rm poly}\left(f_1, f_2,d,\epsilon /\sigma_1^{1/4}(\Sigma)\right)}\right)$, $T_2 = O\left(\frac{1}{\eta \sigma_1^{3/2}(\Sigma)} \cdot {\rm poly}\left( f_1, f_2, d, \sigma_1^{1/4} / \epsilon \right) \right)$, refer to Theorem \ref{stage 1: alignment stage, GD} and \ref{stage 2: saddle avoidance stage, gd} respectively. 

\subsection{Stage 1: Alignment Stage}

During this stage, the weight matrices align with each other under the convergence of the regularization term, while the hermitian main term stays away from origin at the end of this stage. 

\subsubsection{Convergence of Regularization term: }\label{subsubsection: Convergence of Regularization term}

The convergence rate of the regularization term is related to the smallest singular value of weight matrices: 

\begin{theorem} (Informal) Convergence rate of the regularization term. 

For four-layer matrix factorization, suppose the maximum and minimum singular values of the weight matrices are bounded by $M$ and $\delta$ respectively, then the regularization term decays by

\begin{equation}
  \mathcal{L}_{\rm reg}(t+1) \le \left(1 - \Omega\left( \eta a \delta^4 M^{-2} \right) \right) \cdot \mathcal{L}_{\rm reg}(t) + O(\eta^2 a^2) .
\end{equation}

\end{theorem}

The formal version can be found in Theorem \ref{regularization term, convergence bound, GD}. A $N$-layer version of this Theorem, along with a generalized loss function under gradient flow is provided in Theorem \ref{regularization term, convergence bound}. This shows the importance of bounding the extreme singular values of $W_j$, otherwise the linear convergence of the regularization term (along with the balancedness) might not be guaranteed. 

\begin{theorem} (Informal) Under a small learning rate, the change in the maximum and minimum singular values is approximately independent of the regularization term: 

\begin{equation}
 \begin{aligned}
  \max_{j,k} \sigma_k^2(W_j(t+1)) - \max_{j,k} \sigma_k^2(W_j(t)) &\le 2\eta \max_{j,k} \sigma_k(W_j(t)) \max_{j}\left\| \nabla_{W_j} \mathcal{L}_{\rm ori}(t) \right\|_{op} + O(\eta^2 a^2) \\ 
  \min_{j,k} \sigma_k^2(W_j(t+1)) - \min_{j,k} \sigma_k^2(W_j(t)) &\ge - 2\eta \min_{j,k} \sigma_k(W_j(t)) \max_{j}\left\| \nabla_{W_j} \mathcal{L}_{\rm ori}(t) \right\|_{op} + O(\eta^2 a^2) .
 \end{aligned}
\end{equation}

Here $a$ is the coefficient of the regularization term. 

\end{theorem}

This Theorem ensures the smooth change of the extreme singular values over short time intervals. Although the regularization term can induce significant fluctuations in individual singular values due to its potentially large coefficient, the largest and smallest singular values remain stable. This theoretical conclusion is corroborated by numerical simulations, as shown in Figure \ref{fig: random init, log singular values, max-min}. The complete formal statement can be found in Theorem \ref{maximum and minimum singular values are irrelevant of the regularization term, GD} (and Theorem \ref{maximum and minimum singular values are irrelevant of the regularization term} for the continuous-time case) in the Appendix.

\subsubsection{The Limit Behavior of the Hermitian main term}\label{subsubsection: The Limit Behavior of the Hermitian main term}

Typically, the dynamics of the smallest singular value of the hermitian main term $W_1 + W_2^{-1}W_3^H W_4^H$ is involved and does not obtain a non-trivial lower bound during this stage. However its limit behavior after the convergence of regularization term can be characterized. 

To simplify the analysis, ignore the original square loss $\mathcal{L}_{\rm ori}$ and consider gradient flow. For $t\to+\infty$, regularization term is exactly zero and thus the adjacent matrices are strictly balanced. Moreover, the product matrix does not change through the optimization: $W(+\infty) = W(0)$. Then under this scenario, the limit behavior of the hermitian main term is $\left.\left(W_1 + W_2^{-1}W_3^H W_4^H\right)\right|_{t\to+\infty} = \left. \left(W_2^{-1} W_3^{-1} W_4^{-1} \right)\right|_{t\to+\infty} \left(W(0) + \left( W(0) W(0)^H \right)^{1/2} \right)$. 

This explains the reason of studying $\sigma_{\min}\left(W(0) + \left( W(0) W(0)^H \right)^{1/2} \right)$ in the initialization section. Detailed analysis considering error terms is presented in Corollary \ref{Main term at the end of alignment stage, gd}. 

\begin{remark}

Note that $\sigma_{\min}\left(W_1 + W_2^{-1} W_3^H W_4^H \right)$ is not necessarily lower-bounded by the above expression minus some error terms during the alignment stage. Instead, it may exhibit oscillations or a transient decrease, achieving stability only upon convergence of the regularization term. This behavior is illustrated in Figure \ref{fig: random init, log singular values, main term} in the Appendix. 

\end{remark}

\subsection{Stage 2: Saddle Avoidance stage}

Intuitively, this section focuses on generalizing Theorem \ref{antisym-error, general, informal} and \ref{dynamics of minimum singular value of hermitian term} into unbalanced case by bounding the error terms introduced by unbalanceness. 

The main technical challenge is to bound the operator norm of the inverse of $W_2$ below infinity, since both the skew-hermitian term and hermitian main term are characterized by $W_2^{-1}$ and hence need to be well-defined. Under small balance error (equivalently small regularization term) which is guaranteed by the previous stage, $W_2^{-1}$, which is rigorously proved in Lemma \ref{bound of w23 inv op, and relevant terms, GD}. 

\begin{lemma}\label{stage 2, skew-hermitian error, GD}

Skew-hermitian error in saddle avoidance stage, gradient descent. For $t \in [T_1, T_1 + T_2]$, 

\begin{equation}
  \left \| W_1 - W_2^{-1} W_3^H W_4^H \right \|_F \le 3 f_1 d \epsilon . 
\end{equation}

\end{lemma}

\begin{lemma}\label{stage 2, main term, GD}

The minimum eigenvalue of Hermitian term. For $t = T_1 + T_2$, 

\begin{equation}
  \sigma_{\min}\left(W_1 + W_2^{-1} W_3^H W_4^H \right)|_{t = T_1 + T_2} \ge 2^{3/4} \sigma_1^{1/4}(\Sigma) . 
\end{equation}

\end{lemma}

Proofs are presented in \ref{stage 2, skew-hermitian error, GD, proof} in the Appendix. 

\subsection{Stage 3: Local Convergence Stage}

Since both the balanced error and skew-Hermitian error remain small, the minimal singular values of the weight matrices, after growing to the scale of the target matrix's, are prevented from decaying. This guarantees the local convergence.

\begin{theorem}\label{local convergence, GD, informal}

(Informal) Local convergence. After the second stage ($t\ge T_1 + T_2$), 

\begin{equation}
 \begin{aligned}
  \mathcal{L}(t) \le \mathcal{L}_{\rm ori}(T_1 + T_2) \exp\left(- \eta \sigma_1^{3/2}(\Sigma) (t - T_1 - T_2) \right) \\
  \sigma_{\min}\left(W_1(t) +W_2(t)^{-1} W_3(t)^H W_4(t)^H \right) \ge 2^{3/4} \sigma_1^{1/4}(\Sigma) \\ 
  \left \| W_1(t) - W_2(t)^{-1} W_3(t)^H W_4(t)^H \right \|_F \le 3 f_1 d \epsilon . 
 \end{aligned}
\end{equation}

\end{theorem}

Proof is presented in \ref{local convergence, GD, proof} in the Appendix.

%% file: subfiles/6.conclusions.tex
\section{Conclusions, Limitations and Future work}

In this work, we establish a polynomial-time global convergence guarantee for gradient descent applied to four-layer matrix decomposition, under the setting of a target matrix with identical singular values and small random Gaussian initialization beyond the NTK regime. For complex random Gaussian initialization, global convergence is ensured with high probability, whereas for real random Gaussian initialization, it is guaranteed with a probability close to $\frac{1}{2}$. 

The analysis developed in this work reveals intrinsic properties of the training dynamics, such as the effective behavior of the regularization term, the monotonically increasing lower bound for the minimum singular value, and the non-increasing nature of the skew-Hermitian error. These findings might provide deeper insight into the training process of Deep Linear Networks.

We anticipate that this work will stimulate further research on global convergence proofs under general random initialization for matrix factorization with arbitrary depth and arbitrary - possibly low-rank - target matrices. 

The observed divergence in convergence behavior between real and complex initializations also reveals a subtle disparity, suggesting that complex initializations may circumvent certain saddle points that real initializations cannot. This insight might motivate more detailed analysis of the performance gap between complex and real neural networks. 

%% file: subfiles/7.Reproducibility_Statement.tex
\section*{Reproducibility Statement}

All theoretical results stated in this paper are proved in full detail in the Appendix , from Section \ref{Appendix: Reduction To Diagonal (Identical) Target} to \ref{section: convergence, staged analysis, gd}, including the proofs of all main-text theorems as well as intermediate lemmas and derivations, so that a reader can verify each step independently. The numerical illustration in Appendix \ref{section: Numerical Simulations}, where we specify the hyper-parameters in that section. Because the experiments are straightforward, we have not released an implementation. 

%% file: subfiles/appendix/appendix_1_reduction_to_diagonal.tex
\section{Reduction To Diagonal (Identical) Target}\label{Appendix: Reduction To Diagonal (Identical) Target}

For arbitrary ground truth $\Sigma \in \mathbb{F}^{d\times d}$, $\mathbb{F} = \mathbb{C}$ or $\mathbb{R}$, suppose its singular value decomposition is $\Sigma = U_\Sigma \Sigma^\prime V_\Sigma^H$ (replace $\cdot^H$ by $\cdot^\top$ for the real case, same for the rest of the analysis), we apply the following transformation: 

\begin{equation}\label{reduction to diagonal, eq}
 \begin{cases}
  W_1^\prime &=  W_1 V_\Sigma \\
  W_{j}^\prime &= W_j,\, j \in [2,N-1]\cap \mathbb{N}^* \\
  W_N^\prime &= U_\Sigma^H W_N 
 \end{cases} \quad .
\end{equation}

Then the balance error can be rewritten as 

\begin{equation}
  \Delta_{j,j+1} = \begin{cases}
    W_j^\prime{W_j^\prime}^H - {W_{j+1}^\prime}^H W_{j+1}^\prime &,\, j\in[1,N-1]\cap \mathbb{N}^*\\
    O^{d\times d} &,\, j\in\{0,N\} 
  \end{cases} \quad .
\end{equation}

\subsection{Training Dynamics}

For gradient flow, the dynamics becomes 

\begin{equation}
 \begin{aligned}
  \frac{\mathrm{d} W_j^\prime}{\mathrm{d}t} = \left( \prod_{k=j+1}^{N} {W_k^\prime}^H \right) \left(\Sigma^\prime - \prod_{k=N}^{1}W_k^\prime \right) \left( \prod_{k=1}^{j-1} {W_k^\prime}^H \right) + a W_j^\prime \Delta_{j-1,j} - a \Delta_{j,j+1} W_j^\prime .
 \end{aligned}
\end{equation}

For gradient descent, 

\begin{equation}
 \begin{aligned}
  W_j^\prime(t+1) &= W_j^\prime(t) + \eta \left( \prod_{k=j+1}^{N} {W_k^\prime(t)}^H \right) \left(\Sigma^\prime - \prod_{k=N}^{1}W_k^\prime(t) \right) \left( \prod_{k=1}^{j-1} {W_k^\prime(t)}^H \right) \\ &+ \eta a W_j^\prime(t) \Delta_{j-1,j}(t) - \eta a \Delta_{j,j+1}(t) W_j^\prime(t) .
 \end{aligned}
\end{equation}

Both share the same form as the original one (by replacing $\Sigma$ with $\Sigma^\prime$). 

\subsection{Initialization}

However, the distributions of $W_1$ and $W_N$ at initialization change correspondingly. To address this issue, we introduce the following definition: 

\begin{definition}\label{Input-Output Unitary(Orthogonal)-Invariant initialization} Input-Output Unitary(Orthogonal)-Invariant initialization. 

For a $N$-layer complex (real) matrix factorization $W = \prod_{j=N}^{1} W_j$, an initialization is input-output unitary-invariant (in the complex case) or orthogonal-invariant (in the real case) if the distribution of $W_N$ is left unitarily (or orthogonally) invariant and the distribution of $W_1$ is right unitarily (or orthogonally) invariant. That is, for all $U, V \in U(d, \mathbb{C})$ (or $O(d, \mathbb{R})$ in the real case), 

\begin{equation}
  W_N \overset{d}{=} U W_N ,\,W_1 \overset{d}{=} W_1 V .
\end{equation}

\end{definition}

\begin{remark}

The distribution of $W_{j,j\in[1,N]\cap\mathbb{N}^*}$ does not change under transformation \ref{reduction to diagonal, eq} if the initialization is Input-Output Unitary(Orthogonal)-Invariant. 

\end{remark}

Throughout this work, the initialization schemes discussed (including random Gaussian initialization and balanced Gaussian initialization) are Input-Output Unitary(Orthogonal)-Invariant. This is from the left and right invariance under multiplication of unitary/orthogonal matrices. 

Thus without loss of generality, the target matrix can be reduced to positive semi-definite diagonal matrix. Under Input-Output Unitary(Orthogonal)-Invariant initialization discussed in Definition \ref{Input-Output Unitary(Orthogonal)-Invariant initialization}, the initialization on $W_1$ and $W_N$ is not affected by this reduction. 

Moreover, if all singular values of $\Sigma$ are the same (to rephrase, a unitary/orthogonal matrix scaled by a constant), the convergence analysis can be reduced to $\Sigma^\prime = \sigma_1(\Sigma) I$.

%% file: subfiles/appendix/appendix_2_initialization.tex
\section{Initialization}\label{section: initialization}

First and foremost, we introduce the concept of Circular ensembles \citep{dyson1962threefold} along with some properties. 

\subsection{Lemmas for Gaussian random matrix ensemble and Haar measure on $U(d,\mathbb{C})$ and $O(d,\mathbb{R})$}

In the following derivations, we denote $O(d,\mathbb{R})$ as the $d$-dimensional orthogonal group on real number, and $U(d,\mathbb{C})$ as the $d$-dimensional unitary group on complex number. 

We list the classical conclusions in Linear Algebra without proof: 

\begin{lemma} The eigenvalues of Orthogonal/Unitary Matrices. 

1. Unitary matrices. $\forall U \in U(d,\mathbb{C})$, $d\in\mathbb{N}^*$, the eigenvalues of $U$ are $e^{i\theta_{1,2,\cdots,d}}$, where $\theta_i \in[0,2\pi)$. 

2. Orthogonal matrices. $\forall O \in O(d,\mathbb{R})$, $d\in\mathbb{N}^*$, the eigenvalues of $O$ are: 

\begin{equation}
 \begin{cases}
  1, e^{\pm i\theta_{1,2,\cdots,m}}&,\, d = 2m+1,\, \det(O) = 1 \\ 
  -1, e^{\pm i\theta_{1,2,\cdots,m}}&,\, d = 2m+1,\, \det(O) = -1 \\ 
  e^{\pm i\theta_{1,2,\cdots,m}}&,\, d = 2m,\, \det(O) = 1 \\ 
  1,-1, e^{\pm i\theta_{1,2,\cdots,m-1}}&,\, d = 2m,\, \det(O) = -1 \\ 
 \end{cases} \quad. 
\end{equation}

\end{lemma}

Following the conventions, we call the argument of the eigenvalues as eigenangles. 

\begin{definition} Circular ensembles. (refer to \cite{dyson1962threefold}, \cite{forrester2010log})

The circular ensembles are measures on spaces of unitary(or orthogonal, when generalizing from complex number to real number) matrices. 

1. Unitary circular ensemble. The distribution of the unitary circular ensemble (CUE) is the Haar measure on $d$-dimensional (complex) unitary group $U(d,\mathbb{C})$. 

2. Circular real ensemble. The distribution of the circular real ensemble (CRE) is the Haar measure on $d$-dimensional real orthogonal group $O(d,\mathbb{R})$. 

\end{definition}

\begin{lemma}\label{1-point correlation function of circular ensemble} 1-point correlation function of $\mathrm{CUE}(d)$ and $\mathrm{CRE}(d)$. 

1. CUE. The 1-point correlation function of $\mathrm{CUE}(d)$ is 

\begin{equation}
  \rho_{(1), \mathrm{CUE}}(\theta) = \frac{d}{2\pi} .
\end{equation}

2. CRE, determinant $1$. The 1-point correlation function of $\mathrm{CRE}(d)$ under determinant $1$ is 

\begin{equation}
  \rho_{(1), \mathrm{CRE}, \det=1}(\theta) = \frac{1}{2\pi} \left( d-1 + (-1)^d \frac{\sin (d-1)|\theta|}{\sin |\theta|} \right) ,\,\theta\in(-\pi,\pi] .
\end{equation}

\end{lemma}

\begin{remark}

1-point correlation function $\rho_{(1)}(\theta)$ can be interpreted as the density of eigenangles at $\theta$ (despite probably existed fixed eigenangles, e.g. $0$, $\pi$). 

\end{remark}

\begin{proof}

Part 1. CUE. 

From (146) of \cite{dyson1962threefold} and \cite{forrester2010log}, the joint probability density function of eigenangles is 

\begin{equation}
  p_{\mathrm{CUE}}(\theta_{k,k\in[1,d]\cap\mathbb{N}^*}) \propto \prod_{1\le k < j \le d} \left| e^{i\theta_j} - e^{i\theta_k} \right|^2 = \prod_{1\le k < j \le d} \left| e^{i(\theta_j - \theta_k)} - 1 \right|^2 .
\end{equation}

Notice that it is rotation invariant, that is $\forall \Delta \theta \in [0,2\pi]$, $p_{\mathrm{CUE}}(\theta_{k,k\in[1,d]\cap\mathbb{N}^*}) = p_{\mathrm{CUE}}((\theta_{k} + \Delta \theta)_{k\in[1,d]\cap\mathbb{N}^*})$. Thus the 1-point correlation function (density of eigenangles at $\theta$) is uniform, which is $\frac{d}{2\pi}$. 

Part 2. CRE. 

Below we define $x_{i} = \cos\theta_i$, then $\rho_{(1)}(\theta) = \sin\theta \cdot \rho_{(1)}(x)$, $p(x_{k,k\in[1,N]\cap\mathbb{N}^*}) = \left(\prod_{k=1}^{N} \frac{1}{\sqrt{1-x^2}} \right) p(\theta_{k,k\in[1,N]\cap\mathbb{N}^*})$. 

By combining Proposition 5.1.1 and 5.1.2 in \cite{forrester2010log} together, suppose with $p_k(x)$ a polynomial of degree $k$ which is further more monic (i.e. the coefficient of $x^k$ is unity), $\{p_k(x)\}_{k\in \mathbb{N}}$ is the orthogonal polynomials associated with the weight function $w_2(x)$, 

\begin{equation}
  \int_{-\infty}^{+\infty} p_j(x) p_k(x) w_2(x) \mathrm{d}x \eqqcolon \langle p_j, p_k \rangle_2 = \langle p_j, p_j \rangle_2 \delta_{j,k} .
\end{equation}

and the joint probability density function satisfies 

\begin{equation}
  p(x_{k,k\in[1,N]\cap\mathbb{N}^*}) \propto \prod_{1\le k < j \le N} \left( x_j - x_k \right)^2 \prod_{l=1}^{N} w_2(x) .
\end{equation}

the 1-point correlation function is

\begin{equation}\label{1-point correlation: in common}
  \rho_{(1)}(x) = w_2(x) \sum_{\nu = 0}^{N-1} \frac{p_\nu^2(x)}{\langle p_\nu, p_\nu \rangle_2} .
\end{equation}

Note that the restriction of monic can be ommited since there is a normalization coefficient on the denominator. 

2.1. CRE, determinant 1, $d=2N$. From (135) of \cite{dyson1962threefold}, Section 2.9 of \cite{forrester2010log} and \cite{girko1985distributionOrthogonal},  

\begin{equation}
  p_{\mathrm{CRE}, {\rm even}, \det=1}(\theta_{k,k\in[1,N]\cap\mathbb{N}^*}) \propto \prod_{1\le k < j \le N} \left| \cos \theta_j - \cos \theta_k \right|^2 ,\, \theta_{k,k\in[1,N]\cap\mathbb{N}^*} \in [0,\pi] .
\end{equation}

By the change of variables, 

\begin{equation}
  p_{\mathrm{CRE}, {\rm even}, \det=1}(x_{k,k\in[1,N]\cap\mathbb{N}^*}) \propto \prod_{1\le k < j \le N} \left( x_j - x_k \right)^2 \prod_{l=1}^{N} \frac{1}{\sqrt{1-x_l^2}}  . 
\end{equation}

Here $w_{2}(x) = \frac{1}{\sqrt{1-x^2}}$. From knowledge of orthogonal polynomials ((1.12.3), (4.1.7), \cite{szegő1939orthogonal}), Chebyshev polynomials of the first kind $T_n(x) = \cos(n \arccos x)$ associates with $w_2(x) = \frac{1}{\sqrt{1-x^2}}$: 

\begin{equation}
  \int_{-1}^{1} T_j(x) T_k(x) w_2(x) \mathrm{d}x = \begin{cases}
    \pi, & j=k=0 \\
    \frac{\pi}{2}, & j=k\ge1 \\
    0, & j\ne k
  \end{cases} \quad.
\end{equation}

By (\ref{1-point correlation: in common}), 

\begin{equation}
 \begin{aligned}
  \rho_{(1), \mathrm{CRE}, {\rm even}, \det=1}(x) &= \frac{1}{\sqrt{1-x^2}} \cdot \left( \frac{1}{\pi} + \frac{2}{\pi} \sum_{\nu=1}^{N-1} \cos^2 \nu\theta \right) \\&= \frac{1}{2\pi \sin\theta} \left[2N-1 + \frac{\sin(2N-1)\theta}{\sin \theta}\right] .
 \end{aligned}
\end{equation}

\begin{equation}
  \rho_{(1), \mathrm{CRE}, {\rm even}, \det=1}(\theta) = \frac{1}{2\pi} \left[d-1 + \frac{\sin(d-1)\theta}{\sin \theta}\right],\,\theta \in [0,\pi] .
\end{equation}

From symmetry, $\rho_{(1), \mathrm{CRE}, {\rm even}, \det=1}(-\theta) = \rho_{(1), \mathrm{CRE}, {\rm even}, \det=1}(\theta)$. 

2.2. CRE, determinant 1, $d=2N+1$. From (137) of \cite{dyson1962threefold}, Section 2.9 of \cite{forrester2010log} and \cite{girko1985distributionOrthogonal}, 

\begin{equation}
  p_{\mathrm{CRE}, {\rm odd}, \det=1}(\theta_{k,k\in[1,N]\cap\mathbb{N}^*}) \propto \prod_{1\le k < j \le N} \left| \cos \theta_j - \cos \theta_k \right|^2 \prod_{l=1}^{N} (1-\cos \theta_l) ,\, \theta_{k,k\in[1,N]\cap\mathbb{N}^*} \in [0,\pi] .
\end{equation}

By the change of variables, 

\begin{equation}
  p_{\mathrm{CRE}, {\rm odd}, \det=1}(x_{k,k\in[1,N]\cap\mathbb{N}^*}) \propto \prod_{1\le k < j \le N} \left( x_j - x_k \right)^2 \prod_{l=1}^{N} \sqrt{\frac{1-x_l}{1+x_l}} .
\end{equation}

Here $w_{2}(x) = \sqrt{\frac{1-x}{1+x}}$. From knowledge of orthogonal polynomials ((1.12.3), (4.1.7), \cite{szegő1939orthogonal}), Chebyshev polynomials of the fourth kind $W_n(x) = \frac{\sin \left(\left(n+\frac{1}{2}\right)\theta\right)}{\sin\left(\frac{\theta}{2}\right)}$, $\theta = \arccos x$ associates with $w_{2}(x) = \sqrt{\frac{1-x}{1+x}}$: 

\begin{equation}
  \int_{-1}^{1} W_j(x) W_k(x) w_2(x) \mathrm{d}x = \begin{cases}
    \pi, & j=k \ge 0 \\
    0, & j\ne k
  \end{cases} \quad .
\end{equation}

By (\ref{1-point correlation: in common}), 

\begin{equation}
 \begin{aligned}
  \rho_{(1), \mathrm{CRE}, {\rm odd}, \det=1}(x) &= \sqrt{\frac{1-x}{1+x}} \cdot \left( \frac{1}{\pi} \sum_{\nu=0}^{N-1} \left(\frac{\sin \left(\left(n+\frac{1}{2}\right)\theta\right)}{\sin\left(\frac{\theta}{2}\right)}\right)^2 \right) \\&= \frac{1}{2\pi \sin\left(\theta \right)} \left[2N - \frac{\sin(2N\theta)}{\sin \theta}\right] .
 \end{aligned}
\end{equation}

\begin{equation}
  \rho_{(1), \mathrm{CRE}, {\rm odd}, \det=1}(\theta) = \frac{1}{2\pi} \left[d-1 - \frac{\sin(d-1)\theta}{\sin \theta}\right],\,\theta \in [0,\pi] .
\end{equation}

From symmetry, $\rho_{(1), \mathrm{CRE}, {\rm odd}, \det=1}(-\theta) = \rho_{(1), \mathrm{CRE}, {\rm odd}, \det=1}(\theta)$. 

This completes the proof.

\end{proof}

\begin{theorem}\label{minimum singular value under Haar measure}

For $Q$ sampled from Haar measure on $U(d,\mathbb{C})$ (or $O(d,\mathbb{R})$ if $\mathbb{F} = \mathbb{R}$), 

1. $\mathbb{F} = \mathbb{C}$. $\Pr(\sigma_{\min}(I+Q) \ge \pi\delta d^{-1}) \ge 1-\delta$.

2. $\mathbb{F} = \mathbb{R}$. If $d\ge2$, $\Pr\left(\sigma_{\min}(I+Q) \ge \frac{\pi\delta}{2} (d-1)^{-1}|\det(Q)=1 \right) \ge 1-\delta$. 

\end{theorem}

\begin{remark}

For $\mathbb{F} = \mathbb{R}$, $d=1$, the eigenvalue of $Q$ is $\det(Q)$, and thus $\Pr(\sigma_{\min}(I+Q) \ge 2 - \Delta |\det(Q)=1) = 1$, $\forall \Delta\in(0,2)$. 

\end{remark}

\begin{remark}

For $\mathbb{F} = \mathbb{R}$, $\Pr(\det(Q)=1) = \Pr(\det(Q)=-1) = \frac{1}{2}$. If $\det(Q)=-1$, $Q$ has an eigenvalue of $-1$, causing $\Pr(\sigma_{\min}(I+Q)) = 0$. 

\end{remark}

\begin{proof}

Consider $\theta_k\in(-\pi,\pi]$, 

\begin{equation}
 \begin{aligned}
  \sigma_{k}(I + Q) &= \sqrt{\lambda_{k}(2I + Q + Q^H)} = \sqrt{2 + e^{i\theta_k} + 1/e^{i\theta_k}} = 2\cos \left(\frac{\theta_k}{2}\right) \\ 
  \sigma_{\min}(I + Q) &= \min_k\cos\left(\frac{\theta_k}{2}\right) .
 \end{aligned}
\end{equation}

The second step is from the fact that $Q^H = Q^{-1}$ shares the same eigenvectors with $Q$, and corresponding eigenvalues are the reciprocal of the original eigenvalues. 

Denote $N(\delta\theta)$ to be number of eigenvectors in $(-\pi, -\pi + \delta\theta]\cup[\pi - \delta\theta, \pi]$, $\delta\theta\in(0,\pi)$. From Markov inequality, 

\begin{equation}
 \begin{aligned}
  \Pr\left(\sigma_{\min}(I+Q) \ge \delta\theta\right) &\ge \Pr\left(\sigma_{\min}(I+Q) \ge 2\sin \frac{\delta\theta}{2}\right) \\
  &= 1-\Pr(N(\delta\theta)\ge1) \\
  &\ge 1 - \mathbb{E}(N(\delta\theta)) = 1 - \int_{\theta \in (-\pi, -\pi + \delta\theta]\cup[\pi - \delta\theta, \pi]} \rho_{(1)}(\theta) \mathrm{d}\theta .
 \end{aligned}
\end{equation}

By invoking Lemma \ref{1-point correlation function of circular ensemble}, 

1. For $\mathbb{F} = \mathbb{C}$, 

\begin{equation}
  \mathbb{E}(N(\delta\theta)) = \frac{d}{2\pi} \cdot 2\delta\theta . 
\end{equation}

By setting $\delta\theta = \pi\delta d^{-1}$, $\Pr\left(\sigma_{\min}(I+Q) \ge \delta\theta\right) \ge 1 - \delta$. 

2. For $\mathbb{F} = \mathbb{R}$ under determinant 1, for $\theta^\prime\in[0,\pi]$, $\rho_{(1)} (\pi - \theta^\prime) = \frac{1}{2\pi} \left( d-1 + \frac{\sin (d-1)\theta^\prime}{\sin \theta^\prime} \right)$. 

If $d=1$, $\rho_{(1)} (\theta)\equiv 0$ and thus $\mathbb{E}(N(\delta\theta)) = 0$. For $d\ge2$:  

From $\frac{\sin (d-1)\theta}{\sin \theta} \le d-1$, 

\begin{equation}
 \begin{aligned}
  \mathbb{E}(N(\delta\theta)) = 2\int_{0}^{\delta\theta} \rho_{(1)} (\pi - \theta^\prime) \mathrm{d}\theta^\prime \le 2\int_{0}^{\delta\theta} \frac{1}{2\pi} \cdot 2(d-1) \mathrm{d}\theta^\prime = \frac{2(d-1)}{\pi}\delta\theta .
 \end{aligned}
\end{equation}

By setting $\delta\theta = \frac{\pi\delta}{2} (d-1)^{-1}$, $\Pr\left(\sigma_{\min}(I+Q) \ge \delta\theta|\det(Q)=1\right) \ge 1 - \delta$. 

This completes the proof. 

\end{proof}

\subsection{Random Gaussian Initialization}\label{subsection: random gaussian initialization}


In the following, we present the proof for Theorem \ref{Gaussian random matrix ensemble product, eigenvalues}. 

For a real/complex Gaussian random matrix of dimension $d\times d$, with probability at least $\delta$, the largest singular value is upper bounded by $ O\left( \left(1 + \sqrt{\frac{\ln\left( \frac{1}{\delta} \right)}{d}} \right) \sqrt{d} \right)$ (Theorem 4.4.5, \cite{Vershynin_2018}), while the smallest is lower bounded by $\Omega\left(\frac{\delta}{\sqrt{d}}\right)$ (Theorem 1.1, \cite{tao2009randommatricesdistributionsmallest}). (also refer to Corollary 2.3.5 and Theorem 2.7.5 of \cite{taotopics} )

\begin{proof}\label{Gaussian random matrix ensemble product, eigenvalues, proof}

The upper and lower bound for singular values of $W_k$ follows immediately. 
The main challenge is the minimum singular value of $W + (W W^H)^{1/2}$. 

At the beginning, we define a modification of Gaussian random matrix ensemble for simplification: 

$W$ is sampled from (complex or real) Gaussian random matrix ensemble, and if $\mathrm{rank}(W)$ is not full, sample $W$ from Gaussian random matrix ensemble again until it is full rank. 

Since the set of $\mathrm{rank}(W)$ not being full is zero measure, the distribution of $W$ shares the same with the one before modification almost surely, and thus changing Gaussian random matrix ensemble to modified version \textit{does not affect} the analysis below essentially. 

This modification is for better expression on definition of left and right unitary (orthogonal) matrix of SVD. For full rank square matrix $W = U \Sigma V^H$, $U$ and $V$ are not unique, but $VU^H$ is (even if the singular values are non-distinct, or changing the order of diagonal elements of $\Sigma$. This is due to the uniqueness of polar decomposition $W = S Q$ under full rank, where $Q=U V^H$, $S = (W W^H)^{1/2}$. ) and thus well-defined. 

Without changing the result, we analysis the initialization scheme of modified Gaussian random matrix ensemble instead. Then $W$ is full rank and thus polar decomposition is unique. 

Generally, suppose the right polar decomposition of $W$ is $W = \left(W W^H \right)^{1/2} Q$, then

\begin{equation}
 \begin{aligned}
  W + \left(WW^H\right)^{1/2} &= \left(W W^H \right)^{1/2} (I + Q) . 
 \end{aligned}
\end{equation}

If $\mathbb{F} = \mathbb{R}$, $\Pr(\det(W)>0) = \Pr(\det(W)<0) = \frac{1}{2}$ due to the symmetry of Gaussian random matrix ensemble. If $\det(W)= \det\left(\left(W W^H \right)^{1/2}\right) \det\left(Q\right) < 0$, $\det\left(Q\right) = -1$, then $\sigma_{\min}(I + Q) = 0$ and further $\sigma_{\min}\left(W + \left(WW^\top\right)^{1/2}\right) = 0$. 

Consider both $\mathbb{F} = \mathbb{C}$ and $\mathbb{F} = \mathbb{R}$, $\det(W) > 0$ (which indicates $\det\left(Q\right) = 1$): 

\begin{equation}
 \begin{aligned}
  \sigma_{\min}\left(W + \left(WW^H\right)^{1/2}\right) &\ge \sigma_{\min}\left(\left(WW^H\right)^{1/2}\right)\sigma_{\min}\left(I + Q \right) \\
  &= \sigma_{\min}(W)\sigma_{\min}\left(I + Q \right) \\
  &\ge \left[\prod_{k=1}^{N}\sigma_{\min}(W_k) \right]\sigma_{\min}\left(I + Q\right) . 
 \end{aligned}
\end{equation}

From Theorem 1.1 of \cite{tao2009randommatricesdistributionsmallest}, by applying union bound, $\sigma_{\min}(W_{k,k\in[1,N]\cap\mathbb{N}^*}) > f_1^{-1}(\delta,N) d^{-1/2} \epsilon$ with high probability $1- \delta/2$, where $f_1(\delta,N) = O\left(\frac{N}{\delta}\right)$. Then $\left[\prod_{k=1}^{N}\sigma_{\min}(W_k) \right] \ge \left(f_1^{-1}(\delta,N) d^{-1/2} \epsilon \right)^N$, and it remains to find lower bound for $\sigma_{\min}\left(I + Q\right)$. 

To apply results in Theorem \ref{minimum singular value under Haar measure}, it is sufficient to show that $Q$ follows Haar measure on $U(d,\mathbb{C})$ (or $O(d,\mathbb{R})$). 

Due to the property of invariance under left and right multiplication of unitary (orthogonal) matrix for Gaussian random matrix ensemble (Section 2.6.2, (2.131), \cite{taotopics}), $\forall$ fixed $Q_0\in U(d,\mathbb{C})$ (or $O(d,\mathbb{R})$ if $\mathbb{F} = \mathbb{R}$), $W_1 Q_0^H$ follows the same distribution as $W_1$ while still independent of $W_{k,k\in[2,N]\cap\mathbb{N}^*}$, resulting that $WQ_0^H$ follows the same distribution as $W$. Since the right polar decomposition of $W Q_0^H$ is $W Q_0^H = \left(W Q_0^H Q_0 W^H \right)^{1/2} Q Q_0^H = \left(W W^H \right)^{1/2} \left(Q Q_0^H\right)$ , we have 

\begin{equation}
  Q_0 Q \overset{d}{=} Q,\,\forall \text{ fixed } Q_0 \in U(d,\mathbb{C}) \text{ (or }O(d,\mathbb{R})\text{ if }\mathbb{F} = \mathbb{R}\text{)} . 
\end{equation}

Likewise 

\begin{equation}
  Q Q_0 \overset{d}{=} Q,\,\forall \text{ fixed } Q_0 \in U(d,\mathbb{C}) \text{ (or }O(d,\mathbb{R})\text{ if }\mathbb{F} = \mathbb{R}\text{)} . 
\end{equation}

From the fact that the only measure invariant under left (or right) multiplication of arbitrary element of a compact lie group is Haar measure, $Q$ follows Haar measure on $U(d,\mathbb{C})$ (or $O(d,\mathbb{R})$), and the proof is completed. 

\end{proof}

By Theorem \ref{Gaussian random matrix ensemble product, eigenvalues}, for depth $N=4$, if $\mathbb{F} = \mathbb{C}$ then with high probability $1 - \delta$ (if $\mathbb{F} = \mathbb{R}$ then with probability $1/2$, $\sigma_{\min}\left(W(0) + \left(W(0)W(0)^\top\right)^{1/2}\right) = 0$, and with probability $(1-\delta)/2$ the following holds), $\exists f_1(\delta) = O\left( \frac{1}{\delta} \right), f_2(\delta) = O(\frac{1}{\delta^5})$ such that 

\begin{equation}\label{random gaussian initialization, conclusions}
 \begin{aligned}
  \max_{j,k} \sigma_k(W_j(0)) &\le f_1(\delta) \sqrt{d} \epsilon \\ 
  \min_{j,k} \sigma_k(W_j(0)) &\le \frac{1}{f_1(\delta) \sqrt{d}}\cdot \epsilon \\ 
  \sigma_{\min}\left(W(0) + \left(W(0)W(0)^H\right)^{1/2}\right) &\ge \frac{1}{f_2(\delta)d^{3}} \cdot \epsilon^4 . 
 \end{aligned}
\end{equation}

Consequently, 

\begin{equation}
  e_{\Delta}(0) \coloneqq \left.\sqrt{\sum_{i=1}^{3} \|\Delta_{i,i+1}\|_F^2} 
  \right|_{t=0} \le \sqrt{3} \cdot 2 \sqrt{d} \cdot \max_{j,k} \sigma_k^2(W_j(0)) = 2\sqrt{3} f_1^2(\delta)d^{3/2}\epsilon^2 . 
\end{equation}

\subsection{Balanced Gaussian Initialization}

This section analyzes the balanced Gaussian initialization scheme. 

\begin{corollary}\label{Balanced Initialization: each matrix is a Gaussian random matrix ensemble}

Under balanced Gaussian initialization scheme (\ref{init, balanced}), each matrix $W_{k,k\in [1,N]\cap\mathbb{N}^*}$ is a Gaussian random matrix ensemble scaled by $\epsilon$. 

\end{corollary}

\begin{proof}

This is immediately from the property of invariance under left and right multiplication of unitary (orthogonal) matrix for Gaussian random matrix ensemble (Section 2.6.2, (2.131), \cite{taotopics}). 

\end{proof}

Due to Corollary \ref{ASVD, non-negative diagonal}, the product matrix can be expressed as $U \Sigma_w^{N} V^H$. Then we present the proof of Theorem \ref{Balanced initialization, final}. 

\begin{proof}\label{Balanced initialization, final, proof}

From (\ref{init, balanced}), $W(t=0) = s\epsilon^N Q_{N,N+1} (G^H G)^{N/2} Q_{01}^H$. 

Naturally $\|\Sigma_w\|_{op} = \epsilon \left\|(G^H G)^{1/2}\right\|_{op} = \epsilon \left\|G\right\|_{op} = O\left(1 + \sqrt{\frac{\ln\left( \frac{1}{\delta} \right)}{d}} \right) \sqrt{d} \epsilon$. Last step is from Theorem 4.4.5 of \cite{Vershynin_2018} directly. 

For the other two terms, 

\begin{equation}
 \begin{aligned}
 &\left. \sigma_{\min}\left((U+V)\Sigma_w\right) \right|_{t=0} \\
  =& \left.\sqrt{\lambda_{\min}\left((U+V)\Sigma_w^2(U+V)^H\right)}\right|_{t=0} \\
  =& \left.\sqrt{\lambda_{\min}\left( \left( WW^H \right)^{\frac{1}{N}} + \left(W^H W\right)^{\frac{1}{N}} + \left(WW^H\right)^{-\frac{N-2}{2N}}W + \left(W^H W\right)^{-\frac{N-2}{2N}} W^H \right)}\right|_{t=0} \\
  =& \epsilon \sqrt{\lambda_{\min}\left( \left( Q_{01} + s Q_{N,N+1} \right) (G^H G) \left( Q_{01} + s Q_{N,N+1} \right)^H \right)}\\
  \in& \left[\epsilon \sigma_{\min} (I + s Q_{01}^H Q_{N,N+1}) \sigma_{\min}(G), \epsilon \sigma_{\min} (I + s Q_{01}^H Q_{N,N+1}) \sigma_{\max}(G)\right] . 
 \end{aligned}
\end{equation}

And 

\begin{equation}
 \begin{aligned}
  \|(U-V)\Sigma_w\|_F|_{t=0} &\le 2\sqrt{d} \epsilon \|G\|_{op} . 
 \end{aligned}
\end{equation}

Since $Q_{N,N+1}$ and $Q_{01}$ are independent and both sampled from Haar measure, then $Q_{01}^H Q_{N,N+1} \sim \mathrm{Haar}$ on $U(d,\mathbb{C})$ (or $O(d,\mathbb{R})$ if $\mathbb{F} = \mathbb{R}$) as well. 

For $\mathbb{F} = \mathbb{R}$, since $s$ is independent of $Q_{j,j\in[0,N]\cap\mathbb{N}}$, $\Pr(s\det(Q_{N,N+1})\det(Q_{01})=1) = \Pr(s\det(Q_{N,N+1})\det(Q_{01})=-1) = \frac{1}{2}$ is directly from symmetry of Haar measure. 

Then by combining Theorem \ref{minimum singular value under Haar measure} and Theorem 4.4.5 of \cite{Vershynin_2018}, Theorem 1.1 of \cite{tao2009randommatricesdistributionsmallest} (with high probability $1-\delta^\prime$, $\max\left(\|G\|_{op},\|G^{-1}\|_{op}\right) \le f_1(\delta^\prime) \sqrt{d}$, $f_1(\delta^\prime) = O\left( \frac{1}{\delta^\prime}\right)$), the proof is completed.  

\end{proof}

\subsection{General Balanced Initialization}

This section introduces a property for general balanced and input-output orthogonal-invariant initialization (refer to Definition \ref{Input-Output Unitary(Orthogonal)-Invariant initialization}) under real field. 

\begin{theorem}\label{General init, 1/2 prob zero sigma_min}

For any real matrix factorization, if the initialization is balanced and input-output orthogonal-invariant, then the minimum singular value of $W + \left(WW^\top\right)^{1/2}$ at $t=0$ is exactly $0$ with at least probability $1/2$: 

\begin{equation}
  \Pr\left( \sigma_{\min}\left(W + \left(WW^\top\right)^{1/2} \right) = 0 \right) \ge 1/2 . 
\end{equation} 

\end{theorem}

\begin{proof}

As a direct consequence of Definition \ref{Input-Output Unitary(Orthogonal)-Invariant initialization}, $W$ is left and right orthogonal invariant: 

\begin{equation}
  W \overset{d}{=} U W V,\,\forall U,V \in O(d,\mathbb{R}) . 
\end{equation}

Suppose the right polar decomposition of $W$ is $W = WW^\top Q$, following the same arguments in the proof (\ref{Gaussian random matrix ensemble product, eigenvalues, proof}) of Theorem \ref{Gaussian random matrix ensemble product, eigenvalues}, 

\begin{equation}
  W + \left(WW^\top\right)^{1/2} = \left(WW^\top\right)^{1/2} (I+Q),\,Q \sim \mathrm{Haar} . 
\end{equation}

From Theorem \ref{minimum singular value under Haar measure}, $\Pr(\sigma_{\min}(I+Q)=0) = \frac{1}{2}$, resulting 

\begin{equation}
  \Pr\left(\sigma_{\min}\left(W + \left(WW^\top\right)^{1/2}\right)=0 \right) \ge \Pr(\sigma_{\min}(I+Q)=0) = \frac{1}{2} . 
\end{equation}

This completes the proof. 

\end{proof}

%% file: subfiles/appendix/appendix_3_basic_lemmas.tex
\section{Basic Lemmas}\label{section: basic lemmas}

\subsection{Classic Matrix Analysis Conclusions}

\begin{lemma}\label{bound of RR^H}
  Let $R\in \mathbb{F}^{d \times d}$, where $\mathbb{F} = \mathbb{C}$ or $\mathbb{R}$. Then: 
  
  1. $I - R R^H$ and $I - R^H R$ (or $I - R R^\top$ and $I - R^\top R$ if $\mathbb{F}=\mathbb{R}$) share the same set of eigenvalues. 

  2. These eigenvalues are real-valued.
\end{lemma}

\begin{proof}
  We prove the complex case, and the real case follows. Suppose the singular value decomposition of $R$ is $U_R \Sigma_R V_R^H$,  then 

  \begin{equation}
   \begin{aligned}
    I - R R^H &= I - U_R \Sigma_R^2 U_R^H = U_R \left( I - \Sigma_R^2 \right) U_R^H \\
    I - R^H R &= I - V_R \Sigma_R^2 V_R^H = V_R \left( I - \Sigma_R^2 \right) V_R^H.
   \end{aligned}
  \end{equation}

  Thus both $I - R R^H$ and $I - R^H R$ are unitarily similar to $I - \Sigma_R^2$, which completes the proof. 
\end{proof}

\begin{lemma}\label{error bound, sqrt}
Given symmetric matrices $X, \Delta \in \mathbb{F}^{d\times d}$, where $\mathbb{F} = \mathbb{C}$ or $\mathbb{R}$, suppose $X \succ \|\Delta\|_{op} I \succ O$, then

\begin{equation}
  \left\| X^{1/2} - (X+\Delta)^{1/2} \right\|_{op} \le \frac{\|\Delta\|_{op}}{2(\lambda_{\min}(X) - \|\Delta\|_{op})^{1/2}}.
\end{equation}

\end{lemma}

\begin{proof}

Directly by Theorem X.3.8 and inequality (X.46) in \cite{bhatia1996matrix}. 

\end{proof}

\begin{lemma}\label{error bound, inverse}

$\forall X, \Delta \in \mathbb{F}^{d\times d}$, where $\mathbb{F} = \mathbb{C}$ or $\mathbb{R}$, if $X$ and $X+\Delta$ are both invertible, then 

\begin{equation}
  (X + \Delta)^{-1} - \left( X^{-1} - X^{-1} \Delta X^{-1} \right) = X^{-1} \Delta X^{-1} \Delta (X + \Delta)^{-1}.
\end{equation}

\end{lemma}

\begin{proof}

\begin{equation}
 \begin{aligned}
  (X + \Delta)^{-1} - \left( X^{-1} - X^{-1} \Delta X^{-1} \right) &= X^{-1} \left[ X - (X - \Delta) X^{-1} (X + \Delta) \right]  (X+\Delta)^{-1} \\ 
  & = X^{-1} \Delta X^{-1} \Delta (X + \Delta)^{-1}.
 \end{aligned}
\end{equation}

\end{proof}

\begin{lemma}\label{bound of eigenvalues under perturbation}Bound of eigenvalues under perturbation. 

For unitary (or orthogonal, for real field) $d$-dimensional matrices $U$, $V$, positive semi-definite matrix $S$, denote $P \coloneqq \left(\frac{U+V}{2}\right) S \left(\frac{U+V}{2}\right)^H$, then the eigenvalues of $S$ are bounded by 

\begin{equation}
 \begin{aligned}
  \lambda_k\left( P \right) \le \lambda_k(S) \le \begin{cases}
    2 \left[ \lambda_k\left( P \right) + \left\| \left(\frac{U-V}{2}\right) S \left(\frac{U-V}{2}\right)^H \right\|_{op} \right] &,\, 1\le k \le d-1 \\
    \lambda_k\left( P \right) + \left\| \left(\frac{U-V}{2}\right) S \left(\frac{U-V}{2}\right)^H \right\|_{op} &,\, k=d
  \end{cases}\quad.
 \end{aligned}
\end{equation}

\end{lemma}

\begin{proof}

Let $Q = U^H V$. 

Due to Courant-Fischer min-max Theorem, $A\succeq B$ indicates $\lambda_k(A) \ge \lambda_k(B)$. Then the lower bound is straight forward: 

\begin{equation}
 \begin{aligned}
  &\lambda_k\left(\left(\frac{U+V}{2}\right) S \left(\frac{U+V}{2}\right)^H\right) = \lambda_k\left( S^{1/2}\left(\frac{U+V}{2}\right) \left(\frac{U+V}{2}\right)^H S^{1/2} \right) \\
  \le& \lambda_k\left( S^{1/2} \left(\left\|\frac{U+V}{2}\right\|_{op}^2 I \right) S^{1/2} \right) \\\le& \lambda_k\left( S^{1/2} \left(\left(\frac{\left\|U\right\|_{op}+\left\|V\right\|_{op}}{2}\right)^2 I \right) S^{1/2} \right) = \lambda_k\left( S \right).
 \end{aligned}
\end{equation}

For upper bound, by applying Wely inequality, 

\begin{equation}
 \begin{aligned}
  &\lambda_k\left(\left(\frac{U+V}{2}\right) S \left(\frac{U+V}{2}\right)^H\right) = \lambda_k\left(\left(\frac{I+Q}{2}\right) S \left(\frac{I+Q^H}{2}\right)\right) \\ \ge& \lambda_k\left(\left(\frac{I+Q}{2}\right) S \left(\frac{I+Q^H}{2}\right)+\left(\frac{I-Q}{2}\right) S \left(\frac{I-Q^H}{2}\right)\right) - \left\| \left(\frac{I-Q}{2}\right) S \left(\frac{I-Q^H}{2}\right) \right\|_{op} \\ 
  =& \frac{1}{2}\lambda_k\left( S + Q S Q^H \right) - \left\| \left(\frac{U-V}{2}\right) S \left(\frac{U-V}{2}\right)^H \right\|_{op}.
 \end{aligned}
\end{equation}

For arbitrary $k$, $\lambda_k\left( S + Q S Q^H \right) \ge \lambda_k\left( S \right)$; for $k = d$, $\lambda_d\left( S + Q S Q^H \right) \ge 2\lambda_d\left( S \right)$. This completes the proof. 

\end{proof}

\subsection{Lemmas on Analytic Singular Value Decomposition of Product Matrix under Balanced Initialization and Gradient Flow}

\begin{lemma}\label{ASVD existence}

Existence of analytic singular value decomposition (ASVD). 

Under Section~\ref{section: problem formulation} with gradient flow and balanced initialization, for $t \in \mathbb{R}^+\cup\{0\}$, there exists analytical singular value decompositions for $W_{j,j\in[1,N]\cap\mathbb{N}^*}(t)$ and $W(t)$. 

\end{lemma}

\begin{proof}

For $\mathbb{F} = \mathbb{R}$, the proof is exactly the same as Lemma 1 in \cite{arora2019implicitregularizationdeepmatrix}: real analytic matrices have ASVD (Theorem 1 in \cite{Bunse1991/92}), and $W_j(t)$ are analytic then so does $W(t)$. For complex case, Theorem 1 and 3 in \cite{demoor1989analyticProperties} gives that complex analytic matrices (of a real parameter) have ASVD, then the rest of proof follows.

\end{proof}

\begin{remark}

For complex field here, the "analytic" here has \textbf{no relation with the standard definition of "complex analytic function"}, who has complex parameters and consequently more restrictions on definition of derivatives. 

Throughout the proof for gradient flow (continuous time), we only deal with real-valued parameter $t\in\mathbb{R}^+\cup\{0\}$, so any "analytic" means real-analytic (for $\mathbb{F} = \mathbb{C}$, it means the real and imaginary part are both real-analytic), not complex-analytic. 

\end{remark}

\begin{lemma}\label{singular value derivative}

Suppose the analytic singular value decomposition of $M(t)$ exists and is $U(t) \Sigma_M(t) V^H(t)$, $M(t)\in \mathbb{F}^{d\times d}$, where $\mathbb{F} = \mathbb{C}$ or $\mathbb{R}$, then the derivative of the $k^{th}$ singular value is 

\begin{equation}
  \frac{\mathrm{d} \sigma_k(M)}{\mathrm{d} t} = \Re \left(u_k^H \frac{\mathrm{d} M}{\mathrm{d} t} v_k\right),
\end{equation}

where $u_k$, $v_k$ are the $k^{th}$ column vectors of left and right unitary (or orthogonal if $\mathbb{F} = \mathbb{R}$) matrices respectively. 
\end{lemma}

\begin{proof} We prove the case when $\mathbb{F} = \mathbb{C}$. For $\mathbb{F} = \mathbb{R}$, replace $\cdot^H$ by $\cdot^\top$.

\begin{equation}
  \frac{\mathrm{d} M}{\mathrm{d} t} = \frac{\mathrm{d} U}{\mathrm{d} t} \Sigma_M V^H + U\frac{\mathrm{d} \Sigma_M}{\mathrm{d} t} V^H  + U \Sigma_M \frac{\mathrm{d} V}{\mathrm{d} t}^H .
\end{equation}

Then 

\begin{equation}
 \begin{aligned}
  \Re \left( u_k^H \frac{\mathrm{d} M}{\mathrm{d} t} v_k \right) &= \Re\left( u_k^H \frac{\mathrm{d} U}{\mathrm{d} t} \Sigma_M V^H v_k + u_k^H U\frac{\mathrm{d} \Sigma_M}{\mathrm{d} t} V^H v_k  + u_k^H U \Sigma_M \frac{\mathrm{d} V}{\mathrm{d} t}^H v_k \right) \\&= \frac{\mathrm{d} \sigma_k(M)}{\mathrm{d} t} + \sigma_k(M) \left( \Re \left( u_k^H \frac{\mathrm{d} u_k}{\mathrm{d} t}\right) + \Re \left( \frac{\mathrm{d} v_k^H}{\mathrm{d} t} v_k \right) \right) .
 \end{aligned}
\end{equation}

From $\Re \left( u_k^H \frac{\mathrm{d} u_k}{\mathrm{d} t} \right) = \frac{\mathrm{d} }{\mathrm{d} t} \left(\frac{1}{2} \|u_k\|^2 \right) = 0$, $\Re \left( \frac{\mathrm{d} v_k^H}{\mathrm{d} t} v_k \right) = \frac{\mathrm{d} }{\mathrm{d} t} \left(\frac{1}{2} \|v_k\|^2 \right) = 0$, the proof is done.

\end{proof}

\begin{remark}

If $M$ is hermitian, then the $\Re$ can be omitted. 

\end{remark}

\begin{remark}

This generalizes Lemma 2 in \cite{arora2019implicitregularizationdeepmatrix} from real field into complex field by adding a $\Re$ on the right side: 

\begin{equation}\label{thm 3, arora, complex}
  \frac{\mathrm{d}\sigma_r(S)}{\mathrm{d}t} = -N (\sigma_r^{2}(S))^{1-1/N} \cdot \Re \left( \left\langle \nabla_W \mathcal{L}(W), u_r v_r^H \right\rangle \right) .
\end{equation}

\end{remark}

\begin{lemma}\label{L ori non-increasing}

Under Section \ref{section: problem formulation} with gradient flow, $\mathcal{L}_{\rm ori}$ is non-increasing. 

For $t\in[0,+\infty)$, 

\begin{equation}
  \frac{\mathrm{d}}{\mathrm{d} t} \mathcal{L}_{\rm ori} \le - 2N \min_{j,k} |\sigma_{k}(W_j)|^{2(N-1)} \mathcal{L}_{\rm ori} .
\end{equation}

\end{lemma}

\begin{proof}

Naturally we have the derivative of product matrix $W(t)$: 

\begin{equation}\label{equation: dW/dt is irrelevant to a}
 \begin{aligned}
  \frac{\mathrm{d} W}{\mathrm{d} t} &= \sum_{j=1}^{N} W_{\prod_L , j+1} \left[W_{\prod_L , j+1}^H \left(\Sigma - W \right) W_{\prod_R , j-1}^H + a \left(W_j \Delta_{j-1,j} - \Delta_{j,j+1} W_j \right) \right] W_{\prod_R , j-1} \\ 
  &= \sum_{j=1}^{N} W_{\prod_L , j+1} W_{\prod_L , j+1}^H \left(\Sigma - W \right) W_{\prod_R , j-1}^H W_{\prod_R , j-1} \\ 
  &+ a \sum_{j=1}^{N} W_{\prod_L , j} \Delta_{j-1,j} W_{\prod_R , j-1} - a \sum_{j=1}^{N} W_{\prod_L , j+1} \Delta_{j,j+1} W_{\prod_R , j}\\ 
  &= \sum_{j=1}^{N} W_{\prod_L , j+1} W_{\prod_L , j+1}^H \left(\Sigma - W \right) W_{\prod_R , j-1}^H W_{\prod_R , j-1} + a \left(W \Delta_{0,1} - \Delta_{N,N+1} W\right) \\ 
  &= \sum_{j=1}^{N} W_{\prod_L , j+1} W_{\prod_L , j+1}^H \left(\Sigma - W \right) W_{\prod_R , j-1}^H W_{\prod_R , j-1} .
 \end{aligned}
\end{equation}

Then 

\begin{equation}
 \begin{aligned}
  \frac{\mathrm{d}}{\mathrm{d} t} \mathcal{L}_{\rm ori} &= - \Re\left(\left\langle \Sigma - W, \frac{\mathrm{d} W}{\mathrm{d} t}\right\rangle\right) \\ 
  &= - \Re\left(\left\langle \Sigma - W, \sum_{j=1}^{N} W_{\prod_L , j+1} W_{\prod_L , j+1}^H \left(\Sigma - W \right) W_{\prod_R , j-1}^H W_{\prod_R , j-1} \right\rangle\right) \\ 
  &= - \sum_{j=1}^{N}\Re\left(\left\langle \Sigma - W, W_{\prod_L , j+1} W_{\prod_L , j+1}^H \left(\Sigma - W \right) W_{\prod_R , j-1}^H W_{\prod_R , j-1} \right\rangle\right) \\ 
  &= - \sum_{j=1}^{N}\Re\left(\left\langle W_{\prod_L , j+1}^H \left(\Sigma - W \right) W_{\prod_R , j-1}^H, W_{\prod_L , j+1}^H \left(\Sigma - W \right) W_{\prod_R , j-1}^H \right\rangle\right) \\ 
  &= - \sum_{j=1}^{N} \left\| W_{\prod_L , j+1}^H \left(\Sigma - W \right) W_{\prod_R , j-1}^H \right\|_F^2 .
 \end{aligned}
\end{equation}

From $\|LXR\|_F \ge \sigma_{\min}(L) \sigma_{\min}(R) \|X\|_F$, $\sigma_{\min}\left(W_{\prod_L , j+1}^H\right) \ge \min_{j,k} |\sigma_{k}(W_j)|^{N-j}$ and $\sigma_{\min}\left(W_{\prod_R , j-1}^H\right) \ge \min_{j,k} |\sigma_{k}(W_j)|^{j-1}$, the proof is completed. 

\end{proof}

\begin{lemma}\label{ASVD, non-negative diagonal}

Analytic singular value decomposition of product matrix with positive semi-definite diagonal matrix. 

Under Section \ref{section: problem formulation} with gradient flow and any bounded (i.e. $W_{j,j\in[1,N]\cap\mathbb{N}^*}(t=0)$ is bounded) balanced initialization, $\forall N\in[2,+\infty)\cap\mathbb{N}^*$, the product matrix $W(t)$ can be expressed as: 

\begin{equation}
  W(t) = U(t) S(t) V(t)^H ,
\end{equation}

where: $U(t) \in \mathbb{F}^{d\times d}$, $S(t) \in \mathbb{R}^{d\times d}$ and $V(t) \in \mathbb{F}^{d\times d}$ are analytic functions of $t$, $U(t)$ and $V(t)$ are orthogonal matrices, $S(t)$ is diagonal and \textit{positive semi-definite} (elements on its diagonal may appear in any order), $\Sigma_w(t) \coloneqq S(t)^{1/N}$ is \textit{well-defined} (meaning the real-valued operation $S_{ii} \mapsto (S_{ii})^{1/N}$ is applied to each diagonal element of $S(t)$, resulting in another semi-positive diagonal matrix) and analytic. 

Moreover, if the singular values of product matrix $W$ are non-zero, then throughout the optimization $W$ remains full rank in finite time. 

\end{lemma}

\begin{proof}

From Lemma \ref{ASVD existence}, it is left to construct a new ASVD (analytic singular value decomposition) of $W(t)$ using existed ASVD $W(t) = U(t) S(t) V(t)^H$ ($S(t)$ is not guaranteed to be positive semi-definite). 

By Lemma \ref{L ori non-increasing}, $\left\| \Sigma -W \right\|_{F} \le \left\| \Sigma -W(t=0) \right\|_{F}$. Then the following term is bounded by a constant for all $t\in \mathbb{R}^+\cup\{0\}$:

\begin{equation}
 \begin{aligned}
  \left|\left\langle \nabla l(W(t)), u_r(t) v_r(t)^H \right\rangle\right| &\le \left\| \nabla l(W(t))\right\|_{op} = \left\| \Sigma -W \right\|_{op} \\&\le \left\| \Sigma -W \right\|_{F} \le \left\| \Sigma -W(t=0) \right\|_{F} .
 \end{aligned}
\end{equation}

By invoking Theorem 3 in \cite{arora2019implicitregularizationdeepmatrix} (for complex case, add $\Re$), the absolute value of time derivative of $\sigma_r(t)$ is bounded by: 

\begin{equation}
  \left|\frac{\mathrm{d} \sigma_r(t)}{\mathrm{d}t} \right| \le \left\| \Sigma -W(t=0) \right\|_{F} \cdot N\left(\sigma_r^2(t)\right)^{1-1/N} .
\end{equation}

Thus all $\sigma_r(t)$ do not change sign for $t\in \mathbb{R}^+\cup\{0\}$. Moreover, if $|\sigma_r(t=0)|>0$, the it never decrease to $0$ in finite time. 

Then we construct $S_{\rm new}(t)$ by flipping the sign of negative diagonal terms, and $U_{\rm new}(t)$ by changing the sign of corresponding columns of $U(t)$. Now $W(t) = U_{\rm new}(t) S_{\rm new}(t) V(t)^H$ is also an ASVD of $W(t)$, $U_{\rm new}(t)$ is analytic and unitary (orthogonal), $S_{\rm new}(t)$ is analytic, diagonal and positive semi-definite. 

Specially, if for some $r$, $\sigma_r(t) = 0$ at time $t$, then it remains zero. Thus, from $S_{\rm new}(t)$ is analytic, so is $\Sigma_w(t)$. This completes the proof. 

\end{proof}

Finally, we generalize Lemma 2 in \cite{arora2019implicitregularizationdeepmatrix} into complex field. Here we assume all matrices are square matrices of dimension $d\times d$. 

\begin{lemma}\label{lemma 2, arora, complex}

Under balanced initialization, assume the singular values of $W(t) = U(t) S(t) V(t)^H$ ($U$, $V$ are unitary, $S$ is real-valued and diagonal) are distinct and different from zero at initialization, then the derivatives of $U$, $V$ satisfy 

\begin{equation}
 \begin{aligned}
  \frac{\mathrm{d} U}{\mathrm{d} t} = U \left( F \odot M_U + D_U\right),\, \frac{\mathrm{d} V}{\mathrm{d} t} = V \left( F \odot M_V + D_V \right) , 
 \end{aligned}
\end{equation}

where $D_U$, $D_V$ are diagonal matrices with pure imaginary entries (and thus skew-Hermitian) satisfying

\begin{equation}\label{lemma 2, arora, complex, detail1}
 \begin{aligned}
  (D_U)_{jj} - (D_V)_{jj} &= - \frac{N}{2} \left(\sigma_j^2(S)\right)^{1/2 -1/N} \left[ \left( U^H (\nabla_W \mathcal{L}_{\rm ori}) V \right)_{jj} - \left( V^H (\nabla_W \mathcal{L}_{\rm ori})^H U \right)_{jj} \right] ,  
 \end{aligned}
\end{equation}

and 

\begin{equation}\label{lemma 2, arora, complex, detail2}
 \begin{aligned}
  M_U &= - \left[ U^H (\nabla_W \mathcal{L}_{\rm ori}) V S + S V^H (\nabla_W \mathcal{L}_{\rm ori})^H U \right] \\ 
  M_V &= - \left[ V^H (\nabla_W \mathcal{L}_{\rm ori})^H U S + S U^H (\nabla_W \mathcal{L}_{\rm ori}) V \right] . 
 \end{aligned}
\end{equation}

Here $\odot$ stands for Hadamard (element-wise) product and $F$ is defined by 

\begin{equation}
  F_{jk} = \begin{cases}
      0 &,\, j=k \\
      \frac{1}{\left(\sigma_{k}^2(S)\right)^{1/N} - \left(\sigma_{j}^2(S)\right)^{1/N}} &,\, j\ne k . 
  \end{cases}
\end{equation}

\end{lemma}

\begin{remark}

Note that only the difference $D_U - D_V$ is uniquely determined. Adding the same purely imaginary diagonal matrix to both $D_U$ and $D_V$ leaves the dynamics of $W$ unchanged, corresponding to a shared phase rotation of $U$ and $V$. 

For real matrices, R.H.S. of equation (\ref{lemma 2, arora, complex, detail1}) is zero, $D_U = D_V = O$, then this Lemma degenerates into Lemma 2 of \cite{arora2019implicitregularizationdeepmatrix}. 

\end{remark}

\begin{proof}

We calculate the time derivative of $U$ and the time derivative of $V$ follows the same way. 

Following the derivations in \cite{arora2019implicitregularizationdeepmatrix}, 

\begin{equation}\label{lemma 2, arora, complex, proof, eq1}
 \begin{aligned}
  U^H \frac{\mathrm{d} W}{\mathrm{d} t} V = U^H \frac{\mathrm{d} U}{\mathrm{d} t} S + \frac{\mathrm{d} S}{\mathrm{d} t} + S \frac{\mathrm{d} V}{\mathrm{d} t}^H V , 
 \end{aligned}
\end{equation}

where $U^H \frac{\mathrm{d} U}{\mathrm{d} t} = - \frac{\mathrm{d} U}{\mathrm{d} t}^H U$ and $V^H \frac{\mathrm{d} V}{\mathrm{d} t} = - \frac{\mathrm{d} V}{\mathrm{d} t}^H V$ are skew-Hermitian matrices, whose diagonal entries are therefore purely imaginary. Since $S$ is real, denote $\bar{I}_d$ to be a matrix holding zeros on its diagonal and ones elsewhere, 

\begin{equation}
 \begin{aligned}
  \Re \left( \bar{I}_d \odot \left( U^H \frac{\mathrm{d} W}{\mathrm{d} t} V S + S V^H \frac{\mathrm{d} W}{\mathrm{d} t}^H U \right) \right) &= \Re \left( U^H \frac{\mathrm{d} U}{\mathrm{d} t} S^2 - S^2 U^H \frac{\mathrm{d} U}{\mathrm{d} t} \right) \\ 
  \Im \left( U^H \frac{\mathrm{d} W}{\mathrm{d} t} V S + S V^H \frac{\mathrm{d} W}{\mathrm{d} t}^H U \right) &= \Im \left( U^H \frac{\mathrm{d} U}{\mathrm{d} t} S^2 - S^2 U^H \frac{\mathrm{d} U}{\mathrm{d} t} \right) . 
 \end{aligned}
\end{equation}

Since $U^H \frac{\mathrm{d} W}{\mathrm{d} t} V S + S V^H \frac{\mathrm{d} W}{\mathrm{d} t}^H U$ is Hermitian, its diagonal entries are real, further giving $\Im \left( U^H \frac{\mathrm{d} W}{\mathrm{d} t} V S + S V^H \frac{\mathrm{d} W}{\mathrm{d} t}^H U \right) = \Im \left( \bar{I}_d \odot \left( U^H \frac{\mathrm{d} W}{\mathrm{d} t} V S + S V^H \frac{\mathrm{d} W}{\mathrm{d} t}^H U \right) \right)$. Combining the real and imaginary parts gives 

\begin{equation}
  \bar{I}_d \odot \left( U^H \frac{\mathrm{d} W}{\mathrm{d} t} V S + S V^H \frac{\mathrm{d} W}{\mathrm{d} t}^H U \right) = U^H \frac{\mathrm{d} U}{\mathrm{d} t} S^2 - S^2 U^H \frac{\mathrm{d} U}{\mathrm{d} t} . 
\end{equation}

Here $U^H \frac{\mathrm{d} W}{\mathrm{d} t} V = - \sum_{j=1}^{N} (S^2)^{\frac{j-1}{N}} U^H (\nabla_W \mathcal{L}_{\rm ori}) V (S^2)^{\frac{N-j}{N}}$. Then the non-diagonal entries of $U^H \frac{\mathrm{d} U}{\mathrm{d} t}$ follows by the proof of Lemma 2 in \cite{arora2019implicitregularizationdeepmatrix}. 

For the diagonal entries of $U^H \frac{\mathrm{d} U}{\mathrm{d} t}$, by taking imaginary part of equation (\ref{lemma 2, arora, complex, proof, eq1}), 

\begin{equation}
 \begin{aligned}
  &\sigma_j(S) \left( \left( U^H \frac{\mathrm{d} U}{\mathrm{d} t}\right)_{jj} - \left( V^H \frac{\mathrm{d} V}{\mathrm{d} t} \right)_{jj} \right) = i\Im \left(\sigma_j(S) \left( \left( U^H \frac{\mathrm{d} U}{\mathrm{d} t}\right)_{jj} - \left( V^H \frac{\mathrm{d} V}{\mathrm{d} t} \right)_{jj} \right)\right) \\=& i \Im\left(U^H \frac{\mathrm{d} W}{\mathrm{d} t} V \right)_{jj} = \frac{1}{2} \left( \left(U^H \frac{\mathrm{d} W}{\mathrm{d} t} V \right)_{jj} - \left(V^H \frac{\mathrm{d} W}{\mathrm{d} t}^H U \right)_{jj} \right) . 
 \end{aligned}
\end{equation}

The last step uses the fact that $i\Im(z) = \frac{1}{2} \left( z - \bar{z} \right) $. This deduces that 

\begin{equation}
  \left( U^H \frac{\mathrm{d} U}{\mathrm{d} t}\right)_{jj} - \left( V^H \frac{\mathrm{d} V}{\mathrm{d} t} \right)_{jj} = - \frac{N}{2}  \left(\sigma_j^2(S)\right)^{1/2 -1/N} \left[ \left( U^H (\nabla_W \mathcal{L}_{\rm ori}) V \right)_{jj} - \left( V^H (\nabla_W \mathcal{L}_{\rm ori})^H U \right)_{jj} \right] . 
\end{equation}

This completes the proof. 

\end{proof}

\subsection{Lemmas on Regularization, Gradient Flow}

\begin{lemma}\label{regularization, total}

Consider optimizing a generalized loss function coupled with a generalized regularization term using gradient flow: 

\begin{equation}\label{generalized loss and reg}
  \mathcal{L}(W_1, \cdots,W_N) \coloneqq \mathcal{L}_{\rm ori}\left( \prod_{j=N}^{1} W_j \right) + \frac{1}{4} \sum_{j=1}^{N-1} a_{j,j+1} \| \Delta_{j,j+1} \|_F^2,\,a_{j,j+1}\in \mathbb{R}^+\cup\{0\} .
\end{equation}

Then the regularization terms decays by: 

\begin{equation}
  \frac{\mathrm{d} }{\mathrm{d} t} \left( \sum_{j=1}^{N-1} a_{j,j+1} \| \Delta_{j,j+1} \|_F^2 \right) = -4 \sum_{j=1}^{N} \left\| a_{j,j+1} \Delta_{j,j+1} W_j -  a_{j-1,j} W_j \Delta_{j-1,j} \right\|_F^2.
\end{equation}
  
\end{lemma}

\begin{proof}

\begin{equation}
  \begin{aligned}
    \frac{\mathrm{d} }{\mathrm{d} t} W_j W_j^H &= - \Big[ \left(\nabla_{W_j} \mathcal{L}_{\rm ori}\right) W_j^H + W_j \left(\nabla_{W_j} \mathcal{L}_{\rm ori}\right)^H \\ & - 2a_{j-1,j} W_j \Delta_{j-1,j} W_j^H \\&+ a_{j,j+1} \left(\Delta_{j,j+1} W_j W_j^H + W_j W_j^H \Delta_{j,j+1} \right) \Big]\\
    \frac{\mathrm{d} }{\mathrm{d} t} W_{j+1}^H W_{j+1} &= - \Big[ \left(\nabla_{W_{j+1}} \mathcal{L}_{\rm ori}\right)^H W_{j+1} + W_{j+1}^H \left(\nabla_{W_{j+1}} \mathcal{L}_{\rm ori}\right) \\ &+ 2a_{j+1,j+2} W_{j+1}^H \Delta_{j+1,j+2} W_{j+1} \\&- a_{j,j+1} \left( \Delta_{j,j+1} W_{j+1}^H W_{j+1} + W_{j+1}^H W_{j+1} \Delta_{j,j+1} \right) \Big] .
  \end{aligned}
\end{equation}

Denote $W_{\prod_L , j} \coloneqq \prod_{k=N}^{j} W_k ,\, W_{\prod_R , j} \coloneqq \prod_{k=j}^{1} W_k,\, W \coloneqq \prod_{k=N}^{1} W_k = W_{\prod_L , 1} = W_{\prod_R , N}$. From property of the loss $\mathcal{L}_{\rm ori}$, 

\begin{equation}
  \begin{aligned}
    \left(\nabla_{W_j} \mathcal{L}_{\rm ori}\right) W_j^H = W_{\prod_L,j+1}^H \left(\nabla_{W} \mathcal{L}_{\rm ori}(W)\right) W_{\prod_R,j} = W_{j+1}^H \left(\nabla_{W_{j+1}} \mathcal{L}_{\rm ori}\right),\, \forall j\in [1,N-1] \cap \mathbb{N}^* .
  \end{aligned}
\end{equation}

Thus we have 

\begin{equation}
  \begin{aligned}
    \frac{\mathrm{d} }{\mathrm{d} t} \Delta_{j,j+1} &= 2a_{j-1,j} W_j \Delta_{j-1,j} W_j^H + 2a_{j+1,j+2} W_{j+1}^H \Delta_{j+1,j+2} W_{j+1} \\ &- a_{j,j+1} \left(\Delta_{j,j+1} \left( W_j W_j^H + W_{j+1}^H W_{j+1} \right) + \left( W_j W_j^H + W_{j+1}^H W_{j+1} \right) \Delta_{j,j+1} \right) , 
  \end{aligned}
\end{equation}

\begin{equation}
  \begin{aligned}
    \frac{\mathrm{d} \| \Delta_{j,j+1} \|_F^2}{\mathrm{d} t} &= 4a_{j-1,j} \mathrm{tr} \left( W_j \Delta_{j-1,j} W_j^H \Delta_{j,j+1} \right) \\&+ 4a_{j+1,j+2} \mathrm{tr} \left( W_{j+1} \Delta_{j,j+1} W_{j+1}^H \Delta_{j+1,j+2} \right) \\
    &- 4a_{j,j+1} \mathrm{tr}\left( ( W_j W_j^H + W_{j+1}^H W_{j+1} ) \Delta_{j,j+1}^2 \right) \\ 
    &= -\frac{2}{a_{j,j+1}} \bigg[ \left\| a_{j,j+1} \Delta_{j,j+1} W_j - a_{j-1,j} W_j \Delta_{j-1,j} \right\|_F^2 \\&+ \left\| a_{j+1,j+2} \Delta_{j+1,j+2} W_{j+1} - a_{j,j+1} W_{j+1} \Delta_{j,j+1} \right\|_F^2 \\ &+ a_{j,j+1}^2\left( \| \Delta_{j,j+1} W_j \|_F^2 + \| W_{j+1} \Delta_{j,j+1} \|_F^2 \right) \\&- a_{j-1,j}^2 \| W_{j} \Delta_{j-1,j} \|_F^2 - a_{j+1,j+2}^2 \| \Delta_{j+1,j+2} W_{j+1} \|_F^2 \bigg] . 
  \end{aligned}
\end{equation}

By taking weighted sum, 

\begin{equation}
  \frac{\mathrm{d} }{\mathrm{d} t} \left( \sum_{j=1}^{N-1} a_{j,j+1} \| \Delta_{j,j+1} \|_F^2 \right) = -4 \sum_{j=1}^{N} \left\| a_{j,j+1} \Delta_{j,j+1} W_j - a_{j-1,j} W_j \Delta_{j-1,j} \right\|_F^2 . 
\end{equation}
\end{proof}

Below we back to $a_{j,j+1} \equiv a\in \mathbb{R}^+\cup\{0\}$, $\forall j \in [1,N-1] \cap \mathbb{N}^*$. Then \ref{generalized loss and reg} becomes

\begin{equation}\label{generalized loss and reg, same a}
  \mathcal{L}(W_1, \cdots,W_N) \coloneqq \mathcal{L}_{\rm ori}\left( \prod_{j=N}^{1} W_j \right) + \frac{1}{4} \sum_{j=1}^{N-1} a \| \Delta_{j,j+1} \|_F^2,\,a\in \mathbb{R}^+\cup\{0\} . 
\end{equation}

\begin{theorem}\label{regularization term, convergence bound}

Suppose for all $j \in [1,N]\cap \mathbb{N}^*$, $\sigma_{\min}(W_j) \ge \delta$, $\sigma_{\max}(W_j) \le M$. Consider optimizing \ref{generalized loss and reg, same a} using gradient flow, then the convergence rate of the regularization term is lower bounded: 

\begin{equation}
  \frac{\mathrm{d}}{\mathrm{d} t} \left(\sum_{j=1}^{N-1} \| \Delta_{j,j+1} \|_F^2 \right) \le  -4 a \cdot \frac{2}{N - 1} \frac{M^2 - \delta^2}{\left( \frac{M}{\delta} \right)^{2\lfloor N/2 \rfloor} - 1} \cdot \left(\sum_{j=1}^{N-1} \| \Delta_{j,j+1} \|_F^2 \right) . 
\end{equation}

\end{theorem}

\begin{proof}

Denote $D_j = \Delta_{j,j+1} W_j - W_j \Delta_{j-1,j}$. Then 

\begin{equation}
  \Delta_{j,j+1} = (D_j + W_j \Delta_{j-1,j}) W_j^{-1} . 
\end{equation}

Deducing 

\begin{equation}
  \|\Delta_{j,j+1}\|_F \le \left\|W_j^{-1}\right\|_{op} \left( \|D_j\|_F + \|\Delta_{j-1,j}\|_F \|W_j\|_{op} \right) \le \frac{1}{\delta}\|D_j\|_F + \frac{M}{\delta} \|\Delta_{j-1,j}\|_F . 
\end{equation}

From $\Delta_{0,1} =O$, inductively we have

\begin{equation}
 \begin{aligned}
  \|\Delta_{j,j+1}\|_F^2 &\le \frac{1}{\delta^2} \left( \sum_{k=1}^{j} \left( \frac{M}{\delta} \right)^{j-k} \|D_{k}\|_F \right)^2 \le \frac{1}{\delta^2} \left( \sum_{k=1}^{j}  \left( \frac{M}{\delta} \right)^{2(j-k)} \right) \left( \sum_{k=1}^{j} \|D_{k}\|_F^2 \right) \\&= \frac{1}{\delta^2} \frac{\left( \frac{M}{\delta} \right)^{2 j} - 1}{\left( \frac{M}{\delta} \right)^{2}-1} \sum_{k=1}^{j} \|D_{k}\|_F^2 . 
 \end{aligned}
\end{equation}

The last two step use Cauchy-Schwarz inequality. 

From $\Delta_{N,N+1} =O$, following the same procedure we have 

\begin{equation}
  \|\Delta_{N-j,N-j+1}\|_F^2 \le \frac{1}{\delta^2} \frac{\left( \frac{M}{\delta} \right)^{2 j} - 1}{\left( \frac{M}{\delta} \right)^{2}-1} \sum_{k=N-j+1}^{N} \|D_{k}\|_F^2 . 
\end{equation}

Summing all terms up, for odd $N$ we have

\begin{equation}
  \begin{aligned}
    \sum_{j=1}^{N-1} \|\Delta_{j,j+1}\|_F^2 &= \sum_{j=1}^{(N-1)/2} \left( \|\Delta_{j,j+1}\|_F^2 + \|\Delta_{N-j,N-j+1}\|_F^2 \right) \\
    &\le \sum_{j=1}^{(N-1)/2} \left( \frac{1}{\delta^2} \frac{\left( \frac{M}{\delta} \right)^{2 j} - 1}{\left( \frac{M}{\delta} \right)^{2}-1} \sum_{k=1}^{j} \left( \|D_{k}\|_F^2 + \|D_{N+1-k}\|_F^2 \right)\right) \\
    &= \sum_{k=1}^{(N-1)/2} \left( \left( \|D_{k}\|_F^2 + \|D_{N+1-k}\|_F^2 \right) \sum_{j=k}^{(N-1)/2} \left(\frac{1}{\delta^2} \frac{\left( \frac{M}{\delta} \right)^{2 j} - 1}{\left( \frac{M}{\delta} \right)^{2}-1}\right) \right) \\
    &\le \frac{N - 1}{2}\frac{\left( \frac{M}{\delta} \right)^{N-1} - 1}{M^2 - \delta^2} \left( \sum_{k=1}^{N}\|D_k\|^2 \right) . 
  \end{aligned}
\end{equation}

For even $N$, 

\begin{equation}
  \begin{aligned}
    \sum_{j=1}^{N-1} \|\Delta_{j,j+1}\|_F^2 &= \sum_{j=1}^{N/2-1} \left( \|\Delta_{j,j+1}\|_F^2 + \|\Delta_{N-j,N-j+1}\|_F^2 \right) + \|\Delta_{N/2,N/2+1}\|_F^2 \\
    &\le \sum_{j=1}^{N/2-1} \left( \frac{1}{\delta^2} \frac{\left( \frac{M}{\delta} \right)^{2 j} - 1}{\left( \frac{M}{\delta} \right)^{2}-1} \sum_{k=1}^{j} \left( \|D_{k}\|_F^2 + \|D_{N+1-k}\|_F^2 \right)\right) \\&+ \frac{1}{2\delta^2} \frac{\left( \frac{M}{\delta} \right)^{N} - 1}{\left( \frac{M}{\delta} \right)^{2}-1} \sum_{k=1}^{N/2} \left( \|D_{k}\|_F^2 + \|D_{N+1-k}\|_F^2 \right) \\
    &= \sum_{k=1}^{N/2-1} \left( \left( \|D_{k}\|_F^2 + \|D_{N+1-k}\|_F^2 \right) \sum_{j=k}^{N/2-1} \left(\frac{1}{\delta^2} \frac{\left( \frac{M}{\delta} \right)^{2 j} - 1}{\left( \frac{M}{\delta} \right)^{2}-1}\right) \right) \\&+ \frac{1}{2\delta^2} \frac{\left( \frac{M}{\delta} \right)^{N} - 1}{\left( \frac{M}{\delta} \right)^{2}-1} \sum_{k=1}^{N/2} \left( \|D_{k}\|_F^2 + \|D_{N+1-k}\|_F^2 \right) \\
    &\le \frac{N - 1}{2}\frac{\left( \frac{M}{\delta} \right)^{N} - 1}{M^2 - \delta^2} \left( \sum_{k=1}^{N}\|D_k\|^2 \right) . 
  \end{aligned}
\end{equation}

Thus 

\begin{equation}
  \sum_{j=1}^{N}\|D_j\|^2 \ge \frac{2}{N - 1} \frac{M^2 - \delta^2}{\left( \frac{M}{\delta} \right)^{2\lfloor N/2 \rfloor} - 1} \sum_{i=1}^{N-1} \|\Delta_{i,i+1}\|_F^2 . 
\end{equation}

Combine with Lemma \ref{regularization, total}, then the proof is done. 

\end{proof}

\begin{remark}

For $N=4$, Theorem \ref{regularization term, convergence bound} reduces to

\begin{equation}
  \frac{\mathrm{d}}{\mathrm{d} t} \left(\sum_{j=1}^{3} \| \Delta_{j,j+1} \|_F^2 \right) \le  - \frac{8a}{3} \frac{\delta^4}{M^2 + \delta^{2}} \cdot \left(\sum_{j=1}^{3} \| \Delta_{j,j+1} \|_F^2 \right) . 
\end{equation}

\end{remark}

\begin{theorem}\label{maximum and minimum singular values are irrelevant of the regularization term} Under problem settings in section \ref{section: problem formulation} with gradient flow, the change of maximum and minimum singular values of $W_j$s have bounds that are irrelevant of the regularization term: 

\begin{equation}
  \begin{aligned}
    \frac{\mathrm{d} \max_{j,k} \sigma_k^2(W_j)}{\mathrm{d}t} &\le 2 \max_{j,k} \left| \sigma_k(W_j) \right| \max_{j} \left\| \nabla_{W_j} \mathcal{L}_{\rm ori} \right\|_{op} \\
    \frac{\mathrm{d} \min_{j,k} \sigma_k^2(W_j)}{\mathrm{d}t} &\ge - 2 \min_{j,k} \left| \sigma_k(W_j) \right| \max_{j} \left\| \nabla_{W_j} \mathcal{L}_{\rm ori} \right\|_{op} . 
  \end{aligned}
\end{equation}
  
\end{theorem}

\begin{remark}

If $\arg\max_{(j,k)} |\sigma_k(W_j)|$, $\arg\min_{(j,k)} |\sigma_k(W_j)|$ are not unique, the derivatives are not well-defined. In these cases, the inequalities become: 

\begin{equation}
  \begin{aligned}
    \frac{\mathrm{d} \sigma_{k^\prime}^2(W_{j^\prime})}{\mathrm{d}t} &\le 2 \max_{j,k} \left| \sigma_k(W_j) \right| \max_{j} \left\| \nabla_{W_j} \mathcal{L}_{\rm ori} \right\|_{op},\, (j^\prime,k^\prime)\in \arg\max_{(j,k)} |\sigma_k(W_j)| \\
    \frac{\mathrm{d} \sigma_{k^\prime}^2(W_{j^\prime})}{\mathrm{d}t} &\ge - 2 \min_{j,k} \left| \sigma_k(W_j) \right| \max_{j} \left\| \nabla_{W_j} \mathcal{L}_{\rm ori} \right\|_{op} ,\, (j^\prime,k^\prime)\in \arg\min_{(j,k)} |\sigma_k(W_j)| . 
  \end{aligned}
\end{equation}

\end{remark}

\begin{proof}

For simplicity, set $W_0 \equiv W_1$, $W_5 \equiv W_4$. 

Denote the analytic singular value decomposition of $W_j(t)$ to be $U^{(j)} \Sigma_w^{(j)} V^{(j)H}$, then from Lemma \ref{singular value derivative}, we have

\begin{equation}
 \begin{aligned}
  \frac{\mathrm{d} \sigma_k(W_j)}{\mathrm{d}t} &= \Re \left( u_k^{(j)H} \left( - \nabla_{W_j} \mathcal{L}_{\rm ori} + a W_j \Delta_{j-1,j} - a \Delta_{j,j+1} W_j \right) v_k^{(j)} \right) \\ 
  &= \Re \left( u_k^{(j)H} \left( - \nabla_{W_j} \mathcal{L}_{\rm ori} \right) v_k^{(j)}\right) \\ &+ a u_k^{(j)H} \left( W_j W_{j-1} W_{j-1}^H + W_{j+1}^H W_{j+1} W_j - 2 W_j W_j^H W_{j} \right) v_k^{(j)} \\
  &= \Re \left( u_k^{(j)H} \left( - \nabla_{W_j} \mathcal{L}_{\rm ori} \right) v_k^{(j)} \right) \\&+ a \left[ \left( u_k^{(j)H} W_{j+1}^H W_{j+1} u_k^{(j)} + v_k^{(j)H} W_{j-1} W_{j-1}^H v_k^{(j)} \right) \sigma_k(W_j)  - 2 \sigma_k(W_j)^3 \right] . 
 \end{aligned}
\end{equation}

From $ u_k^{(j)H} W_{j+1}^H W_{j+1} u_k^{(j)} ,\, v_k^{(j)H} W_{j-1} W_{j-1}^H v_k^{(j)} \in [\min_{j,k} \sigma_{k}^2(W_j), \max_{j,k} \sigma_{k}^2(W_j)] $, the proof is completed. 

\end{proof}

Note: 

\begin{equation}
  \max_{j}\left\| \nabla_{W_j} \mathcal{L}_{\rm ori} \right\|_{op} \le \max_{j,k} |\sigma_k(W_j)|^{N-1} \left( \sigma_1(\Sigma) + \max_{j,k} |\sigma_k(W_j)|^{N} \right) . 
\end{equation}

\subsection{Lemmas on Eigenvalue Change under Discrete Time}\label{subsection: Lemmas on Eigenvalue Change under Discrete Time}

\begin{lemma}\label{maximum and minimum singular values, general, discrete}

Suppose $\Sigma,S\in\mathbb{F}^{d\times d}$ are positive semi-definite matrices, $0 \le \alpha \le \frac{1}{6}\|S\|_{op}^{-1}$, $\mathbb{F} = \mathbb{C}$ or $\mathbb{R}$. Consider $S^\prime = (I + \alpha(\Sigma - S))S(I + \alpha(\Sigma - S))$, 

\begin{equation}
 \begin{aligned}
  \lambda_{\min}\left(S^\prime \right) &\ge \lambda_{\min}(S)(1 + \alpha (\lambda_{\min}(\Sigma) - \lambda_{\min}(S)))^2 + O\left(\alpha^2 \left( \| \Sigma \|_{op}^2 + \| S \|_{op}^2 \right) \| S \|_{op} \right) \\
  \lambda_{\max}\left(S^\prime \right) &\le \lambda_{\max}(S)(1 + \alpha (\lambda_{\max}(\Sigma) - \lambda_{\max}(S)))^2 . 
 \end{aligned}
\end{equation}
    
\end{lemma}

This generalizes Lemma 3.2 in \cite{ye2021globalconvergencegradientdescent}. 

\begin{proof}

Following the derivations in \cite{ye2021globalconvergencegradientdescent}, $\forall \beta \in (0,1)$, rewrite the terms by the following:

\begin{equation}
 \begin{aligned}
  S^\prime &= \beta \left( I - \frac{\alpha}{\beta} S \right) S \left( I - \frac{\alpha}{\beta} S \right) + (1-\beta) \left( I + \frac{\alpha}{1-\beta}\Sigma \right) S \left( I + \frac{\alpha}{1-\beta}\Sigma \right) \\ &- \frac{\alpha^2}{\beta(1-\beta)} \left[ (1-\beta)S + \beta \Sigma \right] S \left[ (1-\beta)S + \beta \Sigma \right] . 
 \end{aligned}
\end{equation}

The first term has eigenvalues $\lambda_{i^\prime}(S^\prime) = \beta \left( 1 - \frac{\alpha}{\beta} \lambda_i(S) \right)^2 \lambda_i(S)$ (note that $f(x) = (1-x)^2 x$ is non-decreasing in $\left[0,\frac{1}{3} \right]$, so $\lambda_{i^\prime}(S^\prime)$ is exactly the $i^{th}$ eigenvalue of the first term when $\beta \ge \frac{1}{2}$), while the second term is bounded by 

\begin{equation}
 \begin{aligned}
   (1-\beta) \left( I + \frac{\alpha}{1-\beta} \lambda_{\min}(\Sigma) \right)^2 \lambda_{\min}(S) \preceq {\rm term 2} \preceq (1-\beta) \left( I + \frac{\alpha}{1-\beta} \lambda_{\max}(\Sigma) \right)^2 \lambda_{\max}(S) . 
 \end{aligned}
\end{equation}

By treating the third term as error term and taking $\beta = \frac{1}{2}$, the proof is completed. 

\end{proof}

\begin{lemma}\label{minimum singular values lower bound, general, discrete}

Suppose $D, S\in\mathbb{F}^{d\times d}$ are positive semi-definite matrices, $E \in \mathbb{F}^{d\times d}$, $\mathbb{F} = \mathbb{C}$ or $\mathbb{R}$. Denote $M = S + D$. Consider $S^\prime = \left(I + \eta\left(aM - M^3 + E\right)\right)S\left(I + \eta\left(aM - M^3 + E\right)\right)$, under $\eta < \frac{1}{16 \left( \|M\|_{op}^3 + \|E\|_{op} \right)}$, 

\begin{equation}
 \begin{aligned}
  \lambda_{\min}(S^\prime) & \ge \lambda_{\min}(S) + 2\eta \left(a  - 2 \|D\|_{op} \|M\|_{op} - \|M\|_{op} \lambda_{\min}(S) \right) \lambda_{\min}^2(S) \\
  &- 2\eta \left( \|E\|_{op} + \|D\|_{op}^2 \|M\|_{op} \right) \lambda_{\min}(S) \\
  &+ O\left( \left( a^2 \|M\|_{op}^2 + \|M\|_{op}^6 + \|E\|_{op}^2 \right) \|S\|_{op} \right) . 
 \end{aligned}
\end{equation}

\end{lemma}

\begin{proof}

Expand the expression of $S^\prime$: 

\begin{equation}
 \begin{aligned}
  S^\prime &= S + \eta \left( aM + E - D M D \right) S + \eta S \left( aM + E - D M D \right) \\ 
  &- \eta (D M S^2 + S^2 M D) - \eta S (M D + D M ) S- \eta (S M S^2 + S^2 M S) + \eta^2 M_{\rm error}^\prime \\ 
  &= \frac{1}{4}\left( I + 4\eta \left( aM + E - D M D \right) \right) S \left( I + 4\eta \left( aM + E - D M D \right) \right) \\ 
  &+ \frac{1}{4s} \left( I - 4 \eta s D M \right) S^2 \left( I - 4 \eta s M D \right) + \frac{1}{4s} S \left( I - 4 \eta s \left(M D + D M\right) \right) S \\
  &+ \frac{1}{4s^2} S \left( I - 4\eta s^2 M \right) S \left( I - 4\eta s^2 M \right) S + \left( \frac{3}{4} S - \frac{1}{2 s} S^2 - \frac{1}{4 s^2} S^3 \right) + \eta^2 M_{\rm error}^\prime . 
 \end{aligned}
\end{equation}

where $\|M_{\rm error}^\prime\|_{op} = O\left( \left( a^2 \|M\|_{op}^2 + \|M\|_{op}^6 + \|E\|_{op}^2 \right) \|S\|_{op} \right)$. 

Notice that $\frac{3}{4} S - \frac{1}{2 s} S^2 - \frac{1}{4 s^2} S^3$ has eigenvalues $\lambda_{i^\prime}(S^\prime) = \frac{3}{4} \lambda_i(S) - \frac{1}{2 s} \lambda_i^2(S) - \frac{1}{4 s^2} \lambda_i^3(S)$, so by taking $s = 2\|S\|_{op}$, $\lambda_{i^\prime}(S^\prime)$ is exactly the $i^{th}$ eigenvalue of $S^\prime$. 

This further gives 

\begin{equation}
 \begin{aligned}
  \lambda_{\min}(S^\prime) &\ge \frac{1}{4}\left( 1 + 4\eta \left( a \lambda_{\min}(M) - \|E\|_{op} - \|D\|_{op}^2 \|M\|_{op} \right) \right)^2 \lambda_{\min}(S) \\ 
  &+ \frac{1}{4s} \left( 1 - 4 \eta s \|D\|_{op} \|M\|_{op} \right)^2 \lambda_{\min}^2(S) + \frac{1}{4s} \left( 1 - 8 \eta s \|M\|_{op} \|D\|_{op} \right) \lambda_{\min}^2(S) \\
  &+ \frac{1}{4s^2} \left( 1 - 4\eta s^2 \|M\|_{op} \right)^2 \lambda_{\min}^3(S) + \left( \frac{3}{4} \lambda_{\min}(S) - \frac{1}{2 s} \lambda_{\min}^2(S) - \frac{1}{4 s^2} \lambda_{\min}^3(S) \right) \\
  &+ \eta^2 \left\| M_{\rm error}^\prime \right\|_{op} \\ 
  &\ge \lambda_{\min}(S) + 2\eta \left(a \lambda_{\min}(M) - 2 \|D\|_{op} \|M\|_{op} \lambda_{\min}(S) - \|M\|_{op} \lambda_{\min}^2(S) \right) \lambda_{\min}(S) \\
  &- 2\eta \left( \|E\|_{op} + \|D\|_{op}^2 \|M\|_{op} \right) \lambda_{\min}(S) + \eta^2 \left\| M_{\rm error}^\prime \right\|_{op} . 
 \end{aligned}
\end{equation}

From $\lambda_{\min}(M) \ge \lambda_{\min}(S)$, the proof is completed. 

\end{proof}

\subsection{Lemmas on Regularization, Gradient Descent}

\begin{theorem}\label{regularization term, convergence bound, GD}

Suppose for all $j \in [1,4]\cap \mathbb{N}^*$, $\sigma_{\min}(W_j(t)) \ge \delta$, $\sigma_{\max}(W_j(t)) \le M$, then the convergence rate of the regularization term is lower bounded by: 

\begin{equation}
 \begin{aligned}
  \mathcal{L}_{\rm reg}(t+1) &\le \left(1 - \frac{8}{3} \frac{\eta a \delta^4}{M^2 + \delta^2}\right) \cdot \mathcal{L}_{\rm reg}(t) \\&+ \eta^2 O\left(a^2 M^4 \mathcal{L}_{\rm reg}(t) + \sqrt{a \mathcal{L}_{\rm reg}(t)} M^6 \mathcal{L}_{\rm ori}(t) \right) \\ &+ \eta^4 O\left(a M^{12} \mathcal{L}_{\rm ori}(t)^2 + a^3 M^4 \mathcal{L}_{\rm reg}(t)^2 \right) . 
 \end{aligned}
\end{equation}

\end{theorem}

\begin{proof}

\begin{equation}
  \begin{aligned}
    \Delta_{j,j+1}(t+1) - \Delta_{j,j+1}(t) &= 2\eta a W_j(t) \Delta_{j-1,j}(t) W_j(t)^H \\&+ 2\eta a W_{j+1}(t)^H \Delta_{j+1,j+2}(t) W_{j+1}(t) \\ &- \eta a \Delta_{j,j+1}(t) \left( W_j(t) W_j(t)^H + W_{j+1}(t)^H W_{j+1}(t) \right) \\ &- \eta a \left( W_j(t) W_j(t)^H + W_{j+1}(t)^H W_{j+1}(t) \right) \Delta_{j,j+1}(t) \\ 
    &+ \eta^2 \left[\nabla_{W_j} \mathcal{L}(t) \nabla_{W_j} \mathcal{L}(t)^H - \nabla_{W_{j+1}} \mathcal{L}(t)^H \nabla_{W_{j+1}} \mathcal{L}(t) \right] . 
  \end{aligned}
\end{equation}

From 

\begin{equation}
 \begin{aligned}
  \left\| \nabla_{W_j} \mathcal{L}(t) \right\|_F &\le \left\| \nabla_{W_j} \mathcal{L}_{\rm ori}(t) \right\|_F + \left\| \nabla_{W_j} \mathcal{L}_{\rm reg}(t) \right\|_F \\ &= O\left( M^3 \sqrt{\mathcal{L}_{\rm ori}(t)} + M \sqrt{a \mathcal{L}_{\rm reg}(t)} \right) \\ 
  \left\| \Delta_{j,j+1}(t+1) - \Delta_{j,j+1}(t) \right\|_F &= O\left( \eta M^2 \sqrt{a \mathcal{L}_{\rm reg}(t)} + \eta^2 \left\| \nabla_{W_j} \mathcal{L}(t) \right\|_F^2 \right) \\
  &= O\left( \eta M^2 \sqrt{a \mathcal{L}_{\rm reg}(t)} + \eta^2 M^6 \mathcal{L}_{\rm ori}(t) + \eta^2 a M^2 \mathcal{L}_{\rm reg}(t) \right) . 
 \end{aligned}
\end{equation}

We have 

\begin{equation}
 \begin{aligned}
  \mathcal{L}_{\rm reg}(t+1) - \mathcal{L}_{\rm reg}(t) &= 2a \sum_{j=1}^{3}\left\langle \Delta_{j,j+1}(t+1) - \Delta_{j,j+1}(t), \Delta_{j,j+1}(t) \right\rangle \\ &+ a\sum_{j=1}^{3}\left\| \Delta_{j,j+1}(t+1) - \Delta_{j,j+1}(t) \right\|_F^2 \\
  &= -4\eta a^2 \sum_{j=1}^{4} \left\| \Delta_{j,j+1}(t) W_j(t) - W_j(t) \Delta_{j-1,j}(t) \right\|_F^2 \\ &+ O\left(\eta^2\sqrt{a \mathcal{L}_{\rm reg}(t)} \left(a M^2 \mathcal{L}_{\rm reg}(t) + M^6 \mathcal{L}_{\rm ori}(t) \right)\right) \\ &+ O\left( \eta^2 a^2 M^4 \mathcal{L}_{\rm reg}(t) + \eta^4 a M^{12} \mathcal{L}_{\rm ori}(t)^2 + \eta^4 a^3 M^4 \mathcal{L}_{\rm reg}(t)^2 \right) \\ 
  &= -4\eta a^2 \sum_{j=1}^{4} \left\| \Delta_{j,j+1}(t) W_j(t) - W_j(t) \Delta_{j-1,j}(t) \right\|_F^2 \\ &+ \eta^2 O\left(a^2 M^4 \mathcal{L}_{\rm reg}(t) + \sqrt{a \mathcal{L}_{\rm reg}(t)} M^6 \mathcal{L}_{\rm ori}(t) \right) \\ &+ \eta^4 O\left(a M^{12} \mathcal{L}_{\rm ori}(t)^2 + a^3 M^4 \mathcal{L}_{\rm reg}(t)^2 \right) . 
 \end{aligned}
\end{equation}

Follow previous analysis in continuous case, 

\begin{equation}
  \sum_{j=1}^{4}\| \Delta_{j,j+1}(t) W_j(t) - W_j(t) \Delta_{j-1,j}(t) \|^2 \ge \frac{2}{3} \frac{\delta^4}{M^2 + \delta^2} \sum_{i=1}^{3} \|\Delta_{i,i+1}(t)\|_F^2 . 
\end{equation}

Then the proof is done. 

\end{proof}

\begin{theorem}\label{maximum and minimum singular values are irrelevant of the regularization term, GD} The maximum and minimum singular values of $W_j$s are irrelevant of the regularization term. 

Under $\eta \le \min\left(\frac{1}{18 a \max_{j,k} \sigma_k^2(W_j(t))}, \frac{\min_{j,k}\sigma_{k}(W_j(t))}{3 \max_{j}\left\| \nabla_{W_j} \mathcal{L}_{\rm ori}(t) \right\|_{op}}\right)$, 

\begin{equation}
 \begin{aligned}
  \max_{j,k} \sigma_k^2(W_j(t+1)) - \max_{j,k} \sigma_k^2(W_j(t)) &\le 2\eta \max_{j,k} \sigma_k(W_j(t)) \max_{j}\left\| \nabla_{W_j} \mathcal{L}_{\rm ori}(t) \right\|_{op} \\&+ \eta^2 O\left( \left\| \nabla_{W_j} \mathcal{L}_{\rm ori}(t) \right\|_{op}^2 + a^2 \max_{j,k}\sigma_k^6(W_j(t)) \right) \\
  \min_{j,k} \sigma_k^2(W_j(t+1)) - \min_{j,k} \sigma_k^2(W_j(t)) &\ge - 2\eta \min_{j,k} \sigma_k(W_j(t)) \max_{j}\left\| \nabla_{W_j} \mathcal{L}_{\rm ori}(t) \right\|_{op} \\&+ \eta^2 O\left( \left\| \nabla_{W_j} \mathcal{L}_{\rm ori}(t) \right\|_{op}^2 + a^2 \max_{j,k}\sigma_k^6(W_j(t)) \right) . 
 \end{aligned}
\end{equation}
  
\end{theorem}

\begin{proof}

For simplicity, set $W_0 \equiv W_1$, $W_5 \equiv W_4$. 

Generally, 

\begin{equation}
 \begin{aligned}
  W_j(t+1) W_j(t+1)^H &= W_j(t) W_j(t)^H - \eta W_j(t) \nabla_{W_j} \mathcal{L}(t)^H - \eta \nabla_{W_j} \mathcal{L}(t) W_j(t)^H \\ &+ \eta^2 \nabla_{W_j} \mathcal{L}(t) \nabla_{W_j} \mathcal{L}(t)^H \\ 
  &= W_j(t) W_j(t)^H - \eta W_j(t) \nabla_{W_j} \mathcal{L}_{\rm ori}(t)^H - \eta \nabla_{W_j} \mathcal{L}_{\rm ori}(t) W_j(t)^H \\&+ 2\eta a W_j(t) \Delta_{j-1,j}(t) W_j(t)^H - \eta a W_j(t) W_j(t)^H \Delta_{j,j+1}(t) \\ & - \eta a \Delta_{j,j+1}(t) W_j(t) W_j(t)^H + \eta^2 \nabla_{W_j} \mathcal{L}(t) \nabla_{W_j} \mathcal{L}(t)^H \\ 
  &= \frac{1}{3} W_j(t) \left( I + 3 \eta a \Delta_{j-1,j}(t) \right)^2 W_j(t)^H \\ &+ \frac{1}{3} \left( I - 3 \eta a \Delta_{j,j+1}(t) \right) W_j(t) W_j(t)^H \left( I - 3 \eta a \Delta_{j,j+1}(t) \right) \\ &+ \frac{1}{3}\left(W_j(t) - 3 \eta \nabla_{W_j} \mathcal{L}_{\rm ori}(t) \right) \left(W_j(t) - 3 \eta \nabla_{W_j} \mathcal{L}_{\rm ori}(t) \right)^H \\ &+ \eta^2 \nabla_{W_j} \mathcal{L}(t) \nabla_{W_j} \mathcal{L}(t)^H - 3 \eta^2 \nabla_{W_j} \mathcal{L}_{\rm ori}(t) \nabla_{W_j} \mathcal{L}_{\rm ori}(t)^H \\ &- 3 \eta^2 a^2 W_j(t) \Delta_{j-1,j}(t)^2 W_j(t)^H \\&- 3 \eta^2 a^2 \Delta_{j,j+1}(t) W_j(t) W_j(t)^H \Delta_{j,j+1}(t) . 
 \end{aligned}
\end{equation}

Notice that $W_j(t) \left( I + 3 \eta a \Delta_{j-1,j}(t) \right)^2 W_j(t)^H$ and $\left( I + 3 \eta a \Delta_{j-1,j}(t) \right) W_j(t)^H W_j(t) \left( I + 3 \eta a \Delta_{j-1,j}(t) \right)$ shares the same eigenvalues. 
Then from Lemma \ref{maximum and minimum singular values, general, discrete}, the maximum and minimum singular values of $W_j(t+1)$ satisfy 

\begin{equation}
 \begin{aligned}
  \sigma_{\max}^2(W_j(t+1)) &\le \frac{1}{3} \sigma_{\max}^2(W_j(t)) \left[1 + 3\eta a \left(\sigma_{\max}^2(W_{j-1}(t)) - \sigma_{\max}^2(W_j(t)) \right) \right]^2 \\ 
  &+ \frac{1}{3} \sigma_{\max}^2(W_j(t)) \left[1 + 3\eta a \left(\sigma_{\max}^2(W_{j+1}(t)) - \sigma_{\max}^2(W_j(t)) \right) \right]^2 \\ 
  &+ \frac{1}{3} \left[ \sigma_{\max}(W_j(t)) + 3\eta \left\| \nabla_{W_j} \mathcal{L}_{\rm ori}(t) \right\|_{op} \right]^2 \\ &+ \eta^2 O\left( \left\| \nabla_{W_j} \mathcal{L}_{\rm ori}(t) \right\|_{op}^2 + a^2 \max_{j,k}\sigma_k^6(W_j(t)) \right) \\ 
  &= \sigma_{\max}^2(W_j(t))\left[1 + 3 \eta a \left(\sigma_{\max}^2(W_{j+1}(t)) + \sigma_{\max}^2(W_{j-1}(t)) - 2 \sigma_{\max}^2(W_j(t)) \right)\right] \\
  &+ 2\eta \sigma_{\max}(W_j(t))\left\| \nabla_{W_j} \mathcal{L}_{\rm ori}(t) \right\|_{op} + \eta^2 O\left( \left\| \nabla_{W_j} \mathcal{L}_{\rm ori}(t) \right\|_{op}^2 + a^2 \max_{j,k}\sigma_k^6(W_j(t)) \right) \\
  \sigma_{\min}^2(W_j(t+1)) &\ge \frac{1}{3} \sigma_{\min}^2(W_j(t)) \left[1 + 3\eta a \left(\sigma_{\min}^2(W_{j-1}(t)) - \sigma_{\min}^2(W_j(t)) \right) \right]^2 \\ 
  &+ \frac{1}{3} \sigma_{\min}^2(W_j(t)) \left[1 + 3\eta a \left(\sigma_{\min}^2(W_{j+1}(t)) - \sigma_{\min}^2(W_j(t)) \right) \right]^2 \\ 
  &+ \frac{1}{3} \left[ \sigma_{\min}(W_j(t)) - 3\eta \left\| \nabla_{W_j} \mathcal{L}_{\rm ori}(t) \right\|_{op} \right]^2 \\ &+ \eta^2 O\left( \left\| \nabla_{W_j} \mathcal{L}_{\rm ori}(t) \right\|_{op}^2 + a^2 \max_{j,k} \sigma_k^6(W_j(t)) \right) \\
  &= \sigma_{\min}^2(W_j(t))\left[1 + 3 \eta a \left(\sigma_{\min}^2(W_{j+1}(t)) + \sigma_{\min}^2(W_{j-1}(t)) - 2 \sigma_{\min}^2(W_j(t)) \right)\right] \\
  &- 2\eta \sigma_{\min}(W_j(t))\left\| \nabla_{W_j} \mathcal{L}_{\rm ori}(t) \right\|_{op} + \eta^2 O\left( \left\| \nabla_{W_j} \mathcal{L}_{\rm ori}(t) \right\|_{op}^2 + a^2 \max_{j,k}\sigma_k^6(W_j(t)) \right) . 
 \end{aligned}
\end{equation}

By taking maximum and minimum over $j\in[1,4]\cap \mathbb{N}^*$ (for $\eta \le \frac{1}{6 a \max_{j,k}\sigma_k^2(W_j(t))}$, the first term of R.H.S can be upper bounded by $\max_{j,k}\sigma_k^2(W_j(t))$ or lower bounded by $\min_{j,k}\sigma_k^2(W_j(t))$ respectively), the proof is completed.

\end{proof}

%% file: subfiles/appendix/appendix_4_balanced.tex
\section{Dynamics under Balanced initialization}

This section analyzes the training dynamics under balanced initialization. 

At the beginning, We derive some properties from Lemma \ref{ASVD, non-negative diagonal}. Under balanced condition, 

\begin{equation}
 \begin{aligned}
  W_{\prod_L , j} W_{\prod_L , j}^H &= \left( \prod_{k=N}^{j} W_k \right)\left( \prod_{k=N}^{j} W_k \right)^H = \left(W_N W_N^H\right)^{N-j+1} \\
  W_{\prod_R , j}^H W_{\prod_R , j} &= \left( \prod_{k=j}^{1} W_k \right)^H \left( \prod_{k=j}^{1} W_k \right) = \left(W_1^H W_1\right)^{N-j+1} .
 \end{aligned}
\end{equation}

Consider $j=1$ and $j=N$, then 

\begin{equation}
 \begin{aligned}
  W_N W_N^H &= \left(W W^H\right)^{1/N} = U \Sigma_w^2 U^H \\
  W_1 W_1^H&= \left(W^H W\right)^{1/N} = V \Sigma_w^2 V^H . 
 \end{aligned}
\end{equation}

Suppose the non-negative ASVD of product matrix is $W = U \Sigma_w^N V^H$, then

\begin{equation}
 \begin{aligned}
  \frac{\mathrm{d}}{\mathrm{d} t} \left( U \Sigma_w^2 U^H \right) &= \frac{\mathrm{d}}{\mathrm{d} t} \left( W_N W_N^H \right) = \Sigma V \Sigma_w^N U^H + U \Sigma_w^N V^H \Sigma^H - 2 U \Sigma_w^{2N} U^H \\ 
  \frac{\mathrm{d}}{\mathrm{d} t} \left( V \Sigma_w^2 V^H \right) &= \frac{\mathrm{d}}{\mathrm{d} t} \left( W_1^H W_1 \right) = V \Sigma_w^N U^H \Sigma + \Sigma^H U \Sigma_w^N V^H - 2 V \Sigma_w^{2N} V^H \\ 
  \frac{\mathrm{d}W}{\mathrm{d} t} &= \sum_{j=1}^{N} U \Sigma_w^{2(j-1)} U^H  \Sigma  V \Sigma_w^{2(N-j)} V^H - N U \Sigma_w^{3N-2} V^H . 
 \end{aligned}
\end{equation}

The dynamics of $\sigma_r \coloneqq \sigma_{w,r}^N$ is presented in (\ref{thm 3, arora, complex}). 

\subsection{Skew-Hermitian Error}

This section formally state and prove Theorem \ref{antisym-error, general, informal}. 

\begin{theorem}\label{antisym-error, general balanced}

The skew-hermitian error is non-increasing. 

Under balanced Gaussian initialization, for $\mathbb{F} = \mathbb{C}$ or $\mathbb{R}$, suppose the ASVD of the product matrix is $W(t) = U(t) \Sigma_w(t)^N V(t)^H$, furthermore assume that the singular values of the product matrix at initialization ($W(0)$) are distinct and different from zero (refer to Lemma 2 in \cite{arora2019implicitregularizationdeepmatrix}). 

Denote $\sigma_{w,j}=(\Sigma_w)_{jj}$, $U^\prime = \Sigma^{1/2} U$, $V^\prime = \Sigma^{1/2} V$, $u_j^\prime$ and $v_j^\prime$ are the $j^{th}$ columns of $U^\prime$ and $V^\prime$ respectively, then 

\begin{equation}\label{antisym-error, general balanced, eq}
 \begin{aligned}
  \frac{\mathrm{d}}{\mathrm{d}t} \| \Sigma^{1/2} (U-V) \Sigma_w\|_{F}^2 &= -2 \sum_j \sigma_{w,j}^N \cdot \left\| \Sigma^{1/2} \left( {u_j^\prime} - {v_j^\prime} \right) \right\|^2 - 2\sum_j \sigma_{w,j}^{2N} \cdot \left\| u_j^\prime - v_j^\prime \right\|^2 \\
  &- \sum_{j,k} f_N \left( \sigma_{w,j} , \sigma_{w,k} \right) \left| {u_j^\prime}^H {v_k^\prime} - {v_j^\prime}^H {u_k^\prime} \right|^2 \\ 
  &\le 0 , 
 \end{aligned}
\end{equation}

where $f_N(x,y) = \begin{cases}
  \frac{x^2 y^2 (x^{N-2} - y^{N-2})}{x^2 - y^2} &,\, y\ne x \\
  \frac{N-2}{2} x^{N} &,\, y = x
\end{cases}$ is a non-negative real-analytic function on $[0,+\infty)^2$. 

\end{theorem}

\begin{proof}

By (\ref{thm 3, arora, complex})

\begin{equation}
 \begin{aligned}
  \frac{\mathrm{d} \sigma_{w,j}}{\mathrm{d}t} = \sigma_{w,j}^{N-1} \left( \frac{\langle u_j^\prime, v_j^\prime \rangle + \langle v_j^\prime, u_j^\prime \rangle}{2} - \sigma_{w,j}^{N} \right) . 
 \end{aligned}
\end{equation}

From Lemma \ref{lemma 2, arora, complex}, 

\begin{equation}
 \begin{aligned}
  \frac{\mathrm{d} U}{\mathrm{d} t} = U \left( F \odot M_U + D_U\right),\, \frac{\mathrm{d} V}{\mathrm{d} t} = V \left( F \odot M_V + D_V \right) , 
 \end{aligned}
\end{equation}

where 

\begin{equation}
 \begin{aligned}
  \begin{cases}
   (M_U)_{jk} &= \left\langle v_k^\prime, u_j^\prime\right\rangle \sigma_{w,k}^N + \left\langle u_k^\prime, v_j^\prime \right\rangle \sigma_{w,j}^N - 2\sigma_{w,j}^{2N}\delta_{jk} \\ 
   (M_V)_{jk} &= \left\langle u_k^\prime, v_j^\prime \right\rangle \sigma_{w,k}^N + \left\langle v_k^\prime, u_j^\prime \right\rangle \sigma_{w,j}^N - 2\sigma_{w,j}^{2N}\delta_{jk} \\ 
  \end{cases},
 \end{aligned}
\end{equation}

$D_{U,V}$ are pure imaginary diagonal matrices defined by  

\begin{equation}
  (D_U)_{jj} - (D_V)_{jj} = \frac{N}{2} \sigma_{w,j}^{N-2} \left[\left\langle v_j^\prime, u_j^\prime \right\rangle - \left\langle u_j^\prime, v_j^\prime \right\rangle \right],\, \Re(D_U)=\Re(D_V)=O. 
\end{equation}

Here $\langle a, b\rangle \coloneqq b^H a$ follows the standard definition of (complex) inner product. Then 

\begin{equation}
 \begin{aligned}
  \frac{\mathrm{d} U^{\prime H} V^\prime}{\mathrm{d}t} &= \frac{\mathrm{d} U}{\mathrm{d}t}^H \Sigma V + U^H \Sigma \frac{\mathrm{d} V}{\mathrm{d}t} = \left( F^H \odot M_U^H - D_U\right) U^H \Sigma V + U^H \Sigma V (F \odot M_V + D_V) \\ 
  \frac{\mathrm{d} U^{\prime H} U^\prime}{\mathrm{d}t} &= \frac{\mathrm{d} U}{\mathrm{d}t}^H \Sigma U + U^H \Sigma \frac{\mathrm{d} U}{\mathrm{d}t} = \left( F^H \odot M_U^H - D_U\right) U^H \Sigma U + U^H \Sigma U (F \odot M_U + D_U) \\ 
  \frac{\mathrm{d} V^{\prime H} V^\prime}{\mathrm{d}t} &= \frac{\mathrm{d} V}{\mathrm{d}t}^H \Sigma V + V^H \Sigma \frac{\mathrm{d} V}{\mathrm{d}t} = \left( F^H \odot M_V^H - D_V\right) V^H \Sigma V + V^H \Sigma V (F \odot M_V + D_V) . 
 \end{aligned}
\end{equation}

For each diagonal entry, 

\begin{equation}
 \begin{aligned}
  &\frac{\mathrm{d}}{\mathrm{d}t} \left\langle v_j^\prime, u_j^\prime \right\rangle = \left( \frac{\mathrm{d} U^{\prime H} V^\prime}{\mathrm{d}t} \right)_{jj} \\ =& - \frac{N}{2} \sigma_{w,j}^{N-2} \left\langle v_j^\prime, u_j^\prime \right\rangle \left[\left\langle v_j^\prime, u_j^\prime \right\rangle - \left\langle u_j^\prime, v_j^\prime \right\rangle \right] \\+& \sum_{k \ne j} \frac{1}{\sigma_{w,j}^2 - \sigma_{w,k}^2} \left[\left( \left| \left\langle u_j^\prime, v_k^\prime \right\rangle \right|^2 + \left| \left\langle u_k^\prime, v_j^\prime \right\rangle \right|^2 \right)\sigma_{w,j}^N + 2 \left\langle v_k^\prime, u_j^\prime \right\rangle \left\langle v_j^\prime, u_k^\prime \right\rangle \sigma_{w,k}^N \right] \\ 
  &\frac{\mathrm{d}}{\mathrm{d}t} \left\langle u_j^\prime, u_j^\prime \right\rangle = \left( \frac{\mathrm{d} U^{\prime H} U^\prime}{\mathrm{d}t} \right)_{jj} \\ =& \sum_{k \ne j} \frac{1}{\sigma_{w,j}^2 - \sigma_{w,k}^2} \left[ \left( \left\langle u_k^\prime, v_j^\prime \right\rangle \left\langle u_j^\prime, u_k^\prime \right\rangle + \left\langle u_k^\prime, u_j^\prime \right\rangle \left\langle v_j^\prime, u_k^\prime \right\rangle \right) \sigma_{w,j}^N \right. \\+& \left. \left( \left\langle v_k^\prime, u_j^\prime \right\rangle \left\langle u_j^\prime, u_k^\prime \right\rangle + \left\langle u_k^\prime, u_j^\prime \right\rangle \left\langle u_j^\prime, v_k^\prime \right\rangle \right) \sigma_{w,k}^N \right] \\ 
  &\frac{\mathrm{d}}{\mathrm{d}t} \left\langle v_j^\prime, v_j^\prime \right\rangle = \left( \frac{\mathrm{d} V^{\prime H} V^\prime}{\mathrm{d}t} \right)_{jj} \\ =& \sum_{k \ne j} \frac{1}{\sigma_{w,j}^2 - \sigma_{w,k}^2} \left[ \left( \left\langle v_k^\prime, u_j^\prime \right\rangle \left\langle v_j^\prime, v_k^\prime \right\rangle + \left\langle v_k^\prime, v_j^\prime \right\rangle \left\langle u_j^\prime, v_k^\prime \right\rangle \right) \sigma_{w,j}^N \right. \\+& \left. \left( \left\langle u_k^\prime, v_j^\prime \right\rangle \left\langle v_j^\prime, v_k^\prime \right\rangle + \left\langle v_k^\prime, v_j^\prime \right\rangle \left\langle v_j^\prime, u_k^\prime \right\rangle \right) \sigma_{w,k}^N \right]  . 
 \end{aligned}
\end{equation}

Notice that for the second and third equation, $D_U$, $D_V$ terms cancel out with each other. This further gives 

\begin{equation}
 \begin{aligned}
  &\frac{\mathrm{d}}{\mathrm{d}t} \left\| u_j^\prime - v_j^\prime \right\|^2 \\ =& \frac{N}{2} \sigma_{w,j}^{N-2} \left[\left\langle v_j^\prime, u_j^\prime \right\rangle - \left\langle u_j^\prime, v_j^\prime \right\rangle \right]^2 \\+&\sum_{k \ne j} \frac{\sigma_{w,j}^N}{\sigma_{w,j}^2 - \sigma_{w,k}^2} \cdot \Big[ - 2\left( \left| \left\langle u_j^\prime, v_k^\prime \right\rangle \right|^2 + \left| \left\langle u_k^\prime, v_j^\prime \right\rangle \right|^2 \right)  \\+& \left( \left\langle u_k^\prime, v_j^\prime \right\rangle \left\langle u_j^\prime, u_k^\prime \right\rangle + \left\langle u_k^\prime, u_j^\prime \right\rangle \left\langle v_j^\prime, u_k^\prime \right\rangle \right) + \left( \left\langle v_k^\prime, u_j^\prime \right\rangle \left\langle v_j^\prime, v_k^\prime \right\rangle + \left\langle v_k^\prime, v_j^\prime \right\rangle \left\langle u_j^\prime, v_k^\prime \right\rangle \right) \Big] \\ 
  +& \sum_{k \ne j} \frac{\sigma_{w,k}^N}{\sigma_{w,j}^2 - \sigma_{w,k}^2} \cdot \left[ - 2 \left( \left\langle v_k^\prime, u_j^\prime \right\rangle \left\langle v_j^\prime, u_k^\prime \right\rangle + \left\langle u_j^\prime, v_k^\prime \right\rangle \left\langle u_k^\prime, v_j^\prime \right\rangle \right) \right. \\ +& \left. \left( \left\langle v_k^\prime, u_j^\prime \right\rangle \left\langle u_j^\prime, u_k^\prime \right\rangle + \left\langle u_k^\prime, u_j^\prime \right\rangle \left\langle u_j^\prime, v_k^\prime \right\rangle \right) + \left( \left\langle u_k^\prime, v_j^\prime \right\rangle \left\langle v_j^\prime, v_k^\prime \right\rangle + \left\langle v_k^\prime, v_j^\prime \right\rangle \left\langle v_j^\prime, u_k^\prime \right\rangle \right) \right] . 
 \end{aligned}
\end{equation}

For the L.H.S. of (\ref{antisym-error, general balanced, eq}),  

\begin{equation}
 \begin{aligned}
  &\frac{\mathrm{d}}{\mathrm{d}t} \left\| (U^\prime - V^\prime) \Sigma_w \right\|_F^2 = \sum_{j} \left\| u_j^\prime - v_j^\prime \right\|^2 \frac{\mathrm{d}}{\mathrm{d}t} \sigma_{w,j}^2 + \sum_{j} \sigma_{w,j}^2 \frac{\mathrm{d}}{\mathrm{d}t} \left\| u_j^\prime - v_j^\prime \right\|^2 . 
 \end{aligned}
\end{equation}

The first term can be written by 

\begin{equation}
 \begin{aligned}
  & \sum_{j} \left\| u_j^\prime - v_j^\prime \right\|^2 \frac{\mathrm{d}}{\mathrm{d}t} \sigma_{w,j}^2 \\=& \sum_{j} \sigma_{w,j}^{N} \left( \langle u_j^\prime, v_j^\prime \rangle + \langle v_j^\prime, u_j^\prime \rangle - 2\sigma_{w,j}^{N} \right)\left\| u_j^\prime - v_j^\prime \right\|^2 \\
  =& \sum_{j} \sigma_{w,j}^{N} \left( {u_j^\prime}^H {u_j^\prime} {u_j^\prime}^H {v_j^\prime} + {v_j^\prime}^H {u_j^\prime} {u_j^\prime}^H {u_j^\prime} + {u_j^\prime}^H {v_j^\prime} {v_j^\prime}^H {v_j^\prime} + {v_j^\prime}^H {v_j^\prime} {v_j^\prime}^H {u_j^\prime} \right) \\ -& \sum_{j} \sigma_{w,j}^{N} \left({u_j^\prime}^H {v_j^\prime} + {v_j^\prime}^H {u_j^\prime}\right)^2 - 2 \sum_{j} \sigma_{w,j}^{2 N} \cdot \left\| u_j^\prime - v_j^\prime \right\|^2 . 
 \end{aligned}
\end{equation}

For the second term, 

\begin{equation}
 \begin{aligned}
  & \sum_{j} \sigma_{w,j}^2 \frac{\mathrm{d}}{\mathrm{d}t} \left\| u_j^\prime - v_j^\prime \right\|^2 \\
  =& \frac{1}{2} \left( \sum_{j} \sigma_{w,j}^2 \frac{\mathrm{d}}{\mathrm{d}t} \left\| u_j^\prime - v_j^\prime \right\|^2 + \sum_{k} \sigma_{w,k}^2 \frac{\mathrm{d}}{\mathrm{d}t} \left\| u_k^\prime - v_k^\prime \right\|^2 \right) \\ 
  =& \frac{N}{2} \sum_{j} \sigma_{w,j}^{N} \left[\left\langle v_j^\prime, u_j^\prime \right\rangle - \left\langle u_j^\prime, v_j^\prime \right\rangle \right]^2 \\
  -& \sum_{j,k,j\ne k} \frac{\sigma_{w,j}^2 \sigma_{w,k}^2 \left(\sigma_{w,j}^{N-2} - \sigma_{w,k}^{N-2}\right)}{\sigma_{w,j}^2 - \sigma_{w,k}^2} \left|\left\langle v_k^\prime, u_j^\prime \right\rangle - \left\langle u_k^\prime, v_j^\prime \right\rangle\right|^2 \\ 
  - & 2 \sum_{j,k,j\ne k} \sigma_{w,j}^{N} \cdot \left( \left| \left\langle u_j^\prime, v_k^\prime \right\rangle \right|^2 + \left| \left\langle u_k^\prime, v_j^\prime \right\rangle \right|^2 \right) \\
  + & 2 \sum_{j,k,j\ne k} \sigma_{w,j}^{N} \cdot \Re\left( \left\langle u_k^\prime, v_j^\prime \right\rangle \left\langle u_j^\prime, u_k^\prime \right\rangle + \left\langle u_j^\prime, v_k^\prime \right\rangle \left\langle v_k^\prime, v_j^\prime \right\rangle \right) . 
 \end{aligned}
\end{equation}

Notice that 

\begin{equation}
 \begin{aligned}
  \left[\left\langle v_j^\prime, u_j^\prime \right\rangle - \left\langle u_j^\prime, v_j^\prime \right\rangle \right]^2 =& 4 \left[i \Im\left(\left\langle v_j^\prime, u_j^\prime \right\rangle \right) \right]^2 = -  \left| {u_j^\prime}^H {v_j^\prime} - {v_j^\prime}^H {u_j^\prime} \right|^2 , 
 \end{aligned}
\end{equation}

\begin{equation}
 \begin{aligned}
  & - \sum_{j} \sigma_{w,j}^{N} \left({u_j^\prime}^H {v_j^\prime} + {v_j^\prime}^H {u_j^\prime}\right)^2 - 2\sum_{j,k,j\ne k} \sigma_{w,j}^{N} \left( \left| \left\langle u_j^\prime, v_k^\prime \right\rangle \right|^2 + \left| \left\langle u_k^\prime, v_j^\prime \right\rangle \right|^2 \right) \\
  =& -\sum_{j} \sigma_{w,j}^{N} \left({u_j^\prime}^H {v_j^\prime} - {v_j^\prime}^H {u_j^\prime}\right)^2 - 2 \sum_{j} \sigma_{w,j}^{N} \left({u_j^\prime}^H {v_j^\prime} {v_j^\prime}^H {u_j^\prime} + {v_j^\prime}^H {u_j^\prime} {u_j^\prime}^H {v_j^\prime}\right) \\-& 2\sum_{j} \sigma_{w,j}^{N} \cdot \left( {u_j^\prime}^H \left( \sum_{k \ne j} v_k^\prime {v_k^\prime}^H\right) {u_j^\prime} + {v_j^\prime}^H \left( \sum_{k \ne j} u_k^\prime {u_k^\prime}^H\right) {v_j^\prime} \right) \\ 
  =& -\sum_{j} \sigma_{w,j}^{N} \left({u_j^\prime}^H {v_j^\prime} - {v_j^\prime}^H {u_j^\prime}\right)^2 - 2\sum_{j} \sigma_{w,j}^{N} \cdot \left( {u_j^\prime}^H {V^\prime} {V^\prime}^H {u_j^\prime} + {v_j^\prime}^H {U^\prime}{U^\prime}^H {v_j^\prime} \right) \\ 
  =& \sum_{j} \sigma_{w,j}^{N} \left|{u_j^\prime}^H {v_j^\prime} - {v_j^\prime}^H {u_j^\prime}\right|^2 - 2\sum_{j} \sigma_{w,j}^{N} \cdot \left( {u_j^\prime}^H \Sigma {u_j^\prime} + {v_j^\prime}^H \Sigma {v_j^\prime} \right) , 
 \end{aligned}
\end{equation}

and 

\begin{equation}
 \begin{aligned}
  &\sum_{j} \sigma_{w,j}^{N} \left( {u_j^\prime}^H {u_j^\prime} {u_j^\prime}^H {v_j^\prime} + {v_j^\prime}^H {u_j^\prime} {u_j^\prime}^H {u_j^\prime} + {u_j^\prime}^H {v_j^\prime} {v_j^\prime}^H {v_j^\prime} + {v_j^\prime}^H {v_j^\prime} {v_j^\prime}^H {u_j^\prime} \right) \\ +& 2 \sum_{j,k,j\ne k} \sigma_{w,j}^{N} \cdot \Re\left( \left\langle u_k^\prime, v_j^\prime \right\rangle \left\langle u_j^\prime, u_k^\prime \right\rangle + \left\langle u_j^\prime, v_k^\prime \right\rangle \left\langle v_k^\prime, v_j^\prime \right\rangle \right) \\ 
  =& 2 \sum_{j,k} \sigma_{w,j}^{N} \cdot \Re\left( {u_j^\prime}^H \left( U{U^\prime}^H + V{V^\prime}^H \right) v_j^\prime \right) \\
  =& 2\sum_{j} \sigma_{w,j}^{N} \cdot \left( {u_j^\prime}^H \Sigma v_j^\prime + {v_j^\prime}^H \Sigma u_j^\prime \right) . 
 \end{aligned}
\end{equation}

By combining the results above, 

\begin{equation}
 \begin{aligned}
  &\frac{\mathrm{d}}{\mathrm{d}t} \left\| \left( U^\prime - V^\prime \right) \Sigma_w \right\|_{F}^2 \\
  =& - 2\sum_{j} \sigma_{w,j}^{N} \cdot \left( {u_j^\prime}^H \Sigma {u_j^\prime} + {v_j^\prime}^H \Sigma {v_j^\prime} \right) 
  + 2\sum_{j} \sigma_{w,j}^{N} \cdot \left( {u_j^\prime}^H \Sigma v_j^\prime + {v_j^\prime}^H \Sigma u_j^\prime \right) \\-& 2 \sum_{j} \sigma_{w,j}^{2 N} \cdot \left\| u_j^\prime - v_j^\prime \right\|^2 \\ 
  -& \sum_{j,k,j\ne k} \frac{\sigma_{w,j}^2 \sigma_{w,k}^2 \left(\sigma_{w,j}^{N-2} - \sigma_{w,k}^{N-2}\right)}{\sigma_{w,j}^2 - \sigma_{w,k}^2} \left| {u_j^\prime}^H {v_k^\prime} - {v_j^\prime}^H {u_k^\prime} \right|^2 - \sum_{j} \frac{N - 2}{2} \sigma_{w,j}^{N} \left| {u_j^\prime}^H {v_j^\prime} - {v_j^\prime}^H {u_j^\prime} \right|^2 \\ 
  =& - 2 \sum_j \sigma_{w,j}^N \cdot \left\| \Sigma^{1/2} \left( {u_j^\prime} - {v_j^\prime} \right) \right\|^2 - 2\sum_j \sigma_{w,j}^{2N} \cdot \left\| u_j^\prime - v_j^\prime \right\|^2 \\
  -& \sum_{j,k} f_N \left( \sigma_{w,j} , \sigma_{w,k} \right) \left| {u_j^\prime}^H {v_k^\prime} - {v_j^\prime}^H {u_k^\prime} \right|^2 . 
 \end{aligned}
\end{equation}

This completes the proof. 

\end{proof}

For even depth $2\mid N$, we have a similar result written in matrix form: 

\begin{theorem}\label{antisym-error for 2|N}
If $2 \mid N$, the singular values of the product matrix $W(0)$ are different from zero at initialization, then 
  
\begin{equation}\label{antisym-error for 2|N, eq}
 \begin{aligned}
  \frac{\mathrm{d}}{\mathrm{d} t} \left\|\Sigma^{1/2}(U - V) \Sigma_w \right\|_F^2 &= -2 \left\|\Sigma (U - V) \Sigma_w^{N/2} \right\|_F^2 -2 \left\|\Sigma^{1/2}(U - V) \Sigma_w^{N} \right\|_F^2 \\ 
  &- 2 \Re\left(\mathrm{tr}\left( \sum_{j=1}^{N/2 - 1} \Sigma U \Sigma_w^{2j} \left(U^H \Sigma V - V^H \Sigma U\right) \Sigma_w^{N-2j} V^H \right)\right) \\
  &\le 0 . 
 \end{aligned}
\end{equation} 
\end{theorem}

We present another approach of proof which \textit{takes the inverse} of some terms. This approach \textit{adapts to the skew-hermitian term in unbalanced initialization}, where the proof of Theorem \ref{antisym-error, general balanced} in does not hold. 

To prove the theorem, we introduce the following lemma. 

\begin{lemma}\label{more on uv by inverse}

If $2 \mid N$, $\Sigma_w$ is full rank at initialization, then $\forall k = 0,1,\cdots,N/2$ we have

\begin{equation}
 \begin{aligned}
  &\frac{\mathrm{d}}{\mathrm{d} t} (U \pm V) \Sigma_w^{2k} (U \pm V)^H \\=& \sum_{j=1}^{k} \left[U \Sigma_w^{2(j-1)} U^H \Sigma V \Sigma_w^{N+2k-2j} U^H + U \Sigma_w^{N + 2(j-1)} V^H \Sigma U \Sigma_w^{2(k-j)} U^H \right. \\ 
  +& \left. V \Sigma_w^{2(j-1)} V^H \Sigma U \Sigma_w^{N+2k-2j} V^H + V \Sigma_w^{N + 2(j-1)} U^H \Sigma V \Sigma_w^{2(k-j)} V^H \right] \\ 
  \pm& \sum_{j=1}^{N/2 + k} \left[U \Sigma_w^{2(j-1)} U^H \Sigma V \Sigma_w^{N+2k-2j} V^H + V \Sigma_w^{2(j-1)} V^H \Sigma U \Sigma_w^{N+2k-2j} U^H \right] \\ 
  \mp& \sum_{j=1}^{N/2 - k} \left[U \Sigma_w^{2(j-1+k)} V^H \Sigma U \Sigma_w^{N-2j} V^H + V \Sigma_w^{2(j-1+k)} U^H \Sigma V \Sigma_w^{N-2j} U^H \right] \\
  -& 2k(U \pm V) \Sigma_w^{2(N+k-1)} (U \pm V)^H . 
 \end{aligned}
\end{equation}

\end{lemma}

\begin{proof}
$\forall l \in \mathbb{N}$ we have
\begin{equation}\label{u_power}
 \begin{aligned}
  &\frac{\mathrm{d}}{\mathrm{d} t}\left(U \Sigma_w^{2l} U^H\right) = \sum_{j=1}^{l} U \Sigma_w^{2(j-1)} U^H \left(\frac{\mathrm{d}}{\mathrm{d} t}\left(U \Sigma_w^{2} U^H\right)\right) U \Sigma_w^{2(l-j)} U^H \\ 
  =& \sum_{j=1}^{l} U \Sigma_w^{2(j-1)} U^H \left(\Sigma V \Sigma_w^N U^H + U \Sigma_w^N V^H \Sigma^H - 2 U \Sigma_w^{2N} U^H \right) U \Sigma_w^{2(l-j)} U^H . 
 \end{aligned}
\end{equation}
  
\begin{equation}\label{v_power}
 \begin{aligned}
  &\frac{\mathrm{d}}{\mathrm{d} t}\left(V \Sigma_w^{2l} V^H\right) = \sum_{j=1}^{l} V \Sigma_w^{2(j-1)} V^H \left(\frac{\mathrm{d}}{\mathrm{d} t}\left(V \Sigma_w^{2} V^H\right)\right) V \Sigma_w^{2(l-j)} V^H \\ 
  =& \sum_{j=1}^{l} V \Sigma_w^{2(j-1)} V^H \left(\Sigma U \Sigma_w^N V^H + V \Sigma_w^N U^H \Sigma^H - 2 V \Sigma_w^{2N} V^H \right) V \Sigma_w^{2(l-j)} V^H . 
 \end{aligned}
\end{equation}
  
From Lemma \ref{ASVD, non-negative diagonal}, $U \Sigma_w^{N-2k} U^H$ is invertible at arbitrary time $t\in[0,+\infty)$, thus

\begin{equation}
 \begin{aligned}
  \frac{\mathrm{d}}{\mathrm{d} t}\left(U \Sigma_w^{-(N-2k)} U^H \right) 
  &= - \left(U \Sigma_w^{N-2k} U^H \right)^{-1} \left[\frac{\mathrm{d}}{\mathrm{d} t} \left(U \Sigma_w^{N-2k} U^H \right) \right] \left(U \Sigma_w^{N-2k} U^H \right)^{-1} \\
  &= - \left(U \Sigma_w^{-(N-2k)} U^H \right) \left[\frac{\mathrm{d}}{\mathrm{d} t}\left(U \Sigma_w^{N-2k} U^H\right) \right] \left(U \Sigma_w^{-(N-2k)} U^H \right) , 
 \end{aligned}
\end{equation}

which further gives 
  
\begin{equation}\label{uv_power}
 \begin{aligned}
  &\frac{\mathrm{d}}{\mathrm{d} t}\left(U \Sigma_w^{2k} V^H \right) \\=& \left[\frac{\mathrm{d}}{\mathrm{d} t}\left(U \Sigma_w^{-(N-2k)} U^H \right) \right] U \Sigma_w^N V^H + U \Sigma_w^{-(N-2k)} U^H \left[\frac{\mathrm{d}}{\mathrm{d} t}\left(U \Sigma_w^{N} V^H \right) \right] \\ 
  =& - \left(U \Sigma_w^{-(N-2k)} U^H \right) \left[\frac{\mathrm{d}}{\mathrm{d} t}\left(U \Sigma_w^{N-2k} U^H \right) \right] \left(U \Sigma_w^{2k} V^H \right) \\+& U \Sigma_w^{-(N-2k)} U^H \left[\frac{\mathrm{d}}{\mathrm{d} t}\left(U \Sigma_w^{N} V^H \right) \right] \\ 
  =& \sum_{j=1}^{N/2+k} U \Sigma_w^{2(j-1)} U^H \Sigma V \Sigma_w^{N+2(k-j)} V^H + \sum_{j=1}^{N/2-k} U \Sigma_w^{2(k+j-1)} V^H \Sigma^H U \Sigma_w^{N-2j} V^H \\-& 2 k U \Sigma_w^{2(N+k-1)} V^H . 
 \end{aligned}
\end{equation}

Combine (\ref{u_power}), (\ref{v_power}) and (\ref{uv_power}) together, then the proof is completed. 

\end{proof}

Now we present the proof of Theorem \ref{antisym-error for 2|N}. 

\begin{proof}\label{antisym-error for 2|N, proof}

Denote $Q = U^H \Sigma V$, calculate the L.H.S. of (\ref{antisym-error for 2|N, eq}) by setting $k=1$ in Lemma \ref{more on uv by inverse}: 
  
\begin{equation}
 \begin{aligned}
  &\frac{\mathrm{d}}{\mathrm{d} t} \left\|\Sigma^{1/2}(U - V) \Sigma_w \right\|_F^2 \\=& \frac{\mathrm{d}}{\mathrm{d} t} \mathrm{tr} \left( \Sigma (U - V) \Sigma_w^{2} (U - V)^H \right) \\ 
  =&  -2 \mathrm{tr}\left(\Sigma^2 (U - V) \Sigma_w^{N} (U - V)^H\right)  - 2\mathrm{tr}\left( \Sigma (U - V) \Sigma_w^{2N} (U - V)^H \right) \\
  -& 2 \Re\left(\mathrm{tr}\left( \sum_{j=1}^{N/2 - 1} \Sigma U \Sigma_w^{2j} \left(U^H \Sigma V - V^H \Sigma U\right) \Sigma_w^{N-2j} V^H \right)\right) \\ 
  =& -2 \left\|\Sigma (U - V) \Sigma_w^{N/2} \right\|_F^2 -2 \left\|\Sigma^{1/2}(U - V) \Sigma_w^{N} \right\|_F^2 \\
  -& 2\Re\left(\mathrm{tr}\left( \sum_{j=1}^{N/2 - 1} \Sigma_w^{2j} (Q - Q^H) \Sigma_w^{N-2j} Q^H \right)\right) . 
 \end{aligned}
\end{equation}

To analyze the last term, 

\begin{equation}
 \begin{aligned}
  & \Re\left(\mathrm{tr}\left( \sum_{j=1}^{N/2 - 1} \Sigma_w^{2j} (Q - Q^H) \Sigma_w^{N-2j} Q^H \right)\right) \\ =& \Re\left( \sum_{m,n} \left( \sum_{j=1}^{N/2 - 1} \sigma_m^{2j}(\Sigma_w) (Q_{mn} - \overline{Q_{nm}}) \sigma_n^{N-2j}(\Sigma_w) \overline{Q_{mn}} \right)\right) \\ 
  =& \frac{1}{2} \sum_{m,n} \left( \sum_{j=1}^{N/2 - 1} \sigma_m^{2j}(\Sigma_w) \sigma_n^{N-2j}(\Sigma_w) (\left|Q_{mn}\right|^2 + \left|Q_{nm}\right|^2 - 2 \Re(Q_{mn} Q_{nm})) \right) \\ 
  =& \frac{1}{2} \sum_{m,n} \left|Q_{mn} - \overline{Q_{nm}} \right|^2 \left( \sum_{j=1}^{N/2 - 1} \sigma_m^{2j}(\Sigma_w) \sigma_n^{N-2j}(\Sigma_w) \right) \ge 0 . 
 \end{aligned}
\end{equation}
  
Thus for arbitrary $\Sigma \succ O$ we have 

\begin{equation}
 \begin{aligned}
  \frac{\mathrm{d}}{\mathrm{d} t} \left\|\Sigma^{1/2}(U - V) \Sigma_w \right\|_F^2 
  &= -2 \left\|\Sigma (U - V) \Sigma_w^{N/2} \right\|_F^2 -2 \left\|\Sigma^{1/2}(U - V) \Sigma_w^{N} \right\|_F^2 \\ 
  &- \sum_{m,n} \left|Q_{mn} - \overline{Q_{nm}} \right|^2 \left( \sum_{j=1}^{N/2 - 1} \sigma_m^{2j}(\Sigma_w) \sigma_n^{N-2j}(\Sigma_w) \right) \\ &\le 0 . 
 \end{aligned}
\end{equation}

which completes the proof. 

\end{proof}

\subsection{Hermitian Main Term}

This section proves Theorem \ref{dynamics of minimum singular value of hermitian term}. 

\begin{proof}\label{dynamics of minimum singular value of hermitian term, proof}

Consider 

\begin{equation}
 \begin{aligned}
  &\frac{\mathrm{d}}{\mathrm{d} t} (U + V) \Sigma_w^2 (U + V)^H \\
  =& \Sigma (U+V)\Sigma_w^N (U+V)^H + (U+V)\Sigma_w^N (U+V)^H \Sigma - 2 (U+V)\Sigma_w^{2N} (U+V)^H \\
  +& \sum_{j=1}^{N/2-1} \left[U \Sigma_w^{2j} \left( U^H \Sigma V - V^H \Sigma U \right) \Sigma_w^{N-2j} V^H + V \Sigma_w^{2j} \left( V^H \Sigma U - U^H \Sigma V \right) \Sigma_w^{N-2j} U^H \right] . 
 \end{aligned}
\end{equation}

Denote $P = \frac{(U + V)\Sigma_w}{2}$, $Q = \frac{(U - V)\Sigma_w}{2}$. Then $P^H Q = - Q^H P$, $\Sigma_w^2 = P^H P + Q^H Q$. 

From $A B C^H - C B A^H = \frac{1}{2} \left[(A-C)B(A+C)^H - (A+C)B(A-C)^H \right]$ for arbitrary $A, B, C$ we have 

\begin{equation}
 \begin{aligned}
  \frac{\mathrm{d}}{\mathrm{d} t} P P^H &= \Sigma P \Sigma_w^{N-2} P^H + P \Sigma_w^{N-2} P^H \Sigma - 2 P \Sigma_w^{2N-2} P^H \\
  &+ \sum_{j=1}^{N/2-1} \left[Q \Sigma_w^{2j-2} \left(Q^H \Sigma P - P^H \Sigma Q\right) \Sigma_w^{N-2j-2} P^H \right.\\&-\left. P \Sigma_w^{2j-2} \left(Q^H \Sigma P - P^H \Sigma Q\right) \Sigma_w^{N-2j-2} Q^H \right] . 
 \end{aligned}
\end{equation}

Suppose the $k^{th}$ eigenvalue and eigenvector of $PP^H$ are $x_k^2$ and $\xi_k$ respectively, $P^H \xi_k = x_k \eta_k$, then

\begin{equation}
 \begin{aligned}
  \frac{\mathrm{d}}{\mathrm{d} t} x_k^2 &= \xi_k^H\left( \frac{\mathrm{d}}{\mathrm{d} t} P P^H \right) \xi_k \\ &= 2 \xi_k^H \Sigma P \Sigma_w^{N-2} P^H \xi_k - 2 \xi_k^H P \Sigma_w^{2N-2} P^H \xi_k \\&+ 2 \xi_k^H \left[\sum_{j=1}^{N/2-1} Q \Sigma_w^{2j-2} \left(Q^H \Sigma P - P^H \Sigma Q\right) \Sigma_w^{N-2j-2} P^\top\right] \xi_k . 
 \end{aligned}
\end{equation}

We focus on $N=4$, $\Sigma = \sigma_1(\Sigma)I$. Then 

\begin{equation}
 \begin{aligned}
  \frac{\mathrm{d}}{\mathrm{d} t} x_k^2 &= 2\sigma_1(\Sigma) \xi_k^H P \Sigma_w^{2} P^H \xi_k - 2 \xi_k^H P \Sigma_w^{6} P^H \xi_k + 4 \sigma_1(\Sigma) \xi_k^H Q  Q^H P  P^H \xi_k \\
  &= 2\sigma_1(\Sigma) \xi_k^H P \Sigma_w^{2} P^H \xi_k - 2 \xi_k^H P \Sigma_w^{6} P^H \xi_k + 4 \sigma_1(\Sigma) x_k^2 \xi_k^H Q Q^H \xi_k . 
 \end{aligned}
\end{equation}

For the second term: 

\begin{equation}
 \begin{aligned}
  \xi_k^H P \Sigma_w^{6} P^H \xi_k &= \xi_k^H P \left(P^H P + Q^H Q\right)\Sigma_w^{2} \left(P^H P + Q^H Q\right) P^H \xi_k \\ 
  &= x_k^4 \xi_k^H P \Sigma_w^{2} P^H \xi_k +2 x_k^2 \xi_k^H P \Sigma_w^{2} Q^H Q P^H \xi_k + \xi_k^H P Q^H Q \Sigma_w^{2} Q^H Q P^H \xi_k \\
  &\le x_k^4 \xi_k^H P \Sigma_w^{2} P^H \xi_k + 2 x_k^4 \|Q\|_{op}^2 \|\Sigma_w\|_{op}^2 + x_k^2 \|Q\|_{op}^4 \|\Sigma_w\|_{op}^2 . 
 \end{aligned}
\end{equation}

From Theorem \ref{antisym-error for 2|N}, $\|Q\|_{op} \le \|Q\|_{F} \le \|Q(t=0)\|_{F}$. Then 

\begin{equation}
 \begin{aligned}
  &\frac{\mathrm{d}}{\mathrm{d} t} x_k^2 \ge \left(2\sigma_1(\Sigma) - x_k^4\right) \xi_k^H P \Sigma_w^{2} P^H \xi_k - 2 x_k^4 \|Q\|_{op}^2 \|\Sigma_w\|_{op}^2 - x_k^2 \|Q\|_{op}^4 \|\Sigma_w\|_{op}^2 \\
  \ge& \left(2\sigma_1(\Sigma) - x_k^4 - \frac{1}{2} \|\Sigma_w\|_{op}^2 \|((U-V)\Sigma_w)|_{t=0} \|_{F}^2 \right) x_k^4 - \frac{1}{16} x_k^2 \|\Sigma_w\|_{op}^2 \|((U-V)\Sigma_w)|_{t=0}\|_{F}^4 . 
 \end{aligned}
\end{equation}

The lower bound is proved. 

For the upper bound, 

\begin{equation}
 \begin{aligned}
  \frac{\mathrm{d}}{\mathrm{d} t} x_k^2 \le 2\sigma_1(\Sigma) x_k^2 \|\Sigma_w\|_{op}^{2} + 4 \sigma_1(\Sigma) x_k^2 \left\| Q \right\|_{op}^2 . 
 \end{aligned}
\end{equation}

This completes the proof.

\end{proof}

\begin{corollary}\label{never converge to optimum, balanced} If for some $k$, $\sigma_k((U+V)\Sigma_w)|_{t=0} = 0$, then $\sigma_k((U+V)\Sigma_w) \equiv 0$ for finite time $t \in [0,+\infty)$. 

\end{corollary}

\begin{proof}

Denote $x_k \equiv \frac{1}{2}\sigma_k((U+V)\Sigma_w)$. 
By Lemma \ref{L ori non-increasing}, $\|\Sigma - W\|_F \le \|\Sigma - W(0)\|_F$. Then $\|\Sigma_w\|_{op}$ is bounded: 

\begin{equation}
 \begin{aligned}
  \|\Sigma_w\|_{op} &= \|W\|_{op}^{1/N} \le \left( \|\Sigma\|_{op} + \|\Sigma - W\|_{op} \right)^{1/N} \le \left(\|\Sigma\|_{op} + \|\Sigma - W\|_{F} \right)^{1/N} \\&\le \left( \|\Sigma\|_{op} + \|\Sigma - W(0)\|_{F} \right)^{1/N} . 
 \end{aligned}
\end{equation}

Then from Theorem \ref{dynamics of minimum singular value of hermitian term}, there exists some $C \in(0, +\infty)$ such that 

\begin{equation}
 \begin{aligned}
  \frac{\mathrm{d}}{\mathrm{d} t} x_k^2 \le \sigma_1(\Sigma) \left(2 \|\Sigma_w\|_{op}^{2} + \|((U-V)\Sigma_w)|_{t=0} \|_{F}^2 \right) x_k^2 \le C x_k^2 . 
 \end{aligned}
\end{equation}

Giving 

\begin{equation}
  x_k^2(t) \le x_k^2(0) e^{Ct} = 0 . 
\end{equation}

This completes the proof. 

\end{proof}

\subsection{Convergence proof}

This section states the global convergence guarantee under balanced Gaussian initialization, with gradient flow. Below we omit the confidence level $\delta$ in $f_1(\delta) = O\left( \frac{1}{\delta} \right)$ and $f_2^\prime(\delta) = O\left( \frac{1}{\delta^2} \right)$ for simplicity.

\begin{theorem}\label{Total convergence bound, balanced, gf} Global convergence bound under balanced Gaussian initialization, gradient flow. 

For four-layer matrix factorization under gradient flow, balanced Gaussian initialization with scaling factor $\epsilon \le \frac{\sigma_1^{1/4}(\Sigma)}{4 f_1^2 f_2^\prime d^{29/8}}$, then for target matrix with identical singular values, 

1. For $\mathbb{F} = \mathbb{R}$, with probability at least $\frac{1}{2}$ the loss does not converge to zero. Specifically, 

\begin{equation}
  \mathcal{L}(t) \ge \frac{1}{2} \sigma_1^2(\Sigma),\, \forall t\in[0, +\infty) .  
\end{equation}

2. For $\mathbb{F} = \mathbb{C}$ with high probability and for $\mathbb{F} = \mathbb{R}$ with probability close to $\frac{1}{2}$, there exists $T(\epsilon_{\rm conv}) = \frac{16 {f_2^\prime}^2 d^3}{\sigma_1(\Sigma) \epsilon^2} + \frac{1}{8 \sigma_1^{3/2}(\Sigma)} \ln \left(\frac{d\sigma_1^2(\Sigma)}{\epsilon_{\rm conv}}\right)$, such that for any $\epsilon_{\rm conv} > 0$, when $t > T(\epsilon_{\rm conv})$, $\mathcal{L}(t) < \epsilon_{\rm conv}$. 

\end{theorem}

\begin{remark}

The first part of this Theorem can be generalized to general (bounded) balanced initialization. 

\end{remark}

\begin{proof}

For the first conclusion, by Theorem \ref{Balanced initialization, final} and Corollary \ref{never converge to optimum, balanced}, for $\mathbb{F} = \mathbb{R}$, $\sigma_{\min}((U+V)\Sigma_w) \equiv 0$ with probability at least $\frac{1}{2}$. Consequently $\sigma_{\min}((U+V)\Sigma_w^N) \equiv 0
$. 

Suppose at time $t$, for some unit vector $y$, $(U+V)\Sigma_w^N y(t) = 0$. Then

\begin{equation}
 \begin{aligned}
  \|\Sigma - W\|_F &= \|\sigma_1(\Sigma)I - U\Sigma_w^N V^\top \|_F = \|\sigma_1(\Sigma)V - U\Sigma_w^N \|_F \\
  &\ge \|\sigma_1(\Sigma) V - U\Sigma_w^N\|_{op} \ge \left\|(\sigma_1(\Sigma) V - U\Sigma_w^N)y \right\| \\
  &=\left\|(\sigma_1(\Sigma) V +V\Sigma_w^N)y \right\| = \left\|(\sigma_1(\Sigma) + \Sigma_w^N)y \right\| \ge \sigma_1(\Sigma) . 
 \end{aligned}
\end{equation}

For the second part: 

From Lemma \ref{L ori non-increasing}, $\|\Sigma - W\|_F \le \|\Sigma - W(0)\|_F < 2 \sqrt{d} \sigma_1(\Sigma)$. Thus for any time $t$, 

\begin{equation}
 \begin{aligned}
  \|\Sigma_w\|_{op} &= \|W\|_{op}^{1/4} \le \left( \|\Sigma\|_{op} + \|\Sigma - W\|_{op} \right)^{1/4} \le \left(\|\Sigma\|_{op} + \|\Sigma - W\|_{F} \right)^{1/N} \\&\le \left( \|\Sigma\|_{op} + \|\Sigma - W(0)\|_{F} \right)^{1/4} \le \sqrt{2}d^{1/8} \sigma_1^{1/4}(\Sigma) . 
 \end{aligned}
\end{equation}

From Theorem \ref{Balanced initialization, final}, for $\mathbb{F} = \mathbb{C}$ with high probability (while for $\mathbb{F} = \mathbb{R}$ with probability close to $\frac{1}{2}$), $x_k(t=0)\ge \frac{\epsilon}{2 f_2^\prime d^{3/2}}$, $\|(U-V)\Sigma_w\|_F|_{t=0} \le 2 f_1 d \epsilon$. Thus by taking $\epsilon \le \frac{\sigma_1^{1/4}(\Sigma)}{4 f_1^2 f_2^\prime d^{29/8}}$, for $t$ such that $x_k(t) \ge x_k(0)$, 

\begin{equation}
  \frac{\mathrm{d}}{\mathrm{d} t} x_k^2 \ge \left( 2 \sigma_1(\Sigma) - \left(4 f_1^2 d^{9/4} + 8 f_1^4 {f_2^\prime}^2 d^{29/4} \right) \epsilon^2 \sigma_1^{1/2}(\Sigma) - x_k^4 \right) x_k^4 \ge \left(\frac{5}{4} \sigma_1(\Sigma) - x_k^4 \right) x_k^4 . 
\end{equation}

This indicates that all $x_k$ monotonically increase to $ \sigma_1^{1/4}(\Sigma)$ in $T_1 = \frac{4}{\sigma_1(\Sigma)} \cdot x_k(0)^{-2} = \frac{16 {f_2^\prime}^2 d^3}{\sigma_1(\Sigma) \epsilon^2}$, and never decrease to below $ \sigma_1^{1/4}(\Sigma)$ for $t > T_1$. 

By Theorem \ref{bound of eigenvalues under perturbation}, $\sigma_{\min}(\Sigma_w) \ge x_k$. Then combine with Lemma \ref{L ori non-increasing}, 

\begin{equation}
  \mathcal{L}_{\rm ori}(t) \le \mathcal{L}_{\rm ori}(0) e^{-8 \sigma_{\min}^6(\Sigma_w(T_1)) (t-T_1)} \le d \sigma_1^2(\Sigma) e^{-8\sigma_1^{3/2}(\Sigma) (t-T_1)} . 
\end{equation}

Thus it takes at most $t = T_1 + \frac{1}{8 \sigma_1^{3/2}(\Sigma)} \ln \left(\frac{d\sigma_1^2(\Sigma)}{\epsilon_{\rm conv}}\right)$ to reach $\epsilon_{\rm conv}$-convergence. 

\end{proof}

%% file: subfiles/appendix/appendix_5_notations_and_preliminaries.tex
\section{Notations and Preliminaries under the Depth of four, unbalanced}\label{section: notations and Preliminaries}

To tackle the unbalanced initialization with depth $N=4$, we make the following notations and derive some basic properties. 

Below we denote $R = W_2^{-1} W_3^H$, $W_1^\prime = R W_4^H$, $W = W_4 W_3 W_2 W_1$, $M_2 = W_2^H W_2$, $M_1 = W_1 W_1^H$, $M_{\Delta1234} = W_2 W_1 W_1^H W_2^H - W_3^H W_4^H W_4 W_3$ $M_1^\prime = W_1^\prime W_1^{\prime H}$, $e_{\Delta} = \sqrt{\sum_{i=1}^{3} \|\Delta_{i,i+1}\|_F^2}$. Then: 

\begin{equation}
  W = W_1^{\prime H} M_2 W_1 , 
\end{equation}

\begin{equation}
  RR^H = W_2^{-1} W_3^H W_3 W_2^{H-1} = I - W_2^{-1} \Delta_{23} W_2^{H-1} , 
\end{equation}

\begin{equation}
  R^{-1}R^{H-1} = W_3^{H-1} W_2 W_2^H W_3^{-1} = I + W_3^{H-1} \Delta_{23} W_3^{-1} , 
\end{equation}

\begin{equation}
 \begin{aligned}
  M_{\Delta1234} &= \left( \left(W_2^H W_2\right)^2 - \left(W_3 W_3^H\right)^2 \right) + W_3^H \Delta_{34} W_3 + W_2 \Delta_{12} W_2^{H} \\ 
  &= \frac{1}{2} \left( \Delta_{23} \left(W_3^H W_3 + W_2 W_2^H\right) + \left(W_3^H W_3 + W_2 W_2^H\right) \Delta_{23} \right) \\&+ W_3^H \Delta_{34} W_3 + W_2 \Delta_{12} W_2^{H} , 
 \end{aligned}
\end{equation}

\begin{equation}
 \begin{aligned}
  M_1^\prime - M_1 &= W_2^{-1} M_{\Delta1234} W_2^{H -1} . 
 \end{aligned}
\end{equation}

Deducing that 

\begin{equation}
  \left\| R \right\|_{op} \le \sqrt{1 + \frac{1}{\sigma_{\min}^2(W_2)} \cdot \| \Delta_{23} \|_{op}}\le \sqrt{1 + \frac{1}{\min_{j,k}\sigma_k^2(W_j)} \cdot e_{\Delta} } , 
\end{equation}

\begin{equation}
  \left\| R^{-1} \right\|_{op} \le \sqrt{1 + \frac{1}{\sigma_{\min}^2(W_3)} \cdot \| \Delta_{23} \|_{op} } \le \sqrt{1 + \frac{1}{\min_{j,k}\sigma_k^2(W_j)} \cdot e_{\Delta} } , 
\end{equation}

\begin{equation}
  \left\| I - RR^H \right\|_{op} \le \frac{1}{\sigma_{\min}^2(W_2)} \cdot \|\Delta_{23}\|_{op} \le \frac{1}{\min_{j,k}\sigma_k^2(W_j)} \cdot e_{\Delta} , 
\end{equation}

\begin{equation}
  \left\| I - R^{-1}R^{H-1} \right\|_{op} \le \frac{1}{\sigma_{\min}^2(W_3)} \cdot \|\Delta_{23}\|_{op} \le \frac{1}{\min_{j,k}\sigma_k^2(W_j)} \cdot e_{\Delta} , 
\end{equation}

\begin{equation}
 \begin{aligned}
  \left\| M_{\Delta1234} \right\|_{op} &\le \left( \|W_2\|_{op}^2 + \|W_3\|_{op}^2 \right) \|\Delta_{23}\|_{op} + \|W_3\|_{op}^2 \|\Delta_{34}\|_{op} + \|W_2\|_{op}^2 \|\Delta_{12}\|_{op} \\
  & \le \sqrt{6} \max_{j,k} \sigma_{k}^2(W_j) e_{\Delta} , 
 \end{aligned}
\end{equation}

\begin{equation}
  \left\| M_1^\prime - M_1 \right\|_{op} \le \sqrt{6} \cdot\frac{\max_{j,k} \sigma_{k}^2(W_j)}{ \sigma_{\min}^2(W_2)} e_{\Delta} \le \sqrt{6} \cdot\frac{\max_{j,k} \sigma_{k}^2(W_j)}{\min_{j,k} \sigma_{k}^2(W_j)} e_{\Delta} . 
\end{equation}

Applying Lemma \ref{bound of RR^H}, 

\begin{equation}
  \left\| I - R^H R \right\|_{op} \le \frac{1}{\sigma_{\min}^2(W_2)} \cdot \|\Delta_{23}\|_{op} \le \frac{1}{\min_{j,k}\sigma_k^2(W_j)} \cdot e_{\Delta} , 
\end{equation}

\begin{equation}
  \left\| I - R^{H-1} R^{-1} \right\|_{op} \le \frac{1}{\sigma_{\min}^2(W_3)} \cdot \|\Delta_{23}\|_{op} \le \frac{1}{\min_{j,k}\sigma_k^2(W_j)} \cdot e_{\Delta} . 
\end{equation}

%% file: subfiles/appendix/appendix_6_skew_and_main.tex
\section{Skew-Hermitian Error Term and Hermitian Main Term for four-layer matrix decomposition}

In this section, we construct skew-hermitian error term and hermitian main term to prepare for the convergence proof, under four-layer setting with scaled identical target matrix $\Sigma = \sigma_1(\Sigma) I$. 

\subsection{Skew-Hermitian Error Term}

The skew-hermitian error term is defined by $\left \| W_1 - W_1^\prime \right \|_F^2$. To address the dynamics: 

\subsubsection{Gradient Flow}\label{section: skew-hermitian error term, gf}

Consider $\Sigma = \sigma_1(\Sigma) I$. We study $\left\|W_1 - W_1^\prime \right\|_F^2$. From the derivative of inverse,

\begin{equation}
 \begin{aligned}
  \frac{\mathrm{d} W_2^{-1}}{\mathrm{d}t} = - W_2^{-1} \frac{\mathrm{d}W_2}{\mathrm{d}t} W_2^{-1} = - W_1^\prime (\Sigma - W) W_1^H W_2^{-1} - a \Delta_{12} W_2^{-1} + a W_2^{-1} \Delta_{23} , 
 \end{aligned}
\end{equation}

\begin{equation}
 \begin{aligned}
  \frac{\mathrm{d} R}{\mathrm{d}t} &= \frac{\mathrm{d} W_2^{-1}}{\mathrm{d}t} W_3^H + W_2^{-1} \frac{\mathrm{d}W_3^H}{\mathrm{d}t} \\
  &= - R W_4^H (\Sigma - W) W_1^H R + W_1 \left(\Sigma-W^H\right) W_4 \\&- a \Delta_{12} R + 2a W_2^{-1}\Delta_{23} W_3^H - aR \Delta_{34} , 
 \end{aligned}
\end{equation}

\begin{equation}
 \begin{aligned}
  \frac{\mathrm{d} W_1^\prime}{\mathrm{d}t} &= \frac{\mathrm{d} W_2^{-1}}{\mathrm{d}t} W_3^H W_4^H + W_2^{-1} \frac{\mathrm{d}W_3^H}{\mathrm{d}t} W_4^H + W_2^{-1} W_3^H \frac{\mathrm{d}W_4^H}{\mathrm{d}t} \\
  &= - W_1^\prime (\Sigma - W) W_1^H W_1^\prime + W_1 \left(\Sigma - W^H\right) W_1^{\prime H} R^{H-1} R^{-1} W_1^\prime \\
  &+ R R^H W_2^H W_2 W_1 \left(\Sigma - W^H\right) - a \Delta_{12} W_1^\prime + 2a W_2^{-1} \Delta_{23} W_2 W_1^\prime . 
 \end{aligned}
\end{equation}

From $\Re(\mathrm{tr}(PQ)) = 0$ if $P = P^H$ and $Q = -Q^H$, we have

\begin{equation}
 \begin{aligned}
  &\Re\left(\mathrm{tr}\left( \left(W_1^\prime W_1^H - W_1 W_1^{\prime H}\right) W_1^\prime \left(W_1 - W_1^\prime\right)^H \right)\right) \\=& -\frac{1}{2} \mathrm{tr}\left( \left(W_1^\prime W_1^H - W_1 W_1^{\prime H}\right) \left(W_1^\prime W_1^H - W_1 W_1^{\prime H}\right)^H \right) . 
 \end{aligned}
\end{equation}

Thus 

\begin{equation}
 \begin{aligned}
  \frac{\mathrm{d}}{\mathrm{d}t} \left \| W_1 - W_1^\prime \right \|_F^2 &= 2 \Re\left(\mathrm{tr}\left(\frac{\mathrm{d}(W_1 - W_1^\prime)}{\mathrm{d}t}(W_1 - W_1^\prime)^H \right) \right)\\
  &= 2 \Re\Big( \mathrm{tr}\Big( \left[ M_2 W_1^\prime (\Sigma - W) + W_1^\prime (\Sigma - W) W_1^H W_1^\prime \right. \\
  & - W_1 (\Sigma - W^H) W_1^{\prime H} R^{H-1} R^{-1} W_1^\prime - R R^H M_2 W_1 (\Sigma - W^H) \\&\left.\left. - a \Delta_{12} \left(W_1 - W_1^\prime \right) - 2a W_2^{-1} \Delta_{23} W_2 W_1^\prime \right] \left( W_1 - W_1^\prime \right)^H \right)\Big) \\ 
  &= -2 \sigma_1(\Sigma) \mathrm{tr}\left( \left( W_1 - W_1^\prime \right)^H M_2 \left(W_1 - W_1^\prime \right) \right) \\&- \sigma_1(\Sigma) \mathrm{tr}\left( \left( W_1^\prime W_1^H - W_1 W_1^{\prime H} \right) \left( W_1^\prime W_1^H - W_1 W_1^{\prime H} \right)^H \right) \\
  &- \mathrm{tr}\left( M_2 \left( M_1^\prime + M_1\right)M_2\left(W_1 - W_1^\prime\right) \left( W_1 - W_1^\prime \right)^H \right) \\
  &- \mathrm{tr}\left( M_2 \left( M_1^\prime - M_1\right)M_2\left(W_1^\prime + W_1\right) \left( W_1 - W_1^\prime \right)^H\right) \\
  &+ 2 \mathrm{tr}\left( \left[ - M_1^\prime M_2 M_1 + M_1 M_2 M_1^\prime \right] W_1^\prime \left( W_1 - W_1^\prime \right)^H\right) \\
  &+ 2 \Re\left(\mathrm{tr}\left( \left[ W_1 (\Sigma - W^H) W_4 \left( R^{H} R - I \right) W_4^H \right] \left( W_1 - W_1^\prime \right)^H\right)\right) \\
  &+ 2 \Re\left(\mathrm{tr}\left( \left[ \left(I - R R^H \right) W_2^H W_2 W_1 (\Sigma - W^H) \right] \left( W_1 - W_1^\prime \right)^H\right)\right) \\
  &- 2a \Re\left(\mathrm{tr}\left(\Delta_{12} \left(W_1 - W_1^\prime \right) \left( W_1 - W_1^\prime \right)^H \right)\right) \\
  &- 4a \Re\left(\mathrm{tr}\left( W_2^{-1} \Delta_{23} W_2 W_1^\prime \left( W_1 - W_1^\prime \right)^H \right)\right) . 
 \end{aligned}
\end{equation}

Note: $ - M_1^\prime M_2 M_1 + M_1 M_2 M_1^\prime = \frac{1}{2} \left[\left(M_1 - M_1^\prime \right)M_2\left(M_1 + M_1^\prime \right) - \left(M_1 + M_1^\prime \right)M_2\left(M_1 - M_1^\prime\right)\right]$. 

\subsubsection{Gradient Descent}\label{section: skew-hermitian error term, gd}

From Lemma \ref{error bound, inverse},

\begin{equation}
 \begin{aligned}
  &\left\|W_2(t+1)^{-1} - W_2(t)^{-1} \right. \\ -& \left. \eta \left[- W_1^\prime(t) (\Sigma - W(t)) W_1(t)^H W_2(t)^{-1} - a \Delta_{12}(t) W_2(t)^{-1} + a W_2(t)^{-1} \Delta_{23}(t) \right] \right\|_{F} \\ 
  \le& \eta^2 \left[ \left( 1 + e_\Delta(t) \left\|W_2(t)^{-1}\right\|_{op}^2 \right) \|W_1(t)\|_{op} \| \Sigma - W(t)\|_F + \sqrt{2} a e_\Delta(t) \left\|W_2(t)^{-1}\right\|_{op} \right] \\ \cdot & \|W_2(t+1)^{-1}\|_{op} \|\nabla_{W_2} \mathcal{L}(t)\|_F . 
 \end{aligned}
\end{equation}

Under $\|W_j(t+1)\|_{op} = O(\|W_j(t)\|_{op})$, $e_\Delta(t) \left\|W_2(t)^{-1}\right\|_{op}^2 = O(1)$, 

\begin{equation}
 \begin{aligned}
  &\left\| W_1^\prime(t+1) - W_1^\prime(t) \right. \\
  -& \eta \left[ - W_1^\prime(t) (\Sigma - W(t)) W_1(t)^H W_1^\prime(t) \right. \\+& W_1(t) \left(\Sigma - W(t)^H\right) W_1^{\prime}(t)^{H} R(t)^{H-1} R(t)^{-1} W_1^\prime(t) \\+& R(t) R(t)^H W_2(t)^H W_2(t) W_1(t) \left(\Sigma - W(t)^H\right) \\
  -& \left.\left. a \Delta_{12}(t) W_1^\prime(t) + 2a W_2^{-1}(t) \Delta_{23}(t) W_2(t) W_1^\prime(t) \right] \right\|_F \\
  =& \eta^2 O\left(\left[ \max_{j \in [1,4]\cap \mathbb{N}^*} \left\| W_j(t) \right\|_{op} \| \Sigma - W(t)\|_F + a e_\Delta(t) \left\|W_2(t)^{-1}\right\|_{op} \right] \right. \\ \cdot & \left. \max_{j \in [1,4]\cap \mathbb{N}^*} \left\| W_j(t) \right\|_{op}^2 \cdot \|W_2(t+1)^{-1}\|_{op} \cdot \max_{j \in [1,4]\cap \mathbb{N}^*}\|\nabla_{W_j} \mathcal{L}(t)\|_F \right) . 
 \end{aligned}
\end{equation}

Finally giving 

\begin{equation}
 \begin{aligned}
  &\left \| W_1(t+1) - W_1^\prime(t+1) \right \|_F^2 - \left \| W_1(t) - W_1^\prime(t) \right \|_F^2 \\
  =& \Re\left(\mathrm{tr}\left(\left[\left(W_1(t+1) - W_1^\prime(t+1)\right) + \left(W_1(t) - W_1^\prime(t) \right)\right] \right.\right.\\ \cdot& \left.\left. \left[\left(W_1(t+1) - W_1^\prime(t+1)\right) - \left(W_1(t) - W_1^\prime(t) \right)\right]^H \right) \right)\\
  =& -2 \eta \sigma_1(\Sigma) \mathrm{tr}\left( \left( W_1(t) - W_1^\prime(t) \right)^H M_2(t) \left(W_1(t) - W_1^\prime(t) \right) \right) \\-& \eta \sigma_1(\Sigma) \mathrm{tr}\left( \left( W_1^\prime(t) W_1(t)^H - W_1(t) W_1^{\prime}(t)^{H} \right) \left( W_1^\prime(t) W_1(t)^H - W_1(t) W_1^{\prime}(t)^{H} \right)^H \right) \\
  -& \eta \mathrm{tr}\left( M_2(t) \left( M_1^\prime(t) + M_1(t) \right)M_2(t) \left(W_1(t) - W_1^\prime(t) \right) \left( W_1(t) - W_1^\prime(t) \right)^H \right) \\
  -& \eta \mathrm{tr}\left( M_2(t) \left( M_1^\prime(t) - M_1(t)\right)M_2(t)\left(W_1^\prime(t) + W_1(t)\right) \left( W_1(t) - W_1^\prime(t) \right)^H\right) \\
  +& 2 \eta \mathrm{tr}\left( \left[ - M_1^\prime(t) M_2(t) M_1(t) + M_1(t) M_2(t) M_1^\prime(t) \right] W_1^\prime(t) \left( W_1(t) - W_1^\prime(t) \right)^H\right) \\
  +& 2 \eta \Re\left(\mathrm{tr}\left( \left[ W_1(t) (\Sigma - W(t)^H) W_4(t) \left( R(t)^{H} R(t) - I \right) W_4(t)^H \right] \left( W_1(t) - W_1^\prime(t) \right)^H\right)\right) \\
  +& 2 \eta \Re\left(\mathrm{tr}\left( \left[ \left(I - R(t) R(t)^H \right) W_2(t)^H W_2(t) W_1(t) (\Sigma - W(t)^H) \right] \left( W_1(t) - W_1^\prime(t) \right)^H\right)\right) \\
  -& 2\eta a \Re\left(\mathrm{tr}\left(\Delta_{12}(t) \left(W_1(t) - W_1^\prime(t) \right) \left( W_1(t) - W_1^\prime(t) \right)^H \right)\right) \\
  -& 4\eta a \Re\left(\mathrm{tr}\left( W_2^{-1}(t) \Delta_{23}(t) W_2(t) W_1^\prime(t) \left( W_1(t) - W_1^\prime(t) \right)^H \right)\right) \\
  +& \eta^2 O\left(\left[ \max_{j \in [1,4]\cap \mathbb{N}^*} \left\| W_j(t) \right\|_{op} \| \Sigma - W(t)\|_F + a e_\Delta(t) \left\|W_2(t)^{-1}\right\|_{op} \right]^2 \right. \\ \cdot & \left. \max_{j \in [1,4]\cap \mathbb{N}^*} \left\| W_j(t) \right\|_{op}^5 \cdot \|W_2(t+1)^{-1}\|_{op} \right) . 
 \end{aligned}
\end{equation}

\subsection{Skew-Hermitian Error Term}

\subsubsection{Gradient Flow}\label{section: hermitian main term, gf}

For gradient flow, we study the $k^{th}$ singular value of $W_1 + W_1^\prime$, or equivalently $\lambda_{k}\left(\left(W_1 + W_1^\prime \right)^H\left(W_1 + W_1^\prime \right)\right) = \sigma_{k}^2 \left(W_1 + W_1^\prime \right)$. To address the dynamics: 

Suppose the left and right singular vector of $W_1 + W_1^\prime$ corresponding to $\sigma_{k}(t) = \sigma_{k}\left(W_1 + W_1^\prime \right)(t)$ are $\eta_k(t)$ and $\chi_k(t)$ respectively, $\left(W_1 + W_1^\prime \right) \chi_k = \sigma_{k} \eta_k$, $\eta_k^H\left(W_1 + W_1^\prime \right)  = \sigma_{k} \chi_k$, 
$\left\| \chi_k \right\| = \left\| \eta_k \right\| = 1$. Then from Lemma \ref{singular value derivative}, 

\begin{equation}
 \begin{aligned}
  \frac{\mathrm{d}}{\mathrm{d}t} \lambda_{k}\left(\left(W_1 + W_1^{\prime}\right)^H\left(W_1 + W_1^{\prime}\right)\right) &= \chi_k^H \left( \frac{\mathrm{d}}{\mathrm{d}t} \left(W_1 + W_1^{\prime}\right)^H \left(W_1 + W_1^{\prime}\right)\right) \chi_k \\&= 2 \Re\left(\chi_k^H \left(W_1 + W_1^{\prime}\right)^H \left(\frac{\mathrm{d}}{\mathrm{d}t} \left(W_1 + W_1^{\prime}\right)\right)\chi_k \right) , 
 \end{aligned}
\end{equation}

where 

\begin{equation}
 \begin{aligned}
  \frac{\mathrm{d}}{\mathrm{d}t} \left(W_1 + W_1^{\prime}\right) &=  M_2 W_1^{\prime} (\Sigma - W) - W_1^{\prime} (\Sigma - W) W_1^H W_1^{\prime} \\&+  W_1(\Sigma - W^H) W_1^{\prime H} R^{H - 1} R^{-1} W_1^{\prime} + R R^H M_2 W_1 (\Sigma - W^H) \\ &- a\Delta_{12}(W_1 + W_1^{\prime}) + 2 a W_2^{-1} \Delta_{23} W_2 W_1^{\prime} \\ 
  &= M_2\left(W_1 + W_1^{\prime}\right) \Sigma + \left( W_1 \Sigma W_1^{\prime H} - W_1^{\prime} \Sigma W_1^H \right) W_1^{\prime} \\&- M_2\left( \frac{M_1 + M_1^\prime}{2} M_2 \left( W_1 + W_1^{\prime} \right) + \frac{M_1 - M_1^\prime}{2} M_2 \left( W_1 - W_1^{\prime} \right) \right) \\&+ \left( M_1^\prime M_2 M_1 - M_1 M_2 M_1^\prime \right)W_1^{\prime} \\ & - W_1(\Sigma - W^H) W_1^{\prime H} \left(I - R^{H - 1} R^{-1}\right) W_1^{\prime}  \\ &- \left(I - R R^H\right) M_2 W_1 (\Sigma - W^H)  \\ &- a\Delta_{12}(W_1 + W_1^{\prime}) + 2 a W_2^{-1} \Delta_{23} W_2 W_1^{\prime} . 
 \end{aligned}
\end{equation}

Consider arbitrary $\chi \in \mathbb{F}^d$. Notice that $\left( W_1 \Sigma W_1^{\prime H} - W_1^{\prime} \Sigma W_1^H \right)$ is a skew-hermitian matrix: 

\begin{equation}
 \begin{aligned}
  &\Re\left(2\chi^H (W_1 + W_1^{\prime})^H \left( W_1 \Sigma W_1^{\prime H} - W_1^{\prime} \Sigma W_1^H \right) W_1^{\prime} \chi \right) \\=& \Re\left(\chi^H (W_1 + W_1^{\prime})^H \left( W_1 \Sigma W_1^{\prime H} - W_1^{\prime} \Sigma W_1^H \right) W_1^{\prime} \chi \right) \\ -& \Re\left(\chi^H W_1^{\prime H} \left( W_1 \Sigma W_1^{\prime H} - W_1^{\prime} \Sigma W_1^H \right) W_1 \chi \right) \\ -& \Re\left( \chi^H W_1^H \left( W_1 \Sigma W_1^{\prime H} - W_1^{\prime} \Sigma W_1^H \right) W_1 \chi \right) \\
  =& \Re\left(\chi^H (W_1 + W_1^{\prime})^H \left( - W_1 \Sigma W_1^{\prime H} + W_1^{\prime} \Sigma W_1^H \right) (W_1 - W_1^{\prime}) \chi\right) . 
 \end{aligned}
\end{equation}

From $\Sigma = \sigma_1(\Sigma) I$,  

\begin{equation}
  - W_1 \Sigma W_1^{\prime H} + W_1^{\prime} \Sigma W_1^H = \sigma_1(\Sigma) \left(W_1 + W_1^{\prime}\right) \left(W_1 - W_1^{\prime}\right)^H + \sigma_1(\Sigma) \left( M_1^{\prime} - M_1 \right) . 
\end{equation}

Likewise, 

\begin{equation}
 \begin{aligned}
  &\Re\left(2 \chi^H (W_1 + W_1^\prime)^H \left( M_1^\prime M_2 M_1 - M_1 M_2 M_1^\prime \right)W_1^{\prime}  \chi \right) \\ 
  =& \Re\left( \chi^H (W_1 + W_1^\prime)^H \left( M_1^\prime M_2 M_1 - M_1 M_2 M_1^\prime \right) \left( W_1^{\prime} - W_1 \right) \chi \right) . 
 \end{aligned}
\end{equation}

Thus 

\begin{equation}
 \begin{aligned}
  \frac{\mathrm{d}}{\mathrm{d}t} \sigma_{k}^2 &= 2\sigma_1(\Sigma) \sigma_{k}^2 \eta_k^H M_2 \eta_k + \sigma_1(\Sigma) \sigma_{k}^2 \chi_k^H \left(W_1 - W_1^{\prime}\right)^H \left(W_1 - W_1^{\prime}\right)\chi_k \\
  &+ \sigma_1(\Sigma)\sigma_{k} \Re\left(\eta_k^H \left( M_1^\prime - M_1 \right) \left(W_1 - W_1^{\prime}\right)\chi_k\right) \\
  &- \sigma_{k}^2 \eta_k^H M_2 (M_1 + M_1^\prime ) M_2 \eta_k - \sigma_{k} \Re\left(\eta_k^H M_2(M_1 - M_1^\prime) M_2 (W_1 - W_1^{\prime}) \chi_k \right) \\
  &+ \sigma_k \Re\left( \eta_k^H \left( M_1^\prime M_2 M_1 - M_1 M_2 M_1^\prime \right) \left( W_1^{\prime} - W_1 \right) \chi_k \right) \\ 
  &- 2\sigma_{k} \Re\left(\eta_k^H W_1(\Sigma - W^H) W_4 \left(R^{H} R - I\right) W_4^{H} \chi_k \right) \\ 
  &- 2\sigma_{k} \Re\left(\eta_k^H \left(I - R R^H\right) M_2 W_1 (\Sigma - W^H) \chi_k \right) \\ 
  &- 2a\sigma_{k}^2 \Re\left(\eta_k^H \Delta_{12} \eta_k\right) + 4a \sigma_{k} \Re\left(\eta_k^H W_2^{-1} \Delta_{23} W_2 W_1^{\prime} \chi_k\right) . 
 \end{aligned}
\end{equation}

\subsubsection{Gradient Descent}\label{section: hermitian main term, gd}

For gradient descent, we study $\lambda_{\min}\left(\left(W_1 + W_1^\prime \right)^H\left(W_1 + W_1^\prime \right)\right) = \sigma_{\min}^2 \left(W_1 + W_1^\prime \right)$. To address the dynamics: 

\begin{equation}
 \begin{aligned}
  &\left(W_1(t+1) + W_1^{\prime}(t+1)\right) \\=& W_1(t) + W_1^{\prime}(t) \\+& \eta\left[ \sigma_1(\Sigma) M_2(t) - M_2(t)\frac{M_1(t) + M_1^\prime(t)}{2} M_2(t) \right] \left( W_1(t) + W_1^{\prime}(t) \right) \\+& \eta \left( M_1^\prime(t) M_2(t) M_1(t) - M_1(t) M_2(t) M_1^\prime(t) \right) W_1^{\prime}(t)\\+& \eta \sigma_1(\Sigma)\left( W_1(t) W_1^{\prime}(t)^{H} - W_1^{\prime}(t) W_1(t)^H \right) W_1^{\prime}(t) + \eta E_1(t) , 
 \end{aligned}
\end{equation}

where the error term is bounded by 

\begin{equation}
 \begin{aligned}
  \| E_1(t) \|_{op} &\le \frac{1}{2} \max_{j\in[1,4]\cap\mathbb{N}^*} \|W_j(t)\|_{op}^4 \left\|W_1(t) - W_1^\prime(t)\right\|_{op} \left\|M_1(t) - M_1^\prime(t)\right\|_{op} \\&+ \left( \left\|R(t)^{H} R(t) - I\right\|_{op} + \left\|I - R(t) R(t)^{H}\right\|_{op} \right)\max_{j\in[1,4]\cap\mathbb{N}^*} \|W_j(t)\|_{op}^3 \| \Sigma - W(t)\|_{op} \\&+ a e_\Delta(t) \left( \left\|W_1(t) + W_1^\prime(t)\right\|_{op} + 2\left\|R(t)\right\|_{op} \left\| W_2(t)^{-1} \right\|_{op} \max_{j\in[1,4]\cap\mathbb{N}^*} \|W_j(t)\|_{op}^2 \right) \\&+ \eta O\left(\left[ \max_{j \in [1,4]\cap \mathbb{N}^*} \left\| W_j(t) \right\|_{op} \| \Sigma - W(t)\|_F + a e_\Delta(t) \left\|W_2(t)^{-1}\right\|_{op} \right] \right. \\ &\cdot \left. \max_{j \in [1,4]\cap \mathbb{N}^*} \left\| W_j(t) \right\|_{op}^2 \cdot \|W_2(t+1)^{-1}\|_{op} \cdot \max_{j \in [1,4]\cap \mathbb{N}^*}\|\nabla_{W_j} \mathcal{L}(t)\|_F \right) . 
 \end{aligned}
\end{equation}

Follow the tricks in Lemma \ref{maximum and minimum singular values, general, discrete}, 

\begin{equation}
 \begin{aligned}
  &\lambda_{\min}\left( \left(W_1(t+1) + W_1^{\prime}(t+1)\right)^H \left(W_1(t+1) + W_1^{\prime}(t+1)\right) \right) \\ \ge& \lambda_{\min}\left( \left(W_1(t) + W_1^{\prime}(t)\right)^H \left( I + \eta \left[ \sigma_1(\Sigma) M_2(t) - M_2(t)\frac{M_1(t) + M_1^\prime(t)}{2} M_2(t) \right] \right)^2 \left(W_1(t) + W_1^{\prime}(t)\right) \right) \\+& \eta \|E_2(t)\|_{op} + \eta^2 O\left( \left\|\left(W_1(t+1) + W_1^{\prime}(t+1)\right) - \left(W_1(t) + W_1^{\prime}(t)\right) \right\|_{op}^2 \right) , 
 \end{aligned}
\end{equation}

where 

\begin{equation}
 \begin{aligned}
  \|E_2(t)\|_{op} &= \sigma_{\min}\left(W_1(t+1) + W_1^{\prime}(t+1)\right) \\&\cdot \left[\| E_1(t) \|_{op}+ \|W_2(t)\|_{op}^2 \left\|M_1(t) + M_1^\prime(t)\right\|_{op} \left\|M_1(t) - M_1^\prime(t)\right\|_{op} \left\|W_1(t) - W_1^\prime(t)\right\|_{op} \right] . 
 \end{aligned}
\end{equation}

%% file: subfiles/appendix/appendix_7_staged,gf.tex
\section{Convergence under Gradient Flow, staged analysis}\label{section: convergence, staged analysis, gf}

In order to present the proof more clearly, we state the complete proof of convergence under Random Gaussian Initialization \ref{subsection: random gaussian initialization} and gradient flow, before tackling gradient descent. 

At the beginning we assume (\ref{random gaussian initialization, conclusions}) holds. (For the complex case, it holds with high probability $1-\delta$; for the real case, it holds with probability $\frac{1}{2}(1-\delta)$. ) We omit the confidence level $\delta$ in $f_1(\delta) = O(\frac{1}{\delta})$ and $f_2(\delta) = O(\frac{1}{\delta^5})$ for simplicity. 

\subsection{Stage 1: alignment stage}

In this section, we set $\epsilon \le \frac{\sigma_1^{1/4}(\Sigma)}{2f_1\sqrt{d}}$, $a \ge 2^5 f_1^{20} f_2 d^{13} \sigma_1(\Sigma) b$, where $b \ge 2^4 \ln(4f_1d) + \ln f_2$. 

Without loss of generality, $f_1 \ge 2$, and for simplicity we can further relax $f_2$ appearing in the lower bounds to $f_2 \ge f_1^6$ (now $f_2 = O\left(\frac{1}{\delta^6}\right)$). 

\begin{theorem}\label{stage 1: alignment stage}

At $T_1 = \frac{1}{32 f_1^{14} f_2 d^{10} \epsilon^2 \sigma_1(\Sigma)}$, the following conclusions hold: 

\begin{equation}
 \begin{aligned}
  \left.\sigma_{\min}\left(W_1 + W_1^{\prime}\right)\right|_{t=T_1} &\ge \frac{\epsilon}{2 f_1^3 f_2 d^{9/2}} \\
  e_{\Delta}(T_1) &\le 2\sqrt{3} f_1^2 d^{3/2} \epsilon^2 \exp{\left( - \frac{a}{32f_1^{20}f_2d^{13} \sigma_1(\Sigma)}\right)} \\
  \max_{j,k} |\sigma_k(W_j(T_1))| &\le (1+2^{-21}) f_1\sqrt{d}\epsilon \\ 
  \min_{j,k} |\sigma_k(W_j(T_1))| &\ge (1-2^{-17}) \frac{\epsilon}{f_1\sqrt{d}} . 
 \end{aligned}
\end{equation}

\end{theorem}

This section proves the theorem above by following Lemmas and Corollaries. 

\begin{lemma}\label{alignment stage, bound of max and min singular values}

Maximum and minimum singular value bound of weight matrices in alignment stage. 

For $t \in \left[0,\frac{1}{16 f_1^4 d^2 \epsilon^2 \sigma_1(\Sigma)} \right]$, 

\begin{equation}
  \min_{j,k}\sigma_k(W_j) \ge \frac{\epsilon}{f_1\sqrt{d}} - 16 f_1^3 d^{3/2} \epsilon^3 \sigma_1(\Sigma) t ,\,   \max_{j,k}\sigma_k(W_j) \le \frac{f_1\sqrt{d}\epsilon}{\sqrt{1 - 4 f_1^2 d\epsilon^2 \sigma_1(\Sigma) t}} . 
\end{equation}
\end{lemma}

\begin{proof}

For $t\ge 0$ such that $\max_{j,k} \sigma_k(W_j) \le 2f_1\sqrt{d} \epsilon \le \sigma_1^{1/4}(\Sigma)$, 

\begin{equation}
  \max_{j} \left\| \nabla_{W_j} \mathcal{L}_{\rm ori} \right\|_{op} \le \max_{j,k} |\sigma_k(W_j)|^{3} \left( \sigma_1(\Sigma) + \max_{j,k} |\sigma_k(W_j)|^{4} \right) \le 2 \sigma_1(\Sigma) \max_{j,k} |\sigma_k(W_j)|^{3} . 
\end{equation}

By invoking Theorem \ref{maximum and minimum singular values are irrelevant of the regularization term}, 

\begin{equation}
 \begin{aligned}
  \frac{\mathrm{d} \max_{j,k} \sigma_k^2(W_j)}{\mathrm{d}t} &\le 4 \max_{j,k} |\sigma_k(W_j)|^{4} \sigma_1(\Sigma) \\
  \frac{\mathrm{d} \min_{j,k} \sigma_k^2(W_j)}{\mathrm{d}t} &\ge - 4 \min_{j,k} \left| \sigma_k(W_j) \right| \max_{j,k} |\sigma_k(W_j)|^{3} \sigma_1(\Sigma) . 
 \end{aligned}
\end{equation}

By solving the differential inequality,  

\begin{equation}
  \max_{j,k} \sigma_k|W_j| \le \frac{\max_{j,k} \sigma_k|W_j(0)|}{\sqrt{1 - 4 \sigma_1(\Sigma) \max_{j,k} \sigma_k|W_j(0)|^2 \cdot t}} \le \frac{f_1\sqrt{d}\epsilon}{\sqrt{1 - 4 f_1^2 d\epsilon^2 \sigma_1(\Sigma) t}},\, t \in \left[0,\frac{3}{16 f_1^2 d \epsilon^2 \sigma_1(\Sigma)}\right] . 
\end{equation}

\begin{equation}
  \min_{j,k} |\sigma_k(W_j)| \ge \frac{\epsilon}{f_1 \sqrt{d}} - 16 f_1^3 d^{3/2} \epsilon^3 \sigma_1(\Sigma) t,\,t\in \left[0,\frac{1}{16 f_1^4 d^2 \epsilon^2 \sigma_1(\Sigma)} \right] . 
\end{equation}

This completes the proof.

\end{proof}

Notice that  

\begin{equation}
 \begin{aligned}
  \max_{j,k} |\sigma_k(W_j(t\le T_1))| &\le \frac{f_1\sqrt{d} \epsilon}{\sqrt{1-\frac{1}{8f_1^{12}f_2}}} \le (1+2^{-21}) f_1\sqrt{d}\epsilon \\ 
  \min_{j,k} |\sigma_k(W_j(t\le T_1))| &\ge \left(1-\frac{1}{2f_1^{10}f_2}\right) \cdot \frac{\epsilon}{f_1\sqrt{d}} \ge (1-2^{-17}) \frac{\epsilon}{f_1\sqrt{d}} . 
 \end{aligned}
\end{equation}

\begin{corollary}\label{stage 1, e_delta}

Balanced term error in alignment stage.  

For $t \in \left[0,T_1 \right]$, 

\begin{equation}
 \begin{aligned}
  e_{\Delta}(t) &\le 2\sqrt{3} f_1^2 d^{3/2} \epsilon^2 \exp\left( - \frac{a \epsilon^2}{ f_1^6 d^3} t \right) . 
 \end{aligned}
\end{equation}

Specially, at $t = T_1$, 

\begin{equation}
 \begin{aligned}
  e_{\Delta}(T_1) &\le 2\sqrt{3} f_1^2 d^{3/2} \epsilon^2 \exp{\left( - \frac{a}{32f_1^{20}f_2d^{13} \sigma_1(\Sigma)}\right)} \le \sqrt{3} \cdot 2^{-31} f_1^{-14} f_2^{-1} d^{-29/2} \epsilon^2 . 
 \end{aligned}
\end{equation}

\end{corollary}

\begin{proof}

By simply combining Theorem \ref{regularization term, convergence bound} and Lemma \ref{alignment stage, bound of max and min singular values}.

\end{proof}

\begin{corollary}\label{Main term at the end of alignment stage, gf}

Main term at the end of alignment stage. 

At $t=T_1$, 

\begin{equation}
  \left.\sigma_{\min}\left(W_1 + W_1^{\prime}\right)\right|_{t=T_1} \ge \frac{\epsilon}{2 f_1^3 f_2 d^{9/2}} . 
\end{equation}

\end{corollary}

\begin{proof}

For simplicity, denote $\Delta_{X}(t) = X(t) - X(0)$ for arbitrary $X$. Note: $\Delta_{X^H} = \Delta_{X}^H$. 

At $t = T_1$, 

\begin{equation}
 \begin{aligned}
  \| \Delta_{W}(T_1)\|_{op} 
  =& \left\| \int_{0}^{T_1} \sum_{j=1}^{4} \left[ W_{\prod_L,j+1}(t^\prime) W_{\prod_L,j+1}(t^\prime)^H \left(\Sigma - W(t^\prime)\right) W_{\prod_R,j-1}^H (t^\prime) W_{\prod_R,j-1}(t^\prime) \right] \mathrm{d}t^\prime \right\|_{op} \\
  \le& \int_{0}^{T_1} \sum_{j=1}^{4} \left\| W_{\prod_L,j+1}(t^\prime) W_{\prod_L,j+1}(t^\prime)^H \left(\Sigma - W(t^\prime)\right) W_{\prod_R,j-1}^H (t^\prime) W_{\prod_R,j-1}(t^\prime) \right\|_{op} \mathrm{d}t^\prime \\
  \le& \int_{0}^{T_1} \sum_{j=1}^{4} \left(\left\| \Sigma \right\|_{op} + \left\| W(t^\prime) \right\|_{op} \right) \left( \prod_{k \in [1,4] \cap \mathbb{N}^*,\,k\ne j} \left\|W_i(t^\prime) \right\|_{op}^2 \right) \mathrm{d}t^\prime \\
  \le& \int_{0}^{T_1} 4\cdot 2 \sigma_1(\Sigma) \cdot \left( \left(1+2^{-21}\right) f_1 \sqrt{d} \epsilon \right)^6 \mathrm{d}t^\prime \\
  \le& 8\left(1+2^{-18}\right) f_1^6 d^{3} \epsilon^6 \sigma_1(\Sigma) T_1 = \left(1+2^{-18}\right) \cdot \frac{1}{4} f_1^{-8} f_2^{-1} d^{-7} \epsilon^4 . 
 \end{aligned}
\end{equation}

Thus 

\begin{equation}
 \begin{aligned}
  \left\| \Delta_{W^H W}(T_1) \right\|_{op} &= \left\| \frac{1}{2}\left[\left(W(T_1) + W(0)\right)^H \Delta_{W}(T_1) + \Delta_{W}(T_1)^H \left(W(T_1) + W(0)\right) \right] \right\|_{op} \\ 
  &\le  \left(\left\|W(T_1)\right\|_{op} + \left\|W(0)\right\|_{op}\right) \left\| \Delta_{W}(T_1) \right\|_{op} \\ 
  &\le \left[1 + \left(1 + 2^{-21}\right)^4 \right] f_1^4 d^2 \epsilon^4 \cdot \left\| \Delta_{W}(T_1) \right\|_{op} = (1+2^{-17})\cdot\frac{1}{2} f_1^{-4} f_2^{-1} d^{-5} \epsilon^8 . 
 \end{aligned}
\end{equation}

From Corollary \ref{stage 1, e_delta}, 

\begin{equation}
 \begin{aligned}
  &\left\| \left(W_1(T_1)^H W_2(T_1)^H W_2(T_1) W_1(T_1) \right)^2 - W(T_1)^H W(T_1) \right\|_{op} \\ 
  =& \left\|W_1(T_1)^H W_2(T_1)^H M_{\Delta1234}(T_1) W_2(T_1) W_1(T_1) \right\|_{op} \\ 
  \le& \left\|W_1(T_1)^H W_2(T_1)^H \right\|_{op} \left\| M_{\Delta1234}(T_1) \right\|_{op}
  \left\| W_2(T_1) W_1(T_1) \right\|_{op} \\ 
  \le& \left( \left(1+2^{-21}\right) f_1 \sqrt{d} \epsilon \right)^{4} \cdot \sqrt{6}\left( \left(1+2^{-21}\right) f_1 \sqrt{d} \epsilon \right)^{2} \cdot e_{\Delta}(T_1) \\
  \le& \sqrt{6}(1+2^{-18}) f_1^6 d^3 \epsilon^6 e_{\Delta}(T_1) \le 2^{-28} f_1^{-8} f_2^{-16} d^{-23/2} \epsilon^8 . 
 \end{aligned}
\end{equation}

Thus 

\begin{equation}
 \begin{aligned}
  &\left\| \left(W_1(T_1)^H W_2(T_1)^H W_2(T_1) W_1(T_1) \right)^2 - W(T_0)^H W(T_0) \right\|_{op} \\ 
  \le& \left\| \left(W_1(T_1)^H W_2(T_1)^H W_2(T_1) W_1(T_1) \right)^2 - W(T_1)^H W(T_1) \right\|_{op} + \left\| \Delta_{W^H W}(T_1) \right\|_{op} \\ 
  \le& (1+2^{-16})\cdot\frac{1}{2} f_1^{-4} f_2^{-1} d^{-5} \epsilon^8 . 
 \end{aligned}
\end{equation}

From Lemma \ref{error bound, sqrt}, 

\begin{equation}
 \begin{aligned}
  &\left\| W_1(T_1)^H W_2(T_1)^H W_2(T_1) W_1(T_1) - \left( W(T_0)^H W(T_0) \right)^{1/2} \right\|_{op} \\
  \le& \frac{\left\| \left(W_1(T_1)^H W_2(T_1)^H W_2(T_1) W_1(T_1) \right)^2 - W(T_0)^H W(T_0) \right\|_{op}}{2\sqrt{\lambda_{\min}\left( W(T_0)^H W(T_0) \right) - \left\|  \left(W_1(T_1)^H W_2(T_1)^H W_2(T_1) W_1(T_1) \right)^2 - W(T_0)^H W(T_0)\right\|_{op}}} \\
  \le& \frac{(1+2^{-16})\cdot\frac{1}{2} f_1^{-4} f_2^{-1} d^{-5} \epsilon^8}{2\sqrt{\left( \frac{\epsilon}{f_1 \sqrt{d}} \right)^8 - (1+2^{-16})\cdot\frac{1}{2} f_1^{-4} f_2^{-1} d^{-5} \epsilon^8}} \le 0.27 f_2^{-1} d^{-3} \epsilon^4 . 
 \end{aligned}
\end{equation}

By (\ref{subsection: random gaussian initialization}), 

\begin{equation}
 \begin{aligned}
  &\sigma_{\min} \left( W_1(T_1)^H W_2(T_1)^H W_2(T_1) W_1(T_1) + W(T_1)^H\right) \\ 
  \ge& \sigma_{\min} \left( \left( W(T_0)^H W(T_0) \right)^{1/2} + W(0)^H \right) \\ -& \left\| W_1(T_1)^H W_2(T_1)^H W_2(T_1) W_1(T_1) - \left( W(T_0)^H W(T_0) \right)^{1/2} \right\|_{op} - \left\| \Delta_W(T_1) \right\|_{op} \\ 
  \ge& f_2^{-1} d^{-3} \epsilon^4 - 0.27 f_2^{-1} d^{-3} \epsilon^4 - \left(1+2^{-18}\right) \cdot \frac{1}{4} f_1^{-8} f_2^{-1} d^{-7} \epsilon^4 \\ 
  \ge& 0.72 f_2^{-1} d^{-3} \epsilon^4 , 
 \end{aligned}
\end{equation}

which further gives 

\begin{equation}
 \begin{aligned}
  &\left. \sigma_{\min} \left(W_1 + W_1^\prime \right) \right|_{t=T_1} \\
  =& \sigma_{\min} \left( \left(W_1(T_1)^H W_2(T_1)^H W_2(T_1)\right)^{-1}\left( W_1(T_1)^H W_2(T_1)^H W_2(T_1) W_1(T_1) + W(T_1)^H\right) \right) \\ 
  \ge& \left(\frac{1}{\max_{j,k}|\sigma_k(W_j(T_1))|}\right)^3 \cdot \sigma_{\min} \left( W_1(T_1)^H W_2(T_1)^H W_2(T_1) W_1(T_1) + W(T_1)^H\right) \\ 
  \ge& \frac{\epsilon}{2 f_1^3 f_2 d^{9/2}} . 
 \end{aligned}
\end{equation}

\end{proof}

\subsection{Stage 2: saddle avoidance stage}

In this stage, we further assume $a \ge 32 f_1^{20} f_2 d^{13}\sigma_1(\Sigma) \left( 5\ln\left( \frac{\sigma_1^{1/4}(\Sigma)}{\epsilon} \right) + \frac{281}{8} \ln d + 23\ln(4f_1) + 7\ln f_2 \right)$, while $\frac{\epsilon}{\sigma_1^{1/4}(\Sigma)} \le \frac{1}{32 f_1^5 f_2 d^{53/8}} $. From Lemma \ref{regularization, total} and Theorem \ref{stage 1: alignment stage}, 

\begin{equation}\label{ineq, e, stage 2}
 \begin{aligned}
  e_{\Delta}(t \in [T_1,+\infty)) &\le e_{\Delta}(T_1) \le 2\sqrt{3} f_1^2 d^{3/2} \epsilon^2 \exp{\left( - \frac{a}{32f_1^{20}f_2d^{13} \sigma_1(\Sigma)}\right)} \\
  &\le \sqrt{3} \cdot 2^{-45} f_1^{-21} f_2^{-7} d^{-269/8} \epsilon^7 \sigma_1^{-5/4}(\Sigma)  . 
 \end{aligned}
\end{equation}

Moreover, $a \ge 32f_1^{20}f_2d^{13} \sigma_1(\Sigma) b$, where $b - \ln b \ge 3\ln\left(\frac{\sigma_1^{1/4}(\Sigma)}{\epsilon}\right) + \frac{303}{8}\ln d + 37 \ln(2f_1) + 6 \ln f_2 $. Thus 

\begin{equation}\label{ineq, a e, stage 2}
 \begin{aligned}
  a e_{\Delta}(t \in [T_1,+\infty)) &\le a e_\Delta(T_1) \le 2^6 \sqrt{3} f_1^{22} f_2 d^{29/2} \epsilon^2 \sigma_1(\Sigma) \exp(-(b- \ln b)) \\
  &\le \sqrt{3} \cdot 2^{-31} f_1^{-15} f_2^{-5} d^{-187/8} \epsilon^5 \sigma_1^{1/4}(\Sigma) . 
 \end{aligned}
\end{equation}

\begin{theorem}\label{stage 2: saddle avoidance stage, gf}

At $T_1 + T_2$, $T_2 = \frac{32 f_1^6 f_2^2 d^9}{\sigma_1(\Sigma) \epsilon^2}$, the following conclusions hold: 

\begin{equation}
 \begin{aligned}
  \left \| W_1(T_1 + T_2) - W_1^\prime(T_1 + T_2) \right \|_F &\le 3 f_1 d \epsilon \\
  \sigma_{\min}(W_1 + W_1^\prime)(T_1 + T_2) &\ge  2^{3/4} \sigma_1^{1/4}(\Sigma) . 
 \end{aligned}
\end{equation}
    
\end{theorem}

\begin{lemma}\label{bound of w_j op}

Bound of operator norms throughout the optimization process. 

For $t\in[0,+\infty)$, 

\begin{equation}
 \begin{aligned}
  \|\Sigma - W(t)\|_{op} \le \|\Sigma - W(t)\|_{F} &\le 1.01 \sqrt{d} \sigma_1(\Sigma) \\ 
  \|W\|_{op} \le \|W\|_{F} &\le 3 \sqrt{d} \sigma_1(\Sigma) \\ 
  \max_{j} \|W_j\|_{op} \le \max_{j} \|W_j\|_{F} &\le \sqrt{2} d^{1/8} \sigma_1^{1/4}(\Sigma) . 
 \end{aligned}
\end{equation}

\end{lemma}

\begin{proof}

For $t\in[0,T_1]$, the result is obvious from Theorem \ref{stage 1: alignment stage} and Lemma \ref{alignment stage, bound of max and min singular values}. 

For $t\in(T_1,+\infty)$: from Lemma \ref{L ori non-increasing}, 

\begin{equation}
  \|\Sigma - W(t)\|_{op} \le \|\Sigma - W(t)\|_{F} \le \|\Sigma - W(0)\|_{F} \le \|\Sigma\|_{F} + \| W(0)\|_{F} \le \sqrt{2d} \sigma_1(\Sigma) . 
\end{equation}

Giving 

\begin{equation}\label{proof for bound of w_j op, contradiction}
  \|W(t)\|_{op} \le \|W(t)\|_{F} \le \|\Sigma - W(t)\|_{F} + \|\Sigma\|_{F} \le 3 \sqrt{d} \sigma_1(\Sigma) . 
\end{equation}

For the last inequality, prove by contradiction. 

Suppose $\max_{j} \|W_j\|_{op} \ge \sqrt{2} d^{1/8} \sigma_1^{1/4}(\Sigma)$, then by invoking Corollary \ref{stage 1, e_delta}, 

\begin{equation}
  e_{\Delta} (t) \le e_{\Delta} (T_1) \le \sqrt{3} \cdot 2^{-15} f_1^{-14} f_2^{-16} d^{-29/2} \epsilon^2 \le 2^{-15} \max_{j} \|W_j\|_{op}^2 . 
\end{equation}

Thus for $t > T_1$, 

\begin{equation}
 \begin{aligned}
  \|W\|_{op}^2 &= \left\|WW^H\right\|_{op} = \left\|W_4 W_3 W_2 W_1 W_1^H W_2^H W_3^H W_4^H \right\|_{op} \\
  &\ge \left\|W_4 W_4^H\right\|_{op} - \left\|W_4 W_3 W_2 \Delta_{12} W_2^H W_3^H W_4^H \right\|_{op} \\&- \left\|W_4 W_3 \Delta_{23} W_2 W_2^H W_3^H W_4^H \right\|_{op} - \left\|W_4 W_3 W_2 W_2^H \Delta_{23} W_3^H W_4^H \right\|_{op} \\&- \left\|W_4 \Delta_{34} \left(W_3 W_3^H\right)^2 W_4^H \right\|_{op} - \left\|W_4 W_3 W_3^H \Delta_{34} W_3 W_3^H W_4^H \right\|_{op} - \left\|W_4 \left(W_3 W_3^H\right)^2 \Delta_{34} W_4^H \right\|_{op} \\ 
  &\ge \left( \max_{j} \|W_j\|_{op}^2 - 3 e_{\Delta} \right)^4 - 6 e_{\Delta} \max_{j} \|W_j\|_{op}^6 > 15 \sqrt{d} \sigma_1(\Sigma) . 
 \end{aligned}
\end{equation}

which contradicts inequality (\ref{proof for bound of w_j op, contradiction}). This completes the proof. 

\end{proof}

\begin{lemma}\label{bound of w23 inv op, and relevant terms}

Bound of $\left \| W_2^{-1} \right \|_{op}$, $\left \| W_3^{-1} \right \|_{op}$, and relevant terms. 

For $t\in[T_1, T_1 + T_2]$, 

\begin{equation}\label{ineq: inverse op norm are bounded}
  \max\left(\left\| W_2^{-1} (t) \right\|_{op},\, \left\| W_3^{-1} (t) \right\|_{op} \right) \le 128 f_1^{6} f_2^{2} d^{77/8} \epsilon^{-2} \sigma_1^{1/4}(\Sigma) , 
\end{equation}

\begin{equation}
  \max\left( e_\Delta (t) \left\| W_2^{-1} (t) \right\|_{op}^2,\, e_\Delta (t) \left\| W_3^{-1} (t) \right\|_{op}^2 \right) \le \sqrt{3} \cdot 2^{-31} f_1^{-9} f_2^{-3} d^{-115/8} \epsilon^{3} \sigma_1^{-3/4}(\Sigma) . 
\end{equation}

\end{lemma}

\begin{proof}

We begin with the time derivative of $W_2^{-1}$ and $W_3^{-1}$: 

\begin{equation}
 \begin{aligned}
  \frac{\mathrm{d}W_2^{-1}}{\mathrm{d}t} &= - R W_4^H (\Sigma - W) W_1^H W_2^{-1} - a \Delta_{12} W_2^{-1} + a W_2^{-1} \Delta_{23} \\ 
  \frac{\mathrm{d}W_3^{-1}}{\mathrm{d}t} &= - W_3^{-1} W_4^H (\Sigma - W) W_1^H R^{H-1} - a \Delta_{23} W_3^{-1} + a W_3^{-1} \Delta_{34} . 
 \end{aligned}
\end{equation}

From $\frac{\mathrm{d}}{\mathrm{d}t} \left\|M\right\|_{op} \le \left\|\frac{\mathrm{d}}{\mathrm{d}t} M \right\|_{op}$ (this in equality is from triangular inequality and standard calculus analysis), 

\begin{equation}
 \begin{aligned}
  \frac{\mathrm{d}}{\mathrm{d}t} \left\|W_2^{-1}\right\|_{op} &\le \left\| R \right\|_{op} \left\| W_4 \right\|_{op} \left\| \Sigma - W \right\|_{op} \left\| W_1^H W_2^{-1} \right\|_{op} \\&+ a \left\| \Delta_{12} \right\|_{op} \left\| W_2^{-1} \right\|_{op} + a \left\| W_2^{-1} \right\|_{op} \left\| \Delta_{23} \right\|_{op} \\ 
  \frac{\mathrm{d}}{\mathrm{d}t} \left\|W_3^{-1}\right\|_{op} &\le \left\| W_3^{-1} W_4^H \right\|_{op} \left\| \Sigma - W \right\|_{op} \left\| W_1 \right\|_{op} \left\| R \right\|_{op} \\&+ a \left\| \Delta_{23} \right\|_{op} \left\| W_3^{-1} \right\|_{op} + a \left\| W_3^{-1} \right\|_{op} \left\| \Delta_{34} \right\|_{op} . 
 \end{aligned}
\end{equation}

From Lemma \ref{bound of w_j op} and 

\begin{equation}
 \begin{aligned}
  \left\| R \right\|_{op} &\le \sqrt{1 + \frac{1}{\sigma_{\min}^2(W_2)} \cdot \| \Delta_{23} \|_{op}} \\ 
  \left\| R^{-1} \right\|_{op} &\le \sqrt{1 + \frac{1}{\sigma_{\min}^2(W_3)} \cdot \| \Delta_{23} \|_{op} } \\
  \left\| W_1^H W_2^{-1} \right\|_{op} &= \sqrt{\left\| W_2^{H-1} W_1 W_1^H W_2^{-1} \right\|_{op}} = \sqrt{\left\| I + W_2^{H-1} \Delta_{12} W_2^{-1} \right\|} \\&\le \sqrt{1+e_\Delta\left\| W_2^{-1} \right\|_{op}^2} \\
  \left\| W_3^{-1} W_4^H \right\|_{op} &= \sqrt{\left\| W_3^{-1} W_4^H W_4 W_3^{H-1} \right\|_{op}} = \sqrt{\left\| I - W_3^{-1} \Delta_{34} W_3^{H-1} \right\|} \\&\le \sqrt{1+e_\Delta\left\| W_3^{-1} \right\|_{op}^2} . 
 \end{aligned}
\end{equation}

Further we have 

\begin{equation}
 \begin{aligned}
  \frac{\mathrm{d}}{\mathrm{d}t} \left\|W_2^{-1}\right\|_{op} &\le 2\sqrt{2} \left( 1 + e_\Delta \left\|W_2^{-1}\right\|_{op}^2 \right) d^{5/8} \sigma_1^{5/4}(\Sigma) + \sqrt{2} a e_\Delta \left\|W_2^{-1}\right\|_{op} \\ 
  \frac{\mathrm{d}}{\mathrm{d}t} \left\|W_3^{-1}\right\|_{op} &\le 2\sqrt{2} \left( 1 + e_\Delta \left\|W_3^{-1}\right\|_{op}^2 \right) d^{5/8} \sigma_1^{5/4}(\Sigma) + \sqrt{2} a e_\Delta \left\|W_3^{-1}\right\|_{op} . 
 \end{aligned}
\end{equation}

Combine with (\ref{ineq, e, stage 2}) and (\ref{ineq, a e, stage 2}), for $t \ge T_1$ such that (\ref{ineq: inverse op norm are bounded}) holds, 

\begin{equation}
 \begin{aligned}
  &\max\left( \frac{\mathrm{d}}{\mathrm{d}t}  \left\|W_2^{-1}\right\|_{op} ,\, \frac{\mathrm{d}}{\mathrm{d}t} \left\|W_3^{-1}\right\|_{op} \right) \\
  \le& 2\sqrt{2} (1+\sqrt{3}\cdot 2^{-31}) d^{5/8} \sigma_1^{5/4}(\Sigma) + 2^{-22} f_1^{-9} f_2^{-3} d^{-55/4} \epsilon^{3} \sigma_1^{1/2}(\Sigma) \\
  \le& 2\sqrt{2} (1 + 2^{-20}) d^{5/8} \sigma_1^{5/4}(\Sigma) . 
 \end{aligned}
\end{equation}

From Theorem \ref{stage 1: alignment stage}, $\max\left( \left\|W_2(T_1)^{-1}\right\|_{op} ,\, \left\|W_3(T_1)^{-1}\right\|_{op} \right) \le \frac{1}{\min_{j,k}|\sigma_k(W_j(T_1))|} \le \frac{f_1\sqrt{d}}{(1-2^{-17})\epsilon}$, then the proof of the first inequality is completed via integration during the time interval $[T_1, T_1 + T_2]$. The second inequality follows immediately. 

\end{proof}

\begin{remark}

This Lemma verifies that $W_{2,3}^{-1}$ are bounded (consequently $W_{2,3}$ are full rank), then $R$ is well defined throughout this stage. For $t > T_1 + T_2$, further analysis shows that the minimum singular values of $W_2$ and $W_3$ are lower bounded by $\Omega(\sigma_1^{1/4}(\Sigma))$. 

\end{remark}

\begin{lemma}\label{stage 2, skew-hermitian error}

Skew-hermitian error. 

For $t\in[T_1, T_1+T_2]$, 

\begin{equation}
  \left \| W_1 - W_1^\prime \right \|_F \le 3 f_1 d \epsilon . 
\end{equation}

\end{lemma}

\begin{proof}

From section \ref{section: skew-hermitian error term, gf}, 

\begin{equation}
 \begin{aligned}
  \frac{\mathrm{d}}{\mathrm{d}t} \left \| W_1 - W_1^\prime \right \|_F^2 
  &= -2 \sigma_1(\Sigma) \mathrm{tr}\left( \left( W_1 - W_1^\prime \right)^H M_2 \left(W_1 - W_1^\prime \right) \right) \\&- \sigma_1(\Sigma) \mathrm{tr}\left( \left( W_1^\prime W_1^H - W_1 W_1^{\prime H} \right) \left( W_1^\prime W_1^H - W_1 W_1^{\prime H} \right)^H \right) \\
  &- \mathrm{tr}\left( M_2 \left( M_1^\prime + M_1\right)M_2\left(W_1 - W_1^\prime\right) \left( W_1 - W_1^\prime \right)^H \right) \\
  &- \mathrm{tr}\left( M_2 \left( M_1^\prime - M_1\right)M_2\left(W_1^\prime + W_1\right) \left( W_1 - W_1^\prime \right)^H\right) \\
  &+ 2 \mathrm{tr}\left( \left[ - M_1^\prime M_2 M_1 + M_1 M_2 M_1^\prime \right] W_1^\prime \left( W_1 - W_1^\prime \right)^H\right) \\
  &+ 2 \Re\left(\mathrm{tr}\left( \left[ W_1 (\Sigma - W^H) W_4 \left( R^{H} R - I \right) W_4^H \right] \left( W_1 - W_1^\prime \right)^H\right)\right) \\
  &+ 2 \Re\left(\mathrm{tr}\left( \left[ \left(I - R R^H \right) W_2^H W_2 W_1 (\Sigma - W^H) \right] \left( W_1 - W_1^\prime \right)^H\right)\right) \\
  &- 2a \Re\left(\mathrm{tr}\left(\Delta_{12} \left(W_1 - W_1^\prime \right) \left( W_1 - W_1^\prime \right)^H \right)\right) \\
  &- 4a \Re\left(\mathrm{tr}\left( W_2^{-1} \Delta_{23} W_2 W_1^\prime \left( W_1 - W_1^\prime \right)^H \right)\right) . 
 \end{aligned}
\end{equation}

Note: $ - M_1^\prime M_2 M_1 + M_1 M_2 M_1^\prime = \frac{1}{2} \left[\left(M_1 - M_1^\prime \right)M_2\left(M_1 + M_1^\prime \right) - \left(M_1 + M_1^\prime \right)M_2\left(M_1 - M_1^\prime\right)\right]$. 

From Lemma \ref{bound of w23 inv op, and relevant terms}, for $t\in[T_1, T_1+T_2]$, 

\begin{equation}
 \begin{aligned}
  \max\left(\left\| R^{H} R - I \right\|_{op},\, \left\| I - R R^H \right\|_{op} \right) &\le e_\Delta\left\| W_2^{-1} \right\|_{op}^2 \\&\le \sqrt{3} \cdot 2^{-31} f_1^{-9} f_2^{-3} d^{-115/8} \epsilon^{3} \sigma_1^{-3/4}(\Sigma) , 
 \end{aligned}
\end{equation}

\begin{equation}
 \begin{aligned}
  \left\| M_1 - M_1^\prime \right\|_{op} &\le \sqrt{6} \cdot \frac{\max_{j,k} \sigma_{k}^2(W_j)}{ \sigma_{\min}^2(W_2)} e_{\Delta} \\
  &\le 2^{-27} f_1^{-9} f_2^{-3} d^{-113/8} \epsilon^{3} \sigma_1^{-1/4}(\Sigma) , 
 \end{aligned}
\end{equation}

\begin{equation}
 \begin{aligned}
  \left\|M_2 - \frac{M_1 + M_1^\prime}{2} \right\|_{op} &\le \left\| \Delta_{12} \right\|_{op} + \frac{1}{2} \left\|M_1 - M_1^\prime \right\|_{op} \le \left[1 + \frac{\sqrt{6}}{2} \cdot \frac{\max_{j,k} \sigma_{k}^2(W_j)}{ \sigma_{\min}^2(W_2)} \right] e_{\Delta} \\
  &\le 2^{-28} f_1^{-9} f_2^{-3} d^{-113/8} \epsilon^{3} \sigma_1^{-1/4}(\Sigma) . 
 \end{aligned}
\end{equation}

Consequently: 

\begin{equation}
 \begin{aligned}
  \left\| R\right\|_{op} &\le \sqrt{1+e_\Delta\left\| W_2^{-1} \right\|_{op}^2} \le 1 + \sqrt{3} \cdot 2^{-32} f_1^{-9} f_2^{-3} d^{-115/8} \epsilon^{3} \sigma_1^{-3/4}(\Sigma) , 
 \end{aligned}
\end{equation}

\begin{equation}
  \left\| W_1^\prime \right\|_{op} \le \left\| W_1^\prime \right\|_{F} \le \sqrt{2} d^{1/8} \sigma_1^{1/4}(\Sigma) \left\| R\right\|_{op} \le \left(1+2^{-31}\right)\sqrt{2} d^{1/8} \sigma_1^{1/4}(\Sigma) , 
\end{equation}

\begin{equation}
  \left\| \frac{M_1 + M_1^\prime}{2} \right\|_{op} \le \left\| M_2 \right\|_{op} + \left\|M_2 - \frac{M_1 + M_1^\prime}{2} \right\|_{op} \le \left(1 + 2^{-29}\right)2 d^{1/4} \sigma_1^{1/2}(\Sigma) , 
\end{equation}

\begin{equation}
 \begin{aligned}
  \left\| M_1^\prime M_2 M_1 - M_1 M_2 M_1^\prime \right\|_{op} &\le \left\| M_1 - M_1^\prime \right\| \left\| M_2 \right \| \left\| M_1 + M_1^\prime \right\| \\ &\le \left(1 + 2^{-29}\right) 2^{-25} f_1^{-9} f_2^{-3} d^{-109/8} \epsilon^{3} \sigma_1^{3/4}(\Sigma) . 
 \end{aligned}
\end{equation}

By combining all results above, for $t \in [T_1, T_1 + T_2]$ such that $\left \| W_1 - W_1^\prime \right \|_F \le 3 f_1 d \epsilon$ holds, 

\begin{equation}
 \begin{aligned}
  \frac{\mathrm{d}}{\mathrm{d}t} \left \| W_1 - W_1^\prime \right \|_F^2 
  &\le - 0 - 0 - 0 \\
  &+ \|M_2\|_F \left\| M_1^\prime - M_1\right\|_{op} \|M_2\|_{op} \left( \left\| W_1^\prime \right\|_{op} + \left\| W_1 \right\|_{op} \right) \left\| W_1 - W_1^\prime \right\|_F \\
  &+ 2  \left\| - M_1^\prime M_2 M_1 + M_1 M_2 M_1^\prime \right\|_{op} \left\| W_1^\prime \right\|_{F} \left\| W_1 - W_1^\prime \right\|_F \\
  &+ 2 \max_j \|W_j\|_{op}^3 \|\Sigma - W\|_F \left( \left\| R^{H} R - I \right\|_{op} + \left\|I - R R^H \right\|_{op} \right) \left\| W_1 - W_1^\prime \right\|_F \\
  &+ 2a e_\Delta \left\|W_1 - W_1^\prime \right\|_F^2  \\
  &+ 4a e_\Delta \left\|W_2^{-1}\right\|_{op} \|W_2\|_F \left\| W_1^\prime \right\|_{op} \left\| W_1 - W_1^\prime \right\|_F \\
  &\le 2^{-22} f_1^{-8} f_2^{-3} d^{-25/2} \epsilon^4 \sigma_1(\Sigma) \\
  &+ 2^{-21} f_1^{-8} f_2^{-3} d^{-25/2} \epsilon^4 \sigma_1(\Sigma) \\ 
  &+ 2^{-24} f_1^{-8} f_2^{-3} d^{-25/2} \epsilon^4 \sigma_1(\Sigma) \\ 
  &+ 2^{-26} f_1^{-13} f_2^{-5} d^{-171/8} \epsilon^7 \sigma_1^{1/4}(\Sigma) \\ 
  &+ 2^{-18} f_1^{-8} f_2^{-3} d^{-25/2} \epsilon^4 \sigma_1(\Sigma) \\ 
  &\le 2^{-17} f_1^{-8} f_2^{-3} d^{-25/2} \epsilon^4 \sigma_1(\Sigma) . 
 \end{aligned}
\end{equation}

From Theorem \ref{stage 1: alignment stage}, at $t = T_1$, 

\begin{equation}
 \begin{aligned}
  \left \| W_1(T_1) - W_1^\prime(T_1) \right \|_F &\le \| W_1(T_1) \|_F + \left \| W_1^\prime(T_1) \right \|_F \le \| W_1(T_1) \|_F + \left \| W_4(T_1) \right \|_F \|R(T_1)\|_{op} \\
  &\le \left(1+2^{-32}\right) 2 \sqrt{d} \cdot \left(1 + 2^{-21}\right)f_1 \sqrt{d} \epsilon \le \left(1+2^{-20}\right) 2 f_1 d \epsilon . 
 \end{aligned}
\end{equation}

Thus $\left \| W_1 - W_1^\prime \right \|_F^2 \le \sqrt{\left[\left(1+2^{-20}\right) 2 f_1 d \epsilon \right]^2 + 2^{-17} f_1^{-8} f_2^{-3} d^{-25/2} \epsilon^4 \sigma_1(\Sigma) (t - T_1)}$ , when both $t \in [T_1, T_1 + T_2]$ and $\left \| W_1 - W_1^\prime \right \|_F^2 \le 3 f_1 d \epsilon$ hold. Then 

\begin{equation}
 \begin{aligned}
  &\left \| W_1(T_1 + T_2) - W_1^\prime(T_1 + T_2) \right \|_F^2 \\ \le& \sqrt{\left[\left(1+2^{-20}\right) 2 f_1 d \epsilon \right]^2 + 2^{-17} f_1^{-8} f_2^{-3} d^{-25/2} \epsilon^4 \sigma_1(\Sigma) T_2} \\
  \le& \sqrt{\left[\left(1+2^{-20}\right) 2 f_1 d \epsilon \right]^2 + 2^{-12} f_1^{-2} f_2^{-1} d^{-7/2} \epsilon^2} < 3 f_1 d \epsilon . 
 \end{aligned}
\end{equation}

which completes the proof. 

\end{proof}

\begin{corollary}\label{stage 2, main term}

The minimum eigenvalue of Hermitian term. 

For any $\sigma_k(W_1 + W_1^\prime)(T_1) \ge\frac{\epsilon}{2 f_1^3 f_2 d^{9/2}}$, it takes at most time $T_2$ to increase to $2^{3/4} \sigma_1^{1/4}(\Sigma)$. 

\end{corollary}

\begin{proof}

We analyze the dynamics of $\lambda_{k}\left(\left(W_1 + W_1^\prime \right)^H\left(W_1 + W_1^\prime \right)\right) = \sigma_{k}^2 $. The definition of $\eta_k(t)$ and $\chi_k(t)$ follows section \ref{section: hermitian main term, gf}. The dynamics can be expressed as below: 

\begin{equation}
 \begin{aligned}
  \frac{\mathrm{d}}{\mathrm{d}t} \sigma_{k}^2 &= 2\sigma_1(\Sigma) \sigma_{k}^2 \eta_k^H M_2 \eta_k + \sigma_1(\Sigma) \sigma_{k}^2 \chi_k^H \left(W_1 - W_1^{\prime}\right)^H \left(W_1 - W_1^{\prime}\right)\chi_k \\
  &+ \sigma_1(\Sigma)\sigma_{k} \Re\left(\eta_k^H \left( M_1^\prime - M_1 \right) \left(W_1 - W_1^{\prime}\right)\chi_k\right) \\
  &- \sigma_{k}^2 \eta_k^H M_2 (M_1 + M_1^\prime ) M_2 \eta_k - \sigma_{k} \Re\left(\eta_k^H M_2(M_1 - M_1^\prime) M_2 (W_1 - W_1^{\prime}) \chi_k \right) \\
  &+ \sigma_k \Re\left( \eta_k^H \left( M_1^\prime M_2 M_1 - M_1 M_2 M_1^\prime \right) \left( W_1^{\prime} - W_1 \right) \chi_k \right) \\ 
  &- 2\sigma_{k} \Re\left(\eta_k^H W_1(\Sigma - W^H) W_4 \left(R^{H} R - I\right) W_4^{H} \chi_k \right) \\ 
  &- 2\sigma_{k} \Re\left(\eta_k^H \left(I - R R^H\right) M_2 W_1 (\Sigma - W^H) \chi_k \right) \\ 
  &- 2a\sigma_{k}^2 \Re\left(\eta_k^H \Delta_{12} \eta_k\right) + 4a \sigma_{k} \Re\left(\eta_k^H W_2^{-1} \Delta_{23} W_2 W_1^{\prime} \chi_k\right) . 
 \end{aligned}
\end{equation}

From $\left\|M_2 - \frac{M_1 + M_1^\prime}{2} \right\|_{op} \le 2^{-28} f_1^{-9} f_2^{-3} d^{-113/8} \epsilon^{3} \sigma_1^{-1/4}(\Sigma)$ and $\left\| \frac{M_1 + M_1^\prime}{2} \right\|_{op}  \le \left(1 + 2^{-29}\right) 2 d^{1/4} \sigma_1^{1/2}(\Sigma)$, 

\begin{equation}
 \begin{aligned}
  \eta_k^H M_2 \eta_k &\ge \eta_k^H \left( \frac{M_1 + M_1^\prime}{2} \right) \eta_k - \left\|M_2 - \frac{M_1 + M_1^\prime}{2} \right\|_{op}\\ 
  &\ge \eta_k^H \left( \frac{M_1 + M_1^\prime}{2} \right) \eta_k - 2^{-28} f_1^{-9} f_2^{-3} d^{-113/8} \epsilon^{3} \sigma_1^{-1/4}(\Sigma) \\
  \eta_k^H M_2 (M_1 + M_1^\prime ) M_2 \eta_k &\le \eta_k^H \left( \frac{M_1 + M_1^\prime}{2} \right) (M_1 + M_1^\prime) \left( \frac{M_1 + M_1^\prime}{2} \right) \eta_k \\&+ 2\left\|M_2 - \frac{M_1 + M_1^\prime}{2} \right\|_{op} \left\| \frac{M_1 + M_1^\prime}{2} \right\|_{op} \left( \left\|M_2 \right\|_{op} + \left\| \frac{M_1 + M_1^\prime}{2} \right\|_{op} \right) \\
  &\le \eta_k^H \left( \frac{M_1 + M_1^\prime}{2} \right) (M_1 + M_1^\prime) \left( \frac{M_1 + M_1^\prime}{2} \right) \eta_k \\&+ \left(1 + 2^{-28}\right)2^{-24} f_1^{-9} f_2^{-3} d^{-109/8} \epsilon^3 \sigma_1^{3/4}(\Sigma) . 
 \end{aligned}
\end{equation}

By Lemma \ref{stage 2, skew-hermitian error}, $\left\|W_1 - W_1^{\prime}\right\|_{op} \le \left\|W_1 - W_1^{\prime}\right\|_{F} \le 3f_1 d \epsilon $, 

\begin{equation}
 \begin{aligned}
  \frac{\mathrm{d}}{\mathrm{d}t} \sigma_{k}^2 &\ge 2\sigma_1(\Sigma) \sigma_{k}^2 \eta_k^H M_2 \eta_k + 0 \\
  &- \sigma_1(\Sigma)\sigma_{k} \left\| M_1^\prime - M_1 \right\|_{op} \left\|W_1 - W_1^{\prime}\right\|_{op} \\
  &- \sigma_{k}^2 \eta_k^H M_2 (M_1 + M_1^\prime ) M_2 \eta_k - \sigma_{k} \max_{j}\|W_j\|_{op}^4 \left\|M_1 - M_1^\prime \right\|_{op} \left\|W_1 - W_1^{\prime} \right\|_{op} \\
  &- \sigma_k \left\| M_1^\prime M_2 M_1 - M_1 M_2 M_1^\prime \right\|_{op} \left\| W_1^{\prime} - W_1 \right\|_{op} \\ 
  &- 2\sigma_{k} \max_{j}\|W_j\|_{op}^3 \left\|\Sigma - W\right\|_{op} \left( \left\|R^{H} R - I\right\|_{op} + \left\|I - R R^H\right\|_{op} \right) \\
  &- 2a e_{\Delta} \sigma_{k}^2 - 4a e_{\Delta} \sigma_{k} \left\|W_2^{-1}\right\|_{op} \max_{j}\|W_j\|_{op}^2 \|R\|_{op}  \\
  &\ge 2\sigma_1(\Sigma) \sigma_{k}^2 \left( \eta_k^H \left( \frac{M_1 + M_1^\prime}{2} \right) \eta_k - 2^{-28} f_1^{-9} f_2^{-3} d^{-113/8} \epsilon^{3} \sigma_1^{-1/4}(\Sigma) \right) \\
  &- \sigma_{k} \left\|W_1 - W_1^{\prime}\right\|_{op} \cdot 2^{-27} f_1^{-9} f_2^{-3} d^{-113/8} \epsilon^3 \sigma_1^{3/4}(\Sigma) \\
  &- \sigma_{k}^2 \left[\eta_k^H \left( \frac{M_1 + M_1^\prime}{2} \right) (M_1 + M_1^\prime) \left( \frac{M_1 + M_1^\prime}{2} \right) \eta_k + \left(1 + 2^{-28}\right)2^{-24} f_1^{-9} f_2^{-3} d^{-109/8} \epsilon^3 \sigma_1^{3/4}(\Sigma)\right] \\
  &- \sigma_{k}\left\| W_1 - W_1^{\prime} \right\|_{op} \cdot 2^{-25} f_1^{-9} f_2^{-3} d^{-109/8} \epsilon^3 \sigma_1^{3/4}(\Sigma) \\
  &- \sigma_k \left\| W_1 - W_1^{\prime} \right\|_{op} \cdot \left(1 + 2^{-29}\right) 2^{-25} f_1^{-9} f_2^{-3} d^{-109/8} \epsilon^{3} \sigma_1^{3/4}(\Sigma) \\ 
  &- \sigma_{k} \cdot 2^{-25} f_1^{-9} f_2^{-3} d^{-27/2} \epsilon^3 \sigma_1(\Sigma) \\
  & - \sigma_{k}^2 \cdot 2^{-29} f_1^{-15} f_2^{-5} d^{-187/8} \epsilon^5 \sigma_1^{1/4}(\Sigma) - \sigma_{k} \cdot 2^{-22} f_1^{-9} f_2^{-3} d^{-27/2} \epsilon^3 \sigma_1(\Sigma) \\
  &\ge 2 \sigma_{k}^2 \eta_k^H \left[ \sigma_1(\Sigma) \left( \frac{M_1 + M_1^\prime}{2} \right) - \left( \frac{M_1 + M_1^\prime}{2} \right)^3 \right] \eta_k \\ 
  &- \sigma_{k} \cdot \left(1+2^{-1}\right) 2^{-22} f_1^{-9} f_2^{-3} d^{-27/2} \epsilon^3 \sigma_1(\Sigma) - \sigma_{k}^2 \cdot 2^{-23} f_1^{-9} f_2^{-3} d^{-109/8} \epsilon^3 \sigma_1(\Sigma) . 
 \end{aligned}
\end{equation}

under $\sigma_k \ge \frac{\epsilon}{2 f_1^3 f_2 d^{9/2}}$, 

\begin{equation}
 \begin{aligned}
  \frac{\mathrm{d}}{\mathrm{d}t} \sigma_{k}^2 &\ge 2 \sigma_{k}^2 \eta_k^H \left[ \sigma_1(\Sigma) \left( \frac{M_1 + M_1^\prime}{2} \right) - \left( \frac{M_1 + M_1^\prime}{2} \right)^3 \right] \eta_k - 2^{-18} \sigma_1(\Sigma) \sigma_{k}^4 . 
 \end{aligned}
\end{equation}

Denote $P = \frac{W_1 + W_1^{\prime}}{2}$, $Q = \frac{W_1 - W_1^{\prime}}{2}$. Notice that 

\begin{equation}
  PP^H + QQ^H = \frac{M_1 + M_1^\prime}{2}, P^H \eta_k = \frac{1}{2}
  \sigma_k \chi_k , 
\end{equation}

\begin{equation}
 \begin{aligned}
  \eta_k^H \left( \frac{M_1 + M_1^\prime}{2} \right) \eta_k &= \eta_k^H \left( PP^H + QQ^H \right) \eta_k \ge \frac{1}{4} \sigma_k^2 , 
 \end{aligned}
\end{equation}

\begin{equation}
 \begin{aligned}
  \eta_k^H \left( \frac{M_1 + M_1^\prime}{2} \right)^3 \eta_k &= \eta_k^H \left( PP^H + QQ^H \right) \left( \frac{M_1 + M_1^\prime}{2} \right) \left( PP^H + QQ^H \right) \eta_k \\ 
  &= \frac{1}{16} \sigma_k^4 \eta_k^H \left( \frac{M_1 + M_1^\prime}{2} \right) \eta_k + \eta_k^H QQ^H \left( \frac{M_1 + M_1^\prime}{2} \right) QQ^H \eta_k \\ & + \frac{1}{4} \sigma_k^2 \eta_k^H\left[ QQ^H \left( \frac{M_1 + M_1^\prime}{2} \right) + \left( \frac{M_1 + M_1^\prime}{2} \right) QQ^H \right] \eta_k \\ 
  &\le \frac{1}{16} \sigma_k^4 \eta_k^H \left( \frac{M_1 + M_1^\prime}{2} \right) \eta_k + \left\| \frac{M_1 + M_1^\prime}{2} \right\|_{op} \left( \frac{1}{2} \sigma_k^2 \|Q\|_{op}^2 + \|Q\|_{op}^4 \right) . 
 \end{aligned}
\end{equation}

Notice $\|Q\|_{op} = \frac{1}{2} \left\|W_1 - W_1^{\prime}\right\|_{F} \le \frac{3}{2} f_1 d \epsilon \le \sigma_k \cdot 3 f_1^4 f_2 d^{11/2}$, $\epsilon \le \frac{1}{32 f_1^5 f_2 d^{53/8}} \sigma_1^{1/4}(\Sigma)$, 

\begin{equation}
 \begin{aligned}
  \frac{\mathrm{d}}{\mathrm{d}t} \sigma_{k}^2 &\ge 2 \sigma_{k}^2 \left[ \left(\sigma_1(\Sigma) - \frac{1}{16} \sigma_k^4 \right) \eta_k^H \left( \frac{M_1 + M_1^\prime}{2} \right) \eta_k - \left\| \frac{M_1 + M_1^\prime}{2} \right\|_{op} \left( \frac{1}{2} \sigma_k^2 \|Q\|_{op}^2 + \|Q\|_{op}^4 \right) \right] \\&- 2^{-18} \sigma_1(\Sigma) \sigma_{k}^4 \\ 
  &\ge \frac{1}{2} \sigma_{k}^4  \left(\sigma_1(\Sigma) - \frac{1}{16} \sigma_k^4 \right) - 2 \sigma_{k}^2\left\| \frac{M_1 + M_1^\prime}{2} \right\|_{op} \|Q\|_{op}^2 \left( \frac{1}{2} \sigma_k^2  + \|Q\|_{op}^2 \right) - 2^{-18} \sigma_1(\Sigma) \sigma_{k}^4 \\ 
  &\ge \frac{1}{2} \sigma_{k}^4  \sigma_1(\Sigma) - \frac{1}{32} \sigma_k^8 - 81\left(1+2^{-5}\right) f_1^{10} f_2^{2} d^{53/4} \epsilon^2 \sigma_1^{1/2}(\Sigma) \sigma_{k}^4 - 2^{-18} \sigma_1(\Sigma) \sigma_{k}^4 \\ 
  &\ge \frac{3}{8} \sigma_{k}^4 \sigma_1(\Sigma) - \frac{1}{32} \sigma_k^8 . 
 \end{aligned}
\end{equation}

This indicates that for $\sigma_k \in \left[\frac{\epsilon}{2 f_1^3 f_2 d^{9/2}}, 2^{3/4} \sigma_1^{1/4}(\Sigma)\right]$, $\sigma_k$ is monotonically increasing. By standard calculus, it takes at most time $\Delta t \left(\sigma_k \ge 2^{3/4} \sigma_1^{1/4}(\Sigma) \right) \le T_2$ for $\sigma_k$ to increase from at least $\frac{\epsilon}{2 f_1^3 f_2 d^{9/2}}$ to $2^{3/4} \sigma_1^{1/4}(\Sigma)$: 

\begin{equation}
 \begin{aligned}
  \Delta t \left(\sigma_k \ge 2^{3/4} \sigma_1^{1/4}(\Sigma) \right) &\le \int_{\frac{\epsilon}{2 f_1^3 f_2 d^{9/2}}}^{2 \cdot \sqrt[4]{\frac{\sigma_1(\Sigma)}{{2}} }} \left(\frac{3}{8} \sigma_1(\Sigma) \sigma_{k}^4 - \frac{1}{32} \sigma_k^8 \right)^{-1} \mathrm{d} \left(\sigma_k^2\right) \\ 
  &= \int_{\frac{\epsilon}{4 f_1^6 f_2^2 d^{9}}}^{4 \cdot \sqrt{\frac{\sigma_1(\Sigma)}{{2}}} } \left(\frac{3}{8} \sigma_1(\Sigma) \lambda_k^2 - \frac{1}{32} \lambda_k^4 \right)^{-1} \mathrm{d} \lambda_k \\
  &\le \int_{\frac{\epsilon}{4 f_1^6 f_2^2 d^{9}}}^{4 \cdot \sqrt{\frac{\sigma_1(\Sigma)}{{2}}} } \left(\frac{3}{8} \sigma_1(\Sigma) \lambda_k^2 - \frac{1}{4} \sigma_1(\Sigma) \lambda_k^2 \right)^{-1} \mathrm{d} \lambda_k \\
  &\le 8 \left[ \left( \frac{\epsilon}{4 f_1^6 f_2^2 d^{9}} \right)^{-1} - \left( 4 \cdot \sqrt{\frac{\sigma_1(\Sigma)}{{2}}}  \right)^{-1} \right]\sigma_1^{-1}(\Sigma) \le T_2 . 
 \end{aligned}
\end{equation}

And for $t \in \left[T_1 + \Delta t \left(\sigma_k \ge 2^{3/4} \sigma_1^{1/4}(\Sigma) \right), T_1 + T_2\right]$, $\sigma_k$ does not decrease to less than $2^{3/4} \sigma_1^{1/4}(\Sigma)$ if $t \le T_1 + T_2$. This is from the continuity of $\sigma_k$ and the time derivative of $\sigma_k^2$ at $\sigma_k = 2^{3/4} \sigma_1^{1/4}(\Sigma)$, $t \le T_1 + T_2$ is positive: 

\begin{equation}
  \left. \frac{\mathrm{d}}{\mathrm{d}t}\sigma_{k}^2 \right|_{\sigma_k = 2^{3/4} \sigma_1^{1/4}(\Sigma), t\le T_1+T_2} \ge \frac{1}{8} \sigma_1(\Sigma) \cdot \left( 2^{3/4} \sigma_1^{1/4}(\Sigma) \right)^4 > 0 . 
\end{equation}

\end{proof}

\subsection{Stage 3: local convergence stage}\label{stage 3: delta convergence, gf}

In this stage, we analysis the time to reach $\epsilon_{\rm conv}$-convergence, that is 

\begin{equation}
  T(\epsilon_{\rm conv}) = \inf_t \{\mathcal{L}(t) \le \epsilon_{\rm conv} \} . 
\end{equation}

\begin{lemma}\label{loval convergence stage: lower bound of main term}

$\sigma_{\min}\left(W_1 + W_1^\prime\right)$ is lower bounded, while the skew-hermitian error is upper bounded. 

For $t \ge T_1 + T_2$, 

\begin{equation}\label{ineq: stage 3, lower bound of main term}
 \begin{aligned}
  \sigma_{\min}\left(W_1 + W_1^\prime\right)(t) &\ge 2^{3/4} \sigma_1^{1/4}(\Sigma) \\ 
  \left \| W_1 - W_1^\prime \right \|_F &\le 3 f_1 d \epsilon . 
 \end{aligned}
\end{equation}

\end{lemma}

\begin{proof}

(\ref{ineq: stage 3, lower bound of main term}) holds at $t = T_1 + T_2$. Since both L.H.S. change continuously, it left to prove that the derivatives at the critical points (to be specific, $t^\prime\ge T_2$ such that $\left.\|W_1 - W_1^\prime\|_F\right|_{t=t^\prime} = 3f_1 d\epsilon$ or $\left. \sigma_{k} \left( W_1 + W_1^\prime \right) \right|_{t=t^\prime} = 2^{3/4} \sigma_1(\Sigma)$) are positive/negative. (If such time does not exist, then the proof is done. )

From 

\begin{equation}
 \begin{aligned}
  \sigma_{\min}^2(W_1) + \sigma_{\min}^2(W_1^\prime) &\ge \frac{1}{2} \lambda_{\min}\left((W_1 + W_1 ^\prime)(W_1 + W_1 ^\prime)^H + (W_1 - W_1 ^\prime)(W_1 - W_1 ^\prime)^H \right) \\
  &\ge \frac{1}{2} \sigma_{\min}^2\left(W_1 + W_1^\prime\right) , 
 \end{aligned}
\end{equation}

and 

\begin{equation}
  \sigma_{\min}(W_1^\prime) \le  \sigma_{\min}(W_1) + \left \| W_1 - W_1^\prime \right \|_F . 
\end{equation}

For $t > T_1 + T_2$ such as (\ref{ineq: stage 3, lower bound of main term}) holds, 

\begin{equation}
  \sigma_{\min} (W_2) \ge \sigma_{\min}(W_1) - e_\Delta \ge \frac{1}{\sqrt{2}} \sigma_1^{1/4}(\Sigma) . 
\end{equation}

Then by following almost the same arguments as Lemma \ref{stage 2, skew-hermitian error} and \ref{stage 2, main term}, 

\begin{equation}
 \begin{aligned}
  \frac{\mathrm{d}}{\mathrm{d}t} \left \| W_1 - W_1^\prime \right \|_F^2 
  &\le -2 \sigma_1(\Sigma) \mathrm{tr}\left( \left( W_1 - W_1^\prime \right)^H \sigma_{\min}^2(W_2) \left(W_1 - W_1^\prime \right) \right) - 0 - 0 \\
  &+ 2^{-17}f_1^{-8} f_2^{-3} d^{-25/2} \epsilon^4 \sigma_1(\Sigma) \\ 
  &\le - \sigma_1^{3/2}(\Sigma)\left \| W_1 - W_1^\prime \right \|_F^2  + 2^{-17}f_1^{-8} f_2^{-3} d^{-25/2} \epsilon^4 \sigma_1(\Sigma) , 
 \end{aligned}
\end{equation}

\begin{equation}
  \frac{\mathrm{d}}{\mathrm{d}t} \sigma_{k}^2 \left( W_1 + W_1^\prime \right) \ge \frac{3}{8} \sigma_{k}^4 \left( W_1 + W_1^\prime \right) \sigma_1(\Sigma) - \frac{1}{32} \sigma_k^8 \left( W_1 + W_1^\prime \right) . 
\end{equation}

Suppose for some $t_1,t_2 \ge T_1 + T_2$ such that $\left.\|W_1 - W_1^\prime\|_F\right|_{t=t_1} = 3f_1 d\epsilon$, $\left. \sigma_{k} \left( W_1 + W_1^\prime \right) \right|_{t=t_2} = 2^{3/4} \sigma_1(\Sigma)$, then 

\begin{equation}
 \begin{aligned}
  \left. \frac{\mathrm{d}}{\mathrm{d}t} \left\| W_1 - W_1^\prime \right \|_F^2 \right|_{t=t_1} &\le 0 \\
  \left. \frac{\mathrm{d}}{\mathrm{d}t} \sigma_{k}^2 \left( W_1 + W_1^\prime \right) \right|_{t=t_2} &\ge 0 . 
 \end{aligned}
\end{equation}

This completes the proof. 

\end{proof}

\begin{theorem}\label{Total convergence bound, gf} Global convergence bound. 

For four-layer matrix factorization under gradient flow, with random Gaussian initialization with scaling factor $\epsilon \le \frac{\sigma_1^{1/4}(\Sigma)}{32 f_1^5 f_2 d^{53/8}} $, regularization factor $a \ge 32 f_1^{20} f_2 d^{13}\sigma_1(\Sigma) b$, where $b$ satisfies 

\begin{equation}
 \begin{aligned}
  b &\ge 5\ln\left( \frac{\sigma_1^{1/4}(\Sigma)}{\epsilon} \right) + \frac{281}{8} \ln d + 23\ln(4f_1) + 7\ln f_2 \\ 
  b &- \ln b \ge 3\ln\left(\frac{\sigma_1^{1/4}(\Sigma)}{\epsilon}\right) + \frac{303}{8}\ln d + 37 \ln(2f_1) + 6 \ln f_2 . 
 \end{aligned}
\end{equation}

Then for target matrix with identical singular values, there exists following $T(\epsilon_{\rm conv})$, such that for any $\epsilon_{\rm conv} > 0$, (1) with high probability over the complex initialization (2) with probability close to $\frac{1}{2}$ over the real initialization, when $t > T(\epsilon_{\rm conv})$, $\mathcal{L}(t) < \epsilon_{\rm conv}$.

\begin{equation}
 \begin{aligned}
  T(\epsilon_{\rm conv}) &\le T_1 + T_2 + \sigma_1^{-3/2}(\Sigma) \ln\left( \frac{d \sigma_1^2(\Sigma)}{\epsilon_{\rm conv}} \right) \\ &= \frac{1}{32 f_1^{14} f_2 d^{10} \epsilon^2 \sigma_1(\Sigma)} + \frac{32 f_1^6 f_2^2 d^9}{\sigma_1(\Sigma) \epsilon^2} + \sigma_1^{-3/2}(\Sigma) \ln\left( \frac{d \sigma_1^2(\Sigma)}{\epsilon_{\rm conv}} \right)  \\ 
  &= O\left( \frac{f_1^6 f_2^2 d^9}{\sigma_1(\Sigma) \epsilon^2} + \frac{1}{\sigma_1^{3/2}(\Sigma)} \ln\left( \frac{d \sigma_1^2(\Sigma)}{\epsilon_{\rm conv}} \right)\right) . 
 \end{aligned}
\end{equation}

\end{theorem}

\begin{proof}

Following the derivations in Lemma \ref{loval convergence stage: lower bound of main term}, 

\begin{equation}
  \min_{j,k} \sigma_k(W_j)(t>T_1 + T_2) \ge \frac{1}{\sqrt{2}} \sigma_1^{1/4}(\Sigma) . 
\end{equation}

By Lemma \ref{L ori non-increasing} and \ref{bound of w_j op}, 

\begin{equation}
 \begin{aligned}
  \mathcal{L}_{\rm ori}(t) &\le \mathcal{L}_{\rm ori}(T_1 + T_2) \exp\left(- 8 \min_{j,k} |\sigma_{k}(W_j)(t>T_1 + T_2)|^{6} (t - T_1 - T_2) \right) \\&\le \mathcal{L}_{\rm ori}(0) \exp\left(- 8 \min_{j,k} |\sigma_{k}(W_j)(t>T_1 + T_2)|^{6} (t - T_1 - T_2) \right) \\ 
  &\le 0.52 d \sigma_1^2(\Sigma) \exp\left(- \sigma_1^{3/2}(\Sigma) (t - T_1 - T_2) \right) . 
 \end{aligned}
\end{equation}

For regularization term, by invoking Theorem \ref{regularization term, convergence bound}, \ref{stage 1: alignment stage} and Lemma \ref{bound of w_j op}, 

\begin{equation}
 \begin{aligned}
  \mathcal{L}_{\rm reg}(t) &\le \mathcal{L}_{\rm reg}(T_1 + T_2) \exp\left( - \frac{4a}{3} \frac{\min_{j,k} |\sigma_{k}(W_j)(t>T_1 + T_2)|^{4}}{\max_{j,k} |\sigma_{k}(W_j)|^2} \cdot (t - T_1 - T_2) \right) \\
  &\le \frac{a}{4} e_{\Delta}^2(T_1+T_2) \exp\left( - \frac{4a}{3} \frac{\min_{j,k} |\sigma_{k}(W_j)(t>T_1 + T_2)|^{4}}{\max_{j,k} |\sigma_{k}(W_j)|^2} \cdot (t - T_1 - T_2) \right) \\
  &\le \frac{a}{4} e_{\Delta}^2(T_1) \exp\left( - \frac{4a}{3} \frac{\min_{j,k} |\sigma_{k}(W_j)(t>T_1 + T_2)|^{4}}{\max_{j,k} |\sigma_{k}(W_j)|^2} \cdot (t - T_1 - T_2) \right) \\
  &\le 2^{-76} f_1^{-36} f_2^{-12} d^{-57} \epsilon^{12} \sigma_1^{-1}(\Sigma) \exp\left( - 16 f_1^{20} f_2 d^{51/4} \sigma_1^{3/2}(\Sigma) (t - T_1 - T_2) \right) . 
 \end{aligned}
\end{equation}

By taking logarithm on the summation of these two inequalities, the proof is completed. 

\end{proof}

%% file: subfiles/appendix/appendix_8_staged,gd.tex
\section{Convergence under Gradient Descent, staged analysis}\label{section: convergence, staged analysis, gd}

This section states the complete proof of convergence under Random Gaussian Initialization \ref{subsection: random gaussian initialization}.

At the beginning we still assume (\ref{random gaussian initialization, conclusions}) holds. (For the complex case, it holds with high probability $1-\delta$; for the real case, it holds with probability $\frac{1}{2}(1-\delta)$. ) 

\begin{theorem}\label{Total convergence bound, gd} Global convergence bound under random Gaussian initialization, gradient descent. 

For four-layer matrix factorization under gradient descent, random Gaussian initialization with scaling factor $\epsilon \le \frac{\sigma_1^{1/4}(\Sigma)}{32 f_1^5 f_2 d^{53/8}} $, regularization factor $a \ge 32 f_1^{20} f_2 d^{13}\sigma_1(\Sigma) b$, where $b$ satisfies 

\begin{equation}
 \begin{aligned}
  b &\ge \max\left( 5\ln\left( \frac{\sigma_1^{1/4}(\Sigma)}{\epsilon} \right) + \frac{281}{8} \ln d + 23\ln(4f_1) + 7\ln f_2,\, 16 \ln(2 f_1 f_2 d)\right) \\ 
  b - \ln b &\ge 3\ln\left(\frac{\sigma_1^{1/4}(\Sigma)}{\epsilon}\right) + \frac{303}{8}\ln d + 37 \ln(2f_1) + 6 \ln f_2 . 
 \end{aligned}
\end{equation}

Then for target matrix with identical singular values, there exists following learning rate $\eta$ and convergence time $T(\epsilon_{\rm conv},\eta)$, such that for any $\epsilon_{\rm conv} > 0$, (1) with high probability over the complex initialization (2)  with probability close to $\frac{1}{2}$ over the real initialization, when $t > T(\epsilon_{\rm conv},\eta)$, $\mathcal{L}(t) < \epsilon_{\rm conv}$. 

\begin{equation}
 \begin{aligned}
  \eta &= O\Big(\min \Big( a^{-2}f_1^{-4}d^{-2}\epsilon^{-2} \sigma_1(\Sigma), \\& a f_1^{-56} f_2^{-14} d^{-301/4} \epsilon^8 \sigma_1^{-9/2}(\Sigma) , a^{-1} f_1^{-44} f_2^{-10} d^{-219/4} \epsilon^4 \sigma_1^{-3/2}(\Sigma), \\ & f_1^{-27} f_2^{-9} d^{-355/8} \epsilon^9 \sigma_1^{-15/4}(\Sigma), a^{-1}f_1^{-21} f_2^{-7} d^{-273/8} \epsilon^7 \sigma_1^{-9/4}(\Sigma)\Big)\Big) \\
  T(\epsilon_{\rm conv},\eta) &\le T_1 + T_2 + \eta^{-1} \sigma_1^{-3/2}(\Sigma) \ln\left( \frac{d \sigma_1^2(\Sigma)}{\epsilon_{\rm conv}} \right) \\ 
  &= O\left( \frac{f_1^6 f_2^2 d^9}{\eta\sigma_1(\Sigma) \epsilon^2} + \frac{1}{\eta\sigma_1^{3/2}(\Sigma)} \ln\left( \frac{d \sigma_1^2(\Sigma)}{\epsilon_{\rm conv}} \right)\right) . 
 \end{aligned}
\end{equation}

\end{theorem}

The following section completes the proof. 

\subsection{Stage 1: alignment stage}

In this section, we set $\epsilon \le \frac{\sigma_1^{1/4}(\Sigma)}{4f_1\sqrt{d}}$, $a \ge 2^5 f_1^{20} f_2 d^{13} \sigma_1(\Sigma) b$, where $b \ge 2^4 \ln(4f_1d) + \ln f_2$. $\eta = O\left( \frac{\sigma_1(\Sigma)}{a^2 f_1^4 d^2 \epsilon^2} \right)$, with appropriate small constant. 
Without loss of generality, $f_1 \ge 2$, $f_2 \ge f_1^6$. 

\begin{theorem}\label{stage 1: alignment stage, GD}

At $T_1 = \frac{1}{32 f_1^{14} f_2 d^{10} \epsilon^2 \sigma_1(\Sigma) \eta}$, the following conclusions hold: 

\begin{equation}
 \begin{aligned}
  \left.\sigma_{\min}\left(W_1 + W_1^{\prime}\right)\right|_{t=T_1} &\ge \frac{\epsilon}{2 f_1^3 f_2 d^{9/2}} \\
  e_{\Delta}(T_1) &\le 2\sqrt{3 f_1^4 d^3 \epsilon^4 e^{-2b} + \eta O\left(a^{-1} f_1^{14} d^8 \epsilon^6 \sigma_1^2(\Sigma) \right)} \\ 
  \max_{j,k} |\sigma_k(W_j(T_1))| &\le (1+2^{-21}) f_1\sqrt{d}\epsilon \\ 
  \min_{j,k} |\sigma_k(W_j(T_1))| &\ge (1-2^{-17}) \frac{\epsilon}{f_1\sqrt{d}} . 
 \end{aligned}
\end{equation}

\end{theorem}

This section proves the theorem above by following Lemmas and Corollaries. 

\begin{lemma}\label{alignment stage, bound of max and min singular values, GD}

Maximum and minimum singular value bound of weight matrices in alignment stage. 

For $t \in \left[0,\frac{1}{32 f_1^4 d^2 \epsilon^2 \sigma_1(\Sigma) \eta} \right]$, 

\begin{equation}
  \min_{j,k}\sigma_k(W_j) \ge \frac{\epsilon}{f_1\sqrt{d}} - 16 f_1^3 d^{3/2} \epsilon^3 \sigma_1(\Sigma) t ,\,   \max_{j,k}\sigma_k(W_j) \le \frac{f_1\sqrt{d}\epsilon}{\sqrt{1 - 4 f_1^2 d\epsilon^2 \sigma_1(\Sigma) t}} . 
\end{equation}
\end{lemma}

\begin{proof}

For $t\ge 0$ such that $\max_{j,k} \sigma_k(W_j) \le 2f_1\sqrt{d} \epsilon \le \frac{1}{2}\sigma_1^{1/4}(\Sigma)$, 

\begin{equation}
  \max_{j} \left\| \nabla_{W_j} \mathcal{L}_{\rm ori} \right\|_{op} \le \max_{j,k} |\sigma_k(W_j)|^{3} \left( \sigma_1(\Sigma) + \max_{j,k} |\sigma_k(W_j)|^{4} \right)  \le \frac{3}{2} \max_{j,k} |\sigma_k(W_j)|^{3} \sigma_1(\Sigma) . 
\end{equation}

By invoking Corollary \ref{maximum and minimum singular values are irrelevant of the regularization term, GD}, for $t\ge 0$ such that $\min_{j,k} \sigma_k(W_j(t)) \ge \frac{\epsilon}{2 f_1 \sqrt{d}}$, 

\begin{equation}
 \begin{aligned}
  \max_{j,k} \sigma_k^2(W_j(t+1)) - \max_{j,k} \sigma_k^2(W_j(t)) &\le 3 \eta \max_{j,k} |\sigma_k(W_j(t))|^{4} \sigma_1(\Sigma) \\&+ \eta^2 O\left( a^2 \left( \epsilon f_1 \sqrt{d} \right)^6 \right) \\ 
  & \le 4 \eta \max_{j,k} |\sigma_k(W_j(t))|^{4} \sigma_1(\Sigma) \\ 
  \min_{j,k} \sigma_k^2(W_j(t+1)) - \min_{j,k} \sigma_k^2(W_j(t)) &\ge - 3 \eta \min_{j,k} \left| \sigma_k(W_j(t)) \right| \max_{j,k} |\sigma_k(W_j(t))|^{3} \sigma_1(\Sigma) \\&+ \eta^2 O\left( a^2 \left( \epsilon f_1 \sqrt{d} \right)^6 \right) \\
  &\ge - 2 \eta \left(\min_{j,k} \left| \sigma_k(W_j(t+1)) \right| + \min_{j,k} \left| \sigma_k(W_j(t)) \right| \right) \\& \cdot  \max_{j,k} |\sigma_k(W_j(t))|^{3} \sigma_1(\Sigma) . 
 \end{aligned}
\end{equation}

By solving the differential inequality,  

\begin{equation}
  \max_{j,k} \sigma_k|W_j(t)| \le \frac{\max_{j,k} \sigma_k|W_j(0)|}{\sqrt{1 - 4 \sigma_1(\Sigma) \max_{j,k} \sigma_k|W_j(0)|^2 \eta t}} \le \frac{f_1\sqrt{d}\epsilon}{\sqrt{1 - 4 f_1^2 d\epsilon^2 \sigma_1(\Sigma) \eta t}},\, t \in \left[0,\frac{3}{16 f_1^2 d \epsilon^2 \sigma_1(\Sigma) \eta}\right] , 
\end{equation}

\begin{equation}
  \min_{j,k} |\sigma_k(W_j(t))| \ge \frac{\epsilon}{f_1 \sqrt{d}} - 16 f_1^3 d^{3/2} \epsilon^3 \sigma_1(\Sigma) \eta t,\,t\in \left[0,\frac{1}{32 f_1^4 d^2 \epsilon^2 \sigma_1(\Sigma) \eta} \right] . 
\end{equation}

This completes the proof.

\end{proof}

Notice that  

\begin{equation}
 \begin{aligned}
  \max_{j,k} |\sigma_k(W_j(t\le T_1))| &\le \frac{f_1\sqrt{d} \epsilon}{\sqrt{1-\frac{1}{8f_1^{12}f_2}}} \le (1+2^{-21}) f_1\sqrt{d}\epsilon \\ 
  \min_{j,k} |\sigma_k(W_j(t\le T_1))| &\ge \left(1-\frac{1}{2f_1^{10}f_2}\right) \cdot \frac{\epsilon}{f_1\sqrt{d}} \ge (1-2^{-17}) \frac{\epsilon}{f_1\sqrt{d}} . 
 \end{aligned}
\end{equation}

\begin{corollary}\label{stage 1, e_delta, GD}

Balanced term error in alignment stage.  

\begin{equation}
 \begin{aligned}
  e_{\Delta}(T_1) &\le \sqrt{3} \cdot 2^{-31} f_1^{-14} f_2^{-1} d^{-29/2} \epsilon^2 . 
 \end{aligned}
\end{equation}

\end{corollary}

\begin{proof}

By simply combining Theorem \ref{regularization term, convergence bound, GD} and Lemma \ref{alignment stage, bound of max and min singular values, GD}, denote $M = \max_{j,k,t\le T_1}(W_j(t))$, 

\begin{equation}
 \begin{aligned}
  \mathcal{L}_{\rm reg}(t+1) &\le \left(1 - 2.509 \frac{\eta a \epsilon^2}{f_1^6 d^3}\right) \cdot \mathcal{L}_{\rm reg}(t) + \eta^2 O\left(a^2 M^4 \mathcal{L}_{\rm reg}(t) + \sqrt{a \mathcal{L}_{\rm reg}(t)} M^6 \mathcal{L}_{\rm ori}(t) \right) \\ &+ \eta^4 O\left(a M^{12} \mathcal{L}_{\rm ori}(t)^2 + a^3 M^4 \mathcal{L}_{\rm reg}(t)^2 \right) \\ 
  &\le \left(1 - \frac{2 \eta a \epsilon^2}{f_1^6 d^3}\right) \cdot \mathcal{L}_{\rm reg}(t) + \eta^2 O\left(a M^8 \mathcal{L}_{\rm ori}(t) \right) \\ 
  &\le \left(1 - \frac{2 \eta a \epsilon^2}{f_1^6 d^3}\right) \cdot \mathcal{L}_{\rm reg}(t) + \eta^2 O\left(a f_1^8 d^5 \epsilon^8 \sigma_1^2(\Sigma) \right) , 
 \end{aligned}
\end{equation}

giving 

\begin{equation}
 \begin{aligned}
  \mathcal{L}_{\rm reg}(t) &\le \mathcal{L}_{\rm reg}(0) e^{-\frac{2 \eta a \epsilon^2}{f_1^6 d^3} t} + \eta O\left( f_1^{14} d^8 \epsilon^6 \sigma_1^2(\Sigma) \right) . 
 \end{aligned}
\end{equation}

\begin{equation}
 \begin{aligned}
  \mathcal{L}_{\rm reg}(T_1) 
  &\le 3 a f_1^4 d^3 \epsilon^4 e^{-2b} + \eta O\left( f_1^{14} d^8 \epsilon^6 \sigma_1^2(\Sigma) \right) , 
 \end{aligned}
\end{equation}

\begin{equation}
  e_\Delta(T_1) = 2\sqrt{\frac{\mathcal{L}_{\rm reg}(T_1)}{a}} \le \sqrt{3} \cdot 2^{-31} f_1^{-14} f_2^{-1} d^{-29/2} \epsilon^2 . 
\end{equation}

\end{proof}

\begin{corollary}\label{Main term at the end of alignment stage, gd}

Main term at the end of alignment stage. 

At $t=T_1$, 

\begin{equation}
  \left.\sigma_{\min}\left(W_1 + W_1^{\prime}\right)\right|_{t=T_1} \ge \frac{\epsilon}{2 f_1^3 f_2 d^{9/2}} . 
\end{equation}

\end{corollary}

\begin{proof}

Denote $\Delta_{X}(t) = X(t) - X(0)$ for arbitrary $X$. 

At $t = T_1$, 

\begin{equation}
 \begin{aligned}
  \| \Delta_{W}(T_1)\|_{op} 
  \le& \left\| \sum_{t^\prime = 0}^{T_1-1} \eta \left[ \sum_{j=1}^{4} W_{\prod_L,j+1}(t^\prime) W_{\prod_L,j+1}(t^\prime)^H \left(\Sigma - W(t^\prime)\right) W_{\prod_R,j-1}^H (t^\prime) W_{\prod_R,j-1}(t^\prime) \right] \right\|_{op} \\
  +& \eta^2 \sum_{t^\prime = 0}^{T_1-1} O\left( \max_{j \in [1,4]\cap \mathbb{N}^*} \left\| \nabla_{W_j} \mathcal{L}(t^\prime) \right\|_{F}^2 \cdot \max_{j \in [1,4]\cap \mathbb{N}^*} \|W_j(t^\prime)\|_{op}^2 \right) \\ 
  \le& \eta T_1 \cdot 6 \sigma_1(\Sigma) \cdot \left( \left(1+2^{-21}\right) f_1 \sqrt{d} \epsilon \right)^6 + \eta^2 T_1 O\left( a^2 d \left( f_1 \sqrt{d} \epsilon \right)^8 \right) \\
  \le& \eta T_1 \cdot 8 \sigma_1(\Sigma) \cdot \left( \left(1+2^{-21}\right) f_1 \sqrt{d} \epsilon \right)^6 \\
  \le& \left(1+2^{-18}\right) \cdot \frac{1}{4} f_1^{-8} f_2^{-1} d^{-7} \epsilon^4 . 
 \end{aligned}
\end{equation}

Thus 

\begin{equation}
 \begin{aligned}
  \left\| \Delta_{W^H W}(T_1) \right\|_{op} &= \left\| \frac{1}{2}\left[\left(W(T_1) + W(0)\right)^H \Delta_{W}(T_1) + \Delta_{W}(T_1)^H \left(W(T_1) + W(0)\right) \right] \right\|_{op} \\ 
  &\le (1+2^{-17})\cdot\frac{1}{2} f_1^{-4} f_2^{-1} d^{-5} \epsilon^8 . 
 \end{aligned}
\end{equation}

From Corollary \ref{stage 1, e_delta, GD}, 

\begin{equation}
 \begin{aligned}
  &\left\| \left(W_1(T_1)^H W_2(T_1)^H W_2(T_1) W_1(T_1) \right)^2 - W(T_1)^H W(T_1) \right\|_{op} \\ \le& \left\|W_1(T_1)^H W_2(T_1)^H \right\|_{op} \left\| M_{\Delta1234}(T_1) \right\|_{op}
  \left\| W_2(T_1) W_1(T_1) \right\|_{op} \\ \le& 2^{-12} f_1^{-8} f_2^{-16} d^{-23/2} \epsilon^8 . 
 \end{aligned}
\end{equation}

Thus 

\begin{equation}
 \begin{aligned}
  &\left\| \left(W_1(T_1)^H W_2(T_1)^H W_2(T_1) W_1(T_1) \right)^2 - W(T_0)^H W(T_0) \right\|_{op} \\ 
  \le& \left\| \left(W_1(T_1)^H W_2(T_1)^H W_2(T_1) W_1(T_1) \right)^2 - W(T_1)^H W(T_1) \right\|_{op} + \left\| \Delta_{W^H W}(T_1) \right\|_{op} \\ 
  \le& (1+2^{-16})\cdot\frac{1}{2} f_1^{-4} f_2^{-1} d^{-5} \epsilon^8 . 
 \end{aligned}
\end{equation}

From Lemma \ref{error bound, sqrt}, 

\begin{equation}
 \begin{aligned}
  &\left\| W_1(T_1)^H W_2(T_1)^H W_2(T_1) W_1(T_1) - \left( W(T_0)^H W(T_0) \right)^{1/2} \right\|_{op} \\
  \le& \frac{\left\| \left(W_1(T_1)^H W_2(T_1)^H W_2(T_1) W_1(T_1) \right)^2 - W(T_0)^H W(T_0) \right\|_{op}}{2\sqrt{\lambda_{\min}\left( W(T_0)^H W(T_0) \right) - \left\|  \left(W_1(T_1)^H W_2(T_1)^H W_2(T_1) W_1(T_1) \right)^2 - W(T_0)^H W(T_0)\right\|_{op}}} \\
  \le& \frac{(1+2^{-16})\cdot\frac{1}{2} f_1^{-4} f_2^{-1} d^{-5} \epsilon^8}{2\sqrt{\left( \frac{\epsilon}{f_1 \sqrt{d}} \right)^8 - (1+2^{-16})\cdot\frac{1}{2} f_1^{-4} f_2^{-1} d^{-5} \epsilon^8}} \le 0.27 f_2^{-1} d^{-3} \epsilon^4 . 
 \end{aligned}
\end{equation}

By (\ref{subsection: random gaussian initialization}), 

\begin{equation}
 \begin{aligned}
  &\sigma_{\min} \left( W_1(T_1)^H W_2(T_1)^H W_2(T_1) W_1(T_1) + W(T_1)^H\right) \\ 
  \ge& \sigma_{\min} \left( \left( W(T_0)^H W(T_0) \right)^{1/2} + W(0)^H \right) \\ -& \left\| W_1(T_1)^H W_2(T_1)^H W_2(T_1) W_1(T_1) - \left( W(T_0)^H W(T_0) \right)^{1/2} \right\|_{op} - \left\| \Delta_W(T_1) \right\|_{op} \\  
  \ge& 0.72 f_2^{-1} d^{-3} \epsilon^4 , 
 \end{aligned}
\end{equation}

which further gives 

\begin{equation}
 \begin{aligned}
  &\left. \sigma_{\min} \left(W_1 + W_1^\prime \right) \right|_{t=T_1} \\
  =& \sigma_{\min} \left( \left(W_1(T_1)^H W_2(T_1)^H W_2(T_1)\right)^{-1}\left( W_1(T_1)^H W_2(T_1)^H W_2(T_1) W_1(T_1) + W(T_1)^H\right) \right) \\ 
  \ge& \left(\frac{1}{\max_{j,k}|\sigma_k(W_j(T_1))|}\right)^3 \cdot \sigma_{\min} \left( W_1(T_1)^H W_2(T_1)^H W_2(T_1) W_1(T_1) + W(T_1)^H\right) \\ 
  \ge& \frac{\epsilon}{2 f_1^3 f_2 d^{9/2}} . 
 \end{aligned}
\end{equation}

\end{proof}

\subsection{Stage 2: saddle avoidance stage}

In this stage, we further assume $a \ge 32 f_1^{20} f_2 d^{13}\sigma_1(\Sigma) b$, where $b \ge \left( 5\ln\left( \frac{\sigma_1^{1/4}(\Sigma)}{\epsilon} \right) + \frac{281}{8} \ln d + 23\ln(4 f_1) + 7\ln f_2 \right)$. Meanwhile,  $\frac{\epsilon}{\sigma_1^{1/4}(\Sigma)} \le \frac{1}{32 f_1^5 f_2 d^{53/8}} $. 

From Theorem \ref{stage 1: alignment stage, GD}, for $\eta = O\left( a f_1^{-56} f_2^{-14} d^{-301/4} \epsilon^8 \sigma_1^{-9/2}(\Sigma) \right)$ with appropriate small constant, 

\begin{equation}\label{ineq, e, stage 2, GD}
 \begin{aligned}
  e_{\Delta}(T_1) &\le 2\sqrt{3 f_1^4 d^3 \epsilon^4 e^{-2b} + \eta O\left(a^{-1} f_1^{14} d^8 \epsilon^6 \sigma_1^2(\Sigma) \right)} \\
  &\le 2^{-44} f_1^{-21} f_2^{-7} d^{-269/8} \epsilon^7 \sigma_1^{-5/4}(\Sigma) . 
 \end{aligned}
\end{equation}

Moreover, $b - \ln b \ge 3\ln\left(\frac{\sigma_1^{1/4}(\Sigma)}{\epsilon}\right) + \frac{303}{8}\ln d + 37 \ln(2f_1) + 6 \ln f_2 $. Thus for $\eta = O\left( a^{-1} f_1^{-44} f_2^{-10} d^{-219/4} \epsilon^4 \sigma_1^{-3/2}(\Sigma) \right)$ with appropriate small constant, 

\begin{equation}\label{ineq, a e, stage 2, GD}
 \begin{aligned}
  a e_{\Delta}(T_1) &\le 2\sqrt{3\cdot 2^{10} f_1^{44} f_2^2 d^{29} \epsilon^4 \sigma_1^2(\Sigma) \exp(-2(b- \ln b)) + \eta O\left(a f_1^{14} d^8 \epsilon^6 \sigma_1^2(\Sigma) \right)} \\
  &\le 2^{-30} f_1^{-15} f_2^{-5} d^{-187/8} \epsilon^5 \sigma_1^{1/4}(\Sigma) . 
 \end{aligned}
\end{equation}

\begin{theorem}\label{stage 2: saddle avoidance stage, gd}

At $T_1 + T_2$, $T_2 = \frac{32 f_1^6 f_2^2 d^9}{\eta \sigma_1(\Sigma) \epsilon^2}$, the following conclusions hold: 

\begin{equation}
 \begin{aligned}
  \left \| W_1(T_1+T_2) - W_1^\prime(T_1+T_2) \right \|_F &\le 3 f_1 d \epsilon \\
  \sigma_{\min}(W_1 + W_1^\prime)(T_1+T_2) &\ge  2^{3/4} \sigma_1^{1/4}(\Sigma) . 
 \end{aligned}
\end{equation}
    
\end{theorem}

\begin{lemma}\label{L ori non-increasing, GD}

$\mathcal{L}_{\rm ori}$ is approximately non-increasing. 

For $t\in[0,+\infty)$, suppose $\left\| W_{j\in[1,N] \cap \mathbb{N}^*}(t) \right\|_{op} \le M$, then 

\begin{equation}
 \begin{aligned}
  \mathcal{L}_{\rm ori}(t+1) - \mathcal{L}_{\rm ori}(t) &\le - 2\eta N \min_{j,k} |\sigma_{k}(W_j(t))|^{2(N-1)} \mathcal{L}_{\rm ori}(t) \\
  &+ \eta^2 O\left( M^8 \left(M^4 + \sqrt{\mathcal{L}_{\rm ori}(t)} \right) \mathcal{L}_{\rm ori}(t) + a M^4 \sqrt{\mathcal{L}_{\rm ori}(t)} \mathcal{L}_{\rm reg}(t)  \right) \\
  &+ \eta^4 O\left( M^{16} \mathcal{L}_{\rm ori}(t)^2 + a^2 M^8 \mathcal{L}_{\rm reg}(t)^2 \right) . 
 \end{aligned}
\end{equation}

\end{lemma}

\begin{proof}

Following the continuous case (\ref{equation: dW/dt is irrelevant to a}), the change of product matrix satisfy 

\begin{equation}\label{equation: delta W is irrelevant to a, gd}
 \begin{aligned}
  &\left\| W(t+1)-W(t) - \eta \sum_{j=1}^{N} W_{\prod_L , j+1}(t) W_{\prod_L , j+1}(t)^H \left(\Sigma - W(t) \right) W_{\prod_R , j-1}(t)^H W_{\prod_R , j-1}(t) \right\|_{F} \\=& \eta^2 O\left( \max_{j \in [1,4]\cap \mathbb{N}^*} \left\| \nabla_{W_j} \mathcal{L}(t) \right\|_{F}^2 \cdot \max_{j \in [1,4]\cap \mathbb{N}^*} \|W_j(t)\|_{op}^2 \right) . 
 \end{aligned}
\end{equation}

Then 

\begin{equation}
 \begin{aligned}
  \mathcal{L}_{\rm ori}(t+1) - \mathcal{L}_{\rm ori}(t) &= - \Re\left(\left\langle \Sigma - \frac{W(t+1) + W(t)}{2}, W(t+1) - W(t) \right\rangle\right) \\ 
  &= - \eta \sum_{j=1}^{N} \left\| W_{\prod_L , j+1}(t)^H \left(\Sigma - W(t) \right) W_{\prod_R , j-1}(t)^H \right\|_F^2  \\ 
  &+ \eta^2 O\left( M^2 \sqrt{\mathcal{L}_{\rm ori}(t)} \cdot \max_{j \in [1,4]\cap \mathbb{N}^*} \left\| \nabla_{W_j} \mathcal{L}(t) \right\|_{F}^2 \right) \\
  &+ \eta^2 O\left( M^{6} \cdot \max_{j \in [1,4]\cap \mathbb{N}^*} \left\| \nabla_{W_j} \mathcal{L}_{\rm ori}(t) \right\|_{F}^2 \right) \\
  &+ \eta^4 O\left( M^4 \cdot \max_{j \in [1,4]\cap \mathbb{N}^*} \left\| \nabla_{W_j} \mathcal{L}(t) \right\|_{F}^4 \right) \\
  &\le - 2\eta N \min_{j,k} |\sigma_{k}(W_j(t))|^{2(N-1)} \mathcal{L}_{\rm ori}(t) \\
  &+ \eta^2 O\left( M^8 \left(M^4 + \sqrt{\mathcal{L}_{\rm ori}(t)} \right) \mathcal{L}_{\rm ori}(t) + a M^4 \sqrt{\mathcal{L}_{\rm ori}(t)} \mathcal{L}_{\rm reg}(t)  \right) \\
  &+ \eta^4 O\left( M^{16} \mathcal{L}_{\rm ori}(t)^2 + a^2 M^8 \mathcal{L}_{\rm reg}(t)^2 \right) . 
 \end{aligned}
\end{equation}

\end{proof}

Below we further assume $\eta = O\left(\min \left( f_1^{-27} f_2^{-9} d^{-355/8} \epsilon^9 \sigma_1^{-15/4}(\Sigma), a^{-1}f_1^{-21} f_2^{-7} d^{-273/8} \epsilon^7 \sigma_1^{-9/4}(\Sigma) \right)\right)$ with appropriate small constant. 

\begin{lemma}\label{bound of w_j op, GD}

Bound of operator norms. 

For $t\in[T_1, T_1+T_2]$, 

\begin{equation}
 \begin{aligned}
  \|\Sigma - W(t)\|_{F} \le& 1.01 \sqrt{d} \sigma_1(\Sigma) \\ 
  e_{\Delta} (t) \le& 1.01 \cdot 2^{-44} f_1^{-21} f_2^{-7} d^{-269/8} \epsilon^7 \sigma_1^{-5/4}(\Sigma) \\ 
  a e_{\Delta} (t) \le& 1.01 \cdot 2^{-30} f_1^{-15} f_2^{-5} d^{-187/8} \epsilon^5 \sigma_1^{1/4}(\Sigma) \\
  \|W\|_{op} \le \|W\|_{F} \le& 3 \sqrt{d} \sigma_1(\Sigma) \\ 
  \max_{j} \|W_j\|_{op} \le \max_{j} \|W_j\|_{F} \le& \sqrt{2} d^{1/8} \sigma_1^{1/4}(\Sigma) . 
 \end{aligned}
\end{equation}

\end{lemma}

\begin{proof}

We first prove that if the first three inequalities hold at some time $t$, then the rest follows. Then we prove the first three by mathematical induction. 

1. For some $t$, it the first two hold, then 

\begin{equation}\label{proof for bound of w_j op, contradiction, GD}
  \|W(t)\|_{op} \le \|W(t)\|_{F} \le \|\Sigma - W(t)\|_{F} + \|\Sigma\|_{F} \le 3 \sqrt{d} \sigma_1(\Sigma) . 
\end{equation}

For the last inequality, prove by contradiction. (Omit $t$ here)

Suppose $\max_{j} \|W_j\|_{op} \ge \sqrt{2} d^{1/8} \sigma_1^{1/4}(\Sigma)$, then 

\begin{equation}
  e_{\Delta} (t) \le 1.01 e_{\Delta} (T_1) \le 2^{-15} \max_{j} \|W_j\|_{op}^2 . 
\end{equation}

Thus for $t > T_1$, 

\begin{equation}
 \begin{aligned}
  \|W\|_{op}^2 &= \left\|W_4 W_3 W_2 W_1 W_1^H W_2^H W_3^H W_4^H \right\|_{op} \\
  &\ge \left\|W_4 W_4^H\right\|_{op} - \left\|W_4 W_3 W_2 \Delta_{12} W_2^H W_3^H W_4^H \right\|_{op} \\&- \left\|W_4 W_3 \Delta_{23} W_2 W_2^H W_3^H W_4^H \right\|_{op} - \left\|W_4 W_3 W_2 W_2^H \Delta_{23} W_3^H W_4^H \right\|_{op} \\&- \left\|W_4 \Delta_{34} \left(W_3 W_3^H\right)^2 W_4^H \right\|_{op} - \left\|W_4 W_3 W_3^H \Delta_{34} W_3 W_3^H W_4^H \right\|_{op} - \left\|W_4 \left(W_3 W_3^H\right)^2 \Delta_{34} W_4^H \right\|_{op} \\ 
  &\ge \left( \max_{j} \|W_j\|_{op}^2 - 3 e_{\Delta} \right)^4 - 6 e_{\Delta} \max_{j} \|W_j\|_{op}^6 > 15 \sqrt{d} \sigma_1(\Sigma) , 
 \end{aligned}
\end{equation}

which contradicts inequality (\ref{proof for bound of w_j op, contradiction, GD}). 

2. Mathematical induction. 

For $t=T_1$, 

\begin{equation}
  \|\Sigma - W(T_1)\|_{F} \le \|\Sigma\|_{F} + \| W(T_1)\|_{F} \le \left(1+2^{-39} \right)\sqrt{d} \sigma_1(\Sigma) . 
\end{equation}

Suppose for $t^\prime \in [T_1,t]$ ($T_1 \le t < T_2$), the first two properties hold. Denote $M = \max_j \|W_j(t^\prime \in [T_1,t])\|_{op}$. By invoking Lemma \ref{L ori non-increasing, GD} and \ref{regularization term, convergence bound, GD}, at $t+1$, 

\begin{equation}
 \begin{aligned}
  \mathcal{L}_{\rm ori}(t+1) &= \mathcal{L}_{\rm ori}(T_1) + \eta^2 (t - T_1) O\left( M^8 \left(M^4 + \sqrt{\mathcal{L}_{\rm ori}(T_1)} \right) \mathcal{L}_{\rm ori}(T_1) + a M^4 \sqrt{\mathcal{L}_{\rm ori}(T_1)} \mathcal{L}_{\rm reg}(T_1)  \right) \\
  &+ \eta^4 (t - T_1) O\left( M^{16} \mathcal{L}_{\rm ori}(T_1)^2 + a^2 M^8 \mathcal{L}_{\rm reg}(T_1)^2 \right) \\ 
  &= \mathcal{L}_{\rm ori}(T_1) + \eta^2 T_2 O\left( d^2 \sigma_1(\Sigma)^4  + d \sigma_1(\Sigma)^2 (ae_\Delta(T_1))^2  \right)  \le 1.01^2 \sqrt{d} \sigma_1(\Sigma) . 
 \end{aligned}
\end{equation}

Note that $\mathcal{L}_{\rm ori} = \frac{a}{4} e_\Delta^2$. 
Under $\eta = O\left(\min \left( f_1^{-27} f_2^{-9} d^{-355/8} \epsilon^9 \sigma_1^{-15/4}(\Sigma), a^{-1}f_1^{-21} f_2^{-7} d^{-273/8} \epsilon^7 \sigma_1^{-9/4}(\Sigma) \right)\right)$ with appropriate small constant, 

\begin{equation}
 \begin{aligned}
  \mathcal{L}_{\rm reg}(t+1) &\le \mathcal{L}_{\rm reg}(T_1) + \eta^2 (t - T_1) O\left(a^2 M^4 \mathcal{L}_{\rm reg}(t) + \sqrt{a \mathcal{L}_{\rm reg}(t)} M^6 \mathcal{L}_{\rm ori}(t) \right) \\ &+ \eta^4 (t - T_1) O\left(a M^{12} \mathcal{L}_{\rm ori}(t)^2 + a^3 M^4 \mathcal{L}_{\rm reg}(t)^2 \right) \\ 
  &\le \mathcal{L}_{\rm reg}(T_1) + \eta^2 T_2 O\left( \sqrt{a \mathcal{L}_{\rm reg}(t)} M^6 \mathcal{L}_{\rm ori}(t) \right) + \eta^4 T_2 O\left( a M^{12} \mathcal{L}_{\rm ori}(t)^2 \right) \\ 
  &\le \frac{1.01^2}{4} \min\left( a \cdot \left[2^{-44} f_1^{-21} f_2^{-7} d^{-269/8} \epsilon^7 \sigma_1^{-5/4}(\Sigma) \right]^2, \frac{1}{a} \cdot \left[ 2^{-30} f_1^{-15} f_2^{-5} d^{-187/8} \epsilon^5 \sigma_1^{1/4}(\Sigma) \right]^2 \right) . 
 \end{aligned}
\end{equation}

This completes the proof. 

\end{proof}

\begin{lemma}\label{bound of w23 inv op, and relevant terms, GD}

Bound of $\left \| W_2^{-1} \right \|_{op}$ and relevant term. 

For $t\in[T_1, T_1 + T_2]$, 

\begin{equation}\label{ineq: inverse op norm are bounded, GD}
  \left\| W_2^{-1} (t) \right\|_{op} \le 128 f_1^{6} f_2^{2} d^{77/8} \epsilon^{-2} \sigma_1^{1/4}(\Sigma) , 
\end{equation}

\begin{equation}
  e_\Delta (t) \left\| W_2^{-1} (t) \right\|_{op}^2 \le 1.01 \cdot 2^{-30} f_1^{-9} f_2^{-3} d^{-115/8} \epsilon^{3} \sigma_1^{-3/4}(\Sigma) . 
\end{equation}

\end{lemma}

\begin{proof}

We begin with the update of $W_2^{-1}$. From Lemma \ref{error bound, inverse}, 

\begin{equation}
 \begin{aligned}
  &\left\| W_2^{-1}(t+1) - W_2^{-1}(t) \right. \\ -&\left. \eta\left[  - R(t) W_4(t)^H (\Sigma - W(t)) W_1(t)^H W_2(t)^{-1} - a \Delta_{12}(t) W_2(t)^{-1} + a W_2(t)^{-1} \Delta_{23}(t) \right] \right\|_{op} \\ \le& \eta^2 \| W_2(t)^{-1} \|_{op}^2 \| W_2(t+1)^{-1} \|_{op} \| \nabla_{W_2} \mathcal{L}(t) \|_{op}^2 . 
 \end{aligned}
\end{equation}

By triangular inequality, 

\begin{equation}
 \begin{aligned}
  \left\|W_2(t+1)^{-1}\right\|_{op} - \left\|W_2(t)^{-1}\right\|_{op} &\le \eta\left\| R(t) \right\|_{op} \left\| W_4(t) \right\|_{op} \left\| \Sigma - W(t) \right\|_{op} \left\| W_1(t)^H W_2(t)^{-1} \right\|_{op} \\&+ \eta a \left\| \Delta_{12}(t) \right\|_{op} \left\| W_2(t)^{-1} \right\|_{op} + \eta a \left\| W_2(t)^{-1} \right\|_{op} \left\| \Delta_{23}(t) \right\|_{op} \\&+ \eta^2 \| W_2(t)^{-1} \|_{op}^2 \| W_2(t+1)^{-1} \|_{op} \| \nabla_{W_2} \mathcal{L}(t) \|_{op}^2 . 
 \end{aligned}
\end{equation}

From 

\begin{equation}
 \begin{aligned}
  \left\| R \right\|_{op} &\le \sqrt{1 + \frac{1}{\sigma_{\min}^2(W_2)} \cdot \| \Delta_{23} \|_{op}} \\ 
  \left\| W_1^H W_2^{-1} \right\|_{op} &= \sqrt{\left\| W_2^{H-1} W_1 W_1^H W_2^{-1} \right\|_{op}} = \sqrt{\left\| I + W_2^{H-1} \Delta_{12} W_2^{-1} \right\|} \le \sqrt{1+e_\Delta\left\| W_2^{-1} \right\|_{op}^2} . 
 \end{aligned}
\end{equation}

Further we have 

\begin{equation}
 \begin{aligned}
  \left\|W_2(t+1)^{-1}\right\|_{op} - \left\|W_2(t)^{-1}\right\|_{op} &\le 2\sqrt{2} \eta \left( 1 + e_\Delta(t) \left\|W_2(t)^{-1}\right\|_{op}^2 \right) d^{5/8} \sigma_1^{5/4}(\Sigma) \\&+ \sqrt{2} \eta a e_\Delta(t) \left\|W_2(t)^{-1}\right\|_{op} \\ &+ \eta^2 O\left(  \| W_2(t)^{-1} \|_{op}^2 \| W_2(t+1)^{-1} \|_{op} \| \nabla_{W_2} \mathcal{L}(t) \|_{op}^2 \right) . 
 \end{aligned}
\end{equation}

Combine with Lemma \ref{bound of w_j op, GD}, for $t \ge T_1$ such that (\ref{ineq: inverse op norm are bounded, GD}) holds, 

\begin{equation}
 \begin{aligned}
  & \left\|W_2(t+1)^{-1}\right\|_{op} - \left\|W_2(t)^{-1}\right\|_{op} \\
  \le& 2\sqrt{2} (1+1.01\cdot 2^{-30}) \eta  d^{5/8} \sigma_1^{5/4}(\Sigma) + 2^{-22} \eta f_1^{-9} f_2^{-3} d^{-55/4} \epsilon^{3} \sigma_1^{1/2}(\Sigma) \\+& \eta^2 O\left( f_1^{18} f_2^6 d^{245/8} \epsilon^{-6} \sigma_1^{17/4}(\Sigma) \right) \\
  \le& 2\sqrt{2} (1 + 2^{-20}) \eta d^{5/8} \sigma_1^{5/4}(\Sigma) . 
 \end{aligned}
\end{equation}

From Theorem \ref{stage 1: alignment stage, GD}, $\max\left( \left\|W_2(T_1)^{-1}\right\|_{op} ,\, \left\|W_3(T_1)^{-1}\right\|_{op} \right) \le \frac{1}{\min_{j,k}|\sigma_k(W_j(T_1))|} \le \frac{f_1\sqrt{d}}{(1-2^{-17})\epsilon}$, then the proof of the first inequality is completed via integration during the time interval $[T_1, T_1 + T_2]$. The second inequality follows immediately. 

\end{proof}

\begin{remark}

This Lemma verifies that $W_{2,3}^{-1}$ are bounded (consequently $W_{2,3}$ are full rank), then $R$ is well defined throughout this stage. For $t > T_1 + T_2$, further analysis shows that the minimum singular values of $W_2$ and $W_3$ are lower bounded by $\Omega(\sigma_1^{1/4}(\Sigma))$. 

\end{remark} 

Now we begin the proof of Lemma \ref{stage 2, skew-hermitian error, GD} and \ref{stage 2, main term, GD}. 

Proof for Lemma \ref{stage 2, skew-hermitian error, GD}: 

\begin{proof}\label{stage 2, skew-hermitian error, GD, proof}

From Lemma \ref{bound of w23 inv op, and relevant terms, GD}, for $t\in[T_1, T_1+T_2]$, 

\begin{equation}
 \begin{aligned}
  \max\left(\left\| R^{H} R - I \right\|_{op},\, \left\| I - R R^H \right\|_{op} \right) &\le e_\Delta\left\| W_2^{-1} \right\|_{op}^2 \\&\le 1.01 \cdot 2^{-30} f_1^{-9} f_2^{-3} d^{-115/8} \epsilon^{3} \sigma_1^{-3/4}(\Sigma) , 
 \end{aligned}
\end{equation}

\begin{equation}
 \begin{aligned}
  \left\| M_1 - M_1^\prime \right\|_{op} &\le \sqrt{6} \cdot \frac{\max_{j,k} \sigma_{k}^2(W_j)}{ \sigma_{\min}^2(W_2)} e_{\Delta} \\
  &\le 2^{-27} f_1^{-9} f_2^{-3} d^{-113/8} \epsilon^{3} \sigma_1^{-1/4}(\Sigma) , 
 \end{aligned}
\end{equation}

\begin{equation}
 \begin{aligned}
  \left\|M_2 - \frac{M_1 + M_1^\prime}{2} \right\|_{op} &\le \left\| \Delta_{12} \right\|_{op} + \frac{1}{2} \left\|M_1 - M_1^\prime \right\|_{op} \le \left[1 + \frac{\sqrt{6}}{2} \cdot \frac{\max_{j,k} \sigma_{k}^2(W_j)}{ \sigma_{\min}^2(W_2)} \right] e_{\Delta} \\
  &\le 2^{-28} f_1^{-9} f_2^{-3} d^{-113/8} \epsilon^{3} \sigma_1^{-1/4}(\Sigma) . 
 \end{aligned}
\end{equation}

Consequently: 

\begin{equation}
 \begin{aligned}
  \left\| R\right\|_{op} &\le \sqrt{1+e_\Delta\left\| W_2^{-1} \right\|_{op}^2} \le 1 + 1.01 \cdot 2^{-31} f_1^{-9} f_2^{-3} d^{-115/8} \epsilon^{3} \sigma_1^{-3/4}(\Sigma) , 
 \end{aligned}
\end{equation}

\begin{equation}
  \left\| W_1^\prime \right\|_{op} \le \left\| W_1^\prime \right\|_{F} \le \sqrt{2} d^{1/8} \sigma_1^{1/4}(\Sigma) \left\| R\right\|_{op} \le \left(1+ 1.01 \cdot 2^{-31}\right)\sqrt{2} d^{1/8} \sigma_1^{1/4}(\Sigma) , 
\end{equation}

\begin{equation}
  \left\| \frac{M_1 + M_1^\prime}{2} \right\|_{op} \le \left\| M_2 \right\|_{op} + \left\|M_2 - \frac{M_1 + M_1^\prime}{2} \right\|_{op} \le \left(1 + 2^{-29}\right)2 d^{1/4} \sigma_1^{1/2}(\Sigma) , 
\end{equation}

\begin{equation}
 \begin{aligned}
  \left\| M_1^\prime M_2 M_1 - M_1 M_2 M_1^\prime \right\|_{op} &\le \left\| M_1 - M_1^\prime \right\| \left\| M_2 \right \| \left\| M_1 + M_1^\prime \right\| \\ &\le \left(1 + 2^{-29}\right) 2^{-25} f_1^{-9} f_2^{-3} d^{-109/8} \epsilon^{3} \sigma_1^{3/4}(\Sigma) . 
 \end{aligned}
\end{equation}

By combining all results above, for $t \in [T_1, T_1 + T_2-1]$ such that $\left \| W_1 - W_1^\prime \right \|_F \le 3 f_1 d \epsilon$ holds, 

\begin{equation}
 \begin{aligned}
  &\left \| W_1(t+1) - W_1^\prime(t+1) \right \|_F^2 - \left \| W_1(t) - W_1^\prime(t) \right \|_F^2 \\
  \le& -2\eta \sigma_1(\Sigma)\sigma_{\min}(W_2)^2 \left \| W_1(t) - W_1^\prime(t) \right \|_F^2 \\
  +& \eta \|M_2(t)\|_F \left\| M_1^\prime(t) - M_1(t)\right\|_{op} \|M_2(t)\|_{op} \left( \left\| W_1^\prime(t) \right\|_{op} + \left\| W_1(t) \right\|_{op} \right) \left\| W_1(t) - W_1^\prime(t) \right\|_F \\
  +& 2 \eta \left\| - M_1^\prime(t) M_2(t) M_1(t) + M_1(t) M_2(t) M_1^\prime(t) \right\|_{op} \left\| W_1^\prime(t) \right\|_{F} \left\| W_1(t) - W_1^\prime(t) \right\|_F \\
  +& 2 \eta \max_j \|W_j(t)\|_{op}^3 \|\Sigma - W(t)\|_F \left( \left\| R(t)^{H} R(t) - I \right\|_{op} + \left\|I - R(t) R(t)^H \right\|_{op} \right) \left\| W_1(t) - W_1^\prime(t) \right\|_F \\
  +& 2\eta a e_\Delta(t) \left\|W_1(t) - W_1^\prime(t) \right\|_F^2  \\
  +& 4\eta a e_\Delta(t) \left\|W_2(t)^{-1}\right\|_{op} \|W_2(t)\|_F \left\| W_1^\prime(t) \right\|_{op} \left\| W_1(t) - W_1^\prime(t) \right\|_F \\
  +& \eta^2 O\left(\left[ \max_{j \in [1,4]\cap \mathbb{N}^*} \left\| W_j(t) \right\|_{op} \| \Sigma - W(t)\|_F + a e_\Delta(t) \left\|W_2(t)^{-1}\right\|_{op} \right]^2 \right. \\ \cdot & \left. \max_{j \in [1,4]\cap \mathbb{N}^*} \left\| W_j(t) \right\|_{op}^5 \cdot \|W_2(t+1)^{-1}\|_{op} \right)  \\
  \le& -2\eta \sigma_1(\Sigma)\sigma_{\min}(W_2)^2 \left \| W_1(t) - W_1^\prime(t) \right \|_F^2 + 2^{-17} \eta f_1^{-8} f_2^{-3} d^{-25/2} \epsilon^4 \sigma_1(\Sigma) . 
 \end{aligned}
\end{equation}

From Theorem \ref{stage 1: alignment stage, GD}, at $t = T_1$, 

\begin{equation}
 \begin{aligned}
  \left \| W_1(T_1) - W_1^\prime(T_1) \right \|_F &\le \| W_1(T_1) \|_F + \left \| W_1^\prime(T_1) \right \|_F \le \| W_1(T_1) \|_F + \left \| W_4(T_1) \right \|_F \|R(T_1)\|_{op} \\
  &\le \left(1+2^{-20}\right) 2 f_1 d \epsilon . 
 \end{aligned}
\end{equation}

Thus $\left \| W_1(t) - W_1^\prime(t) \right \|_F^2 \le \sqrt{\left[\left(1+2^{-20}\right) 2 f_1 d \epsilon \right]^2 + 2^{-17} f_1^{-8} f_2^{-3} d^{-25/2} \epsilon^4 \sigma_1(\Sigma) \eta (t - T_1)}$, when both $t \in [T_1, T_1 + T_2]$ and $\left \| W_1(t) - W_1^\prime(t) \right \|_F^2 \le 3 f_1 d \epsilon$ hold. Then 

\begin{equation}
 \begin{aligned}
  \left \| W_1(T_1 + T_2) - W_1^\prime(T_1 + T_2) \right \|_F^2 &\le \sqrt{\left[\left(1+2^{-20}\right) 2 f_1 d \epsilon \right]^2 + 2^{-17} f_1^{-8} f_2^{-3} d^{-25/2} \epsilon^4 \sigma_1(\Sigma) \eta T_2} \\
  &\le \sqrt{\left[\left(1+2^{-20}\right) 2 f_1 d \epsilon \right]^2 + 2^{-12} f_1^{-2} f_2^{-1} d^{-7/2} \epsilon^2} < 3 f_1 d \epsilon , 
 \end{aligned}
\end{equation}

which completes the proof. 

\end{proof}

Proof for Lemma \ref{stage 2, main term, GD}: 

\begin{proof}\label{stage 2, main term, GD, proof}

We analyze the dynamics of $\lambda_{\min}\left(\left(W_1 + W_1^\prime \right)^H\left(W_1 + W_1^\prime \right)\right) = \sigma_{\min}^2 $. 

From $\left\|M_2 - \frac{M_1 + M_1^\prime}{2} \right\|_{op} \le 2^{-28} f_1^{-9} f_2^{-3} d^{-113/8} \epsilon^{3} \sigma_1^{-1/4}(\Sigma)$ and $\left\| \frac{M_1 + M_1^\prime}{2} \right\|_{op}  \le \left(1 + 2^{-29}\right) 2 d^{1/4} \sigma_1^{1/2}(\Sigma)$, define 

\begin{equation}
 \begin{aligned}
  E(t) \coloneqq \sigma_1(\Sigma)\left(M_2(t) - \frac{M_1(t) + M_1^\prime(t)}{2}\right) - \left(M_2(t) \left(\frac{M_1(t) + M_1^\prime(t)}{2}\right) M_2(t) - \left(\frac{M_1(t) + M_1^\prime(t)}{2}\right)^3 \right) . 
 \end{aligned}
\end{equation}

Then 

\begin{equation}
 \begin{aligned}
  \|E(t)\|_{op} &\le 2^{-28} f_1^{-9} f_2^{-3} d^{-113/8} \epsilon^{3} \sigma_1^{3/4}(\Sigma) + \left(1 + 2^{-28}\right)2^{-24} f_1^{-9} f_2^{-3} d^{-109/8} \epsilon^3 \sigma_1^{3/4}(\Sigma) \\ 
  &\le \left(1 + 2^{-4} + 2^{-28}\right)2^{-24} f_1^{-9} f_2^{-3} d^{-109/8} \epsilon^3 \sigma_1^{3/4}(\Sigma) . 
 \end{aligned}
\end{equation}

By Lemma \ref{stage 2, skew-hermitian error, GD}, $\left\|W_1 - W_1^{\prime}\right\|_{op} \le \left\|W_1 - W_1^{\prime}\right\|_{F} \le 3f_1 d \epsilon $, and under $\sigma_{\min}(t) \ge \frac{\epsilon}{2 f_1^3 f_2 d^{9/2}}$, 

\begin{equation}
 \begin{aligned}
  \sigma_{\min}(t+1)^2 &\ge \lambda_{\min} \left( W_{\rm new}(t)^H W_{\rm new}(t) \right) - 2^{-18} \sigma_1(\Sigma) \sigma_{\min}(t)^4 , 
 \end{aligned}
\end{equation}

where 

\begin{equation}
 \begin{aligned}
  W_{\rm new}(t) = \left(I + \eta\left[ \sigma_1(\Sigma) \left( \frac{M_1(t) + M_1^\prime(t)}{2} \right)  - \left( \frac{M_1(t) + M_1^\prime(t)}{2} \right)^3 + E(t) \right]\right) \left( W_1(t) + W_1^\prime(t) \right) . 
 \end{aligned}
\end{equation}

Denote $P = \frac{W_1 + W_1^{\prime}}{2}$, $Q = \frac{W_1 - W_1^{\prime}}{2}$. Notice that $PP^H + QQ^H = \frac{M_1 + M_1^\prime}{2}$. 
Then by invoking Lemma \ref{minimum singular values lower bound, general, discrete} (omit $t$ here) the first term becomes 

\begin{equation}
 \begin{aligned}
  \lambda_{\min} \left( W_{\rm new}^H W_{\rm new} \right) &= \lambda_{\min} \left( W_{\rm new} W_{\rm new}^H \right) \\
  &= 4\lambda_{\min} \left( \left(I + \eta\left[ \sigma_1(\Sigma) \left( PP^H + QQ^H \right)  - \left( PP^H + QQ^H \right)^3 + E \right]\right) PP^H \right. \\ &\cdot \left. \left(I + \eta\left[ \sigma_1(\Sigma) \left( PP^H + QQ^H \right)  - \left( PP^H + QQ^H \right)^3 + E \right]\right) \right) \\ 
  &\ge \sigma_{\min}^2 + 8\eta \left(\sigma_1(\Sigma) - 2 \|Q\|_{op}^2 \left\| \frac{M_1 + M_1^\prime}{2} \right\|_{op} \right) \left(\frac{\sigma_{\min}^2}{4}\right)^2 \\
  &- 8 \eta \left\| \frac{M_1 + M_1^\prime}{2} \right\|_{op} \left(\frac{\sigma_{\min}^2}{4}\right)^3 \\
  &- 8\eta \left( \|E\|_{op} + \|Q\|_{op}^4 \left\| \frac{M_1 + M_1^\prime}{2} \right\|_{op} \right) \left(\frac{\sigma_{\min}^2}{4}\right) \\
  &+ \eta^2 O\left( \left( \sigma_1(\Sigma)^2 \left\| \frac{M_1 + M_1^\prime}{2} \right\|_{op}^2 + \left\| \frac{M_1 + M_1^\prime}{2} \right\|_{op}^6 + \|E\|_{op}^2 \right) \left\| \frac{M_1 + M_1^\prime}{2} \right\|_{op} \right) . 
 \end{aligned}
\end{equation}

Notice $\|Q\|_{op} = \frac{1}{2} \left\|W_1 - W_1^{\prime}\right\|_{F} \le \frac{3}{2} f_1 d \epsilon \le \sigma_k \cdot 3 f_1^4 f_2 d^{11/2}$, $\epsilon \le \frac{1}{32 f_1^5 f_2 d^{53/8}} \sigma_1^{1/4}(\Sigma)$, then under $\sigma_{\min}(t) \ge \frac{\epsilon}{2 f_1^3 f_2 d^{9/2}}$, 

\begin{equation}
 \begin{aligned}
  \sigma_{\min}(t+1)^2 &\ge \sigma_{\min}(t)^2 + (2^{-1} - 81(1+2^{-4})2^{-10})\eta \sigma_1(\Sigma) \sigma_{\min}(t)^4 - \frac{1}{32} \eta \sigma_{\min}(t)^8 . 
 \end{aligned}
\end{equation}

Notice that $\sigma_{\min}(t)$ is bounded by $O\left(d^{1/8} \sigma_1^{1/4}(\Sigma)\right)$. By taking reciprocal, 

\begin{equation}
 \begin{aligned}
  \frac{1}{\sigma_{\min}(t+1)^2} &\le \frac{1}{\sigma_{\min}(t)^2} + \frac{(2^{-1} - 81(1+2^{-4})2^{-10})\eta \sigma_1(\Sigma) \sigma_{\min}(t)^4 - \frac{1}{32} \eta \sigma_{\min}(t)^8}{\sigma_{\min}(t)^4 + (2^{-1} - 81(1+2^{-4})2^{-10})\eta \sigma_1(\Sigma) \sigma_{\min}(t)^6 - \frac{1}{32} \eta \sigma_{\min}(t)^{10}} \\
  &\le \frac{1}{\sigma_{\min}(t)^2} + \frac{3}{8} \eta \sigma_1(\Sigma) - \frac{1}{32} \eta \sigma_{\min}(t)^{4} . 
 \end{aligned}
\end{equation}

This indicates that $\sigma_{\min}(t)$ takes at most time $\Delta t^\prime = \frac{1}{\frac{1}{8} \eta \sigma_1(\Sigma)}\left[\frac{1}{\sigma_{\min}(t=0)^2} - \frac{1}{\left(2^{3/4} \sigma_1^{1/4}(\Sigma) \right)^2} \right] <T_2$ to increase to $2^{3/4} \sigma_1^{1/4}(\Sigma)$, and never decrease to less than $2^{3/4} \sigma_1^{1/4}(\Sigma)$ afterwards (in $t\in[T_1 + \Delta t^\prime,T_2]$). 

\end{proof}

\subsection{Stage 3: local convergence stage}\label{stage 3: delta convergence, gd}

In this stage, we analysis the time to reach $\epsilon_{\rm conv}$-convergence, that is 

\begin{equation}
  T(\epsilon_{\rm conv},\eta) = \inf_t \{\mathcal{L}(t) \le \epsilon_{\rm conv} \} . 
\end{equation}

\begin{theorem}\label{local convergence, GD}

Local convergence. 

For $t \in[T_1 + T_2, +\infty)$, 

\begin{equation}
 \begin{aligned}
  \mathcal{L}_{\rm ori}(t) &\le \mathcal{L}_{\rm ori}(T_1 + T_2) \exp\left(- \eta \sigma_1^{3/2}(\Sigma) (t - T_1 - T_2) \right) \\ 
  \mathcal{L}_{\rm reg}(t) &\le l_{\rm reg} \exp\left( - \eta \sigma_1^{3/2}(\Sigma) (t - T_1 - T_2) \right) \\ 
  \sigma_{\min}\left(W_1(t) + W_1^\prime(t) \right) &\ge 2^{3/4} \sigma_1^{1/4}(\Sigma) \\ 
  \left \| W_1(t) - W_1^\prime(t) \right \|_F &\le 3 f_1 d \epsilon , 
 \end{aligned}
\end{equation}

where $\mathcal{L}_{\rm ori}(T_1 + T_2) = \frac{1.01^2}{2}\cdot d \sigma_1^2(\Sigma)$, and $l_{\rm reg} = \min\left( \frac{a}{4}\left(1.01 \cdot 2^{-44} f_1^{-21} f_2^{-7} d^{-269/8} \epsilon^7 \sigma_1^{-5/4}(\Sigma))\right)^2 , \frac{1}{4a} \left( 1.01 \cdot 2^{-30} f_1^{-15} f_2^{-5} d^{-187/8} \epsilon^5 \sigma_1^{1/4}(\Sigma) \right)^2 \right) $. 

\end{theorem}

\begin{proof}\label{local convergence, GD, proof}

Prove by induction. 

At $t=T_2$ these properties holds. 

Suppose at some time $t \in [T_2, +\infty)$ they holds, then follow the same arguments in Lemma \ref{bound of w_j op, GD}, $\max_j \|W_j(t)\|_{op} \le \sqrt{2} d^{1/8} \sigma_1^{1/4}(\Sigma)$. 

To address the bound of $\left\|W_2^{-1} \right\|_{op}$, 

\begin{equation}
 \begin{aligned}
  \left\| \frac{M_1(t) - M_1^\prime(t)}{2} \right\|_{op} &\le \left\| W_1(t) - W_1^\prime(t) \right\|_{op} \left\| \frac{W_1(t) + W_1^\prime(t)}{2} \right\|_{op} \le 8 f_1 d^{9/8} \sigma_1^{1/4}(\Sigma) \epsilon \\
  \left\|M_2(t) - \frac{M_1(t) + M_1^\prime(t)}{2} \right\|_{op} &\le \|\Delta_{12}(t)\|_{op} + \left\| \frac{M_1(t) - M_1^\prime(t)}{2} \right\|_{op} \le 16 f_1 d^{9/8} \sigma_1^{1/4}(\Sigma) \epsilon \\
  \sigma_{\min}(W_2(t)) &= \sqrt{\lambda_{\min}(M_2(t))} \ge \sqrt{\lambda_{\min}\left( \frac{M_1(t) + M_1^\prime(t)}{2} \right) - 16 f_1 d^{9/8} \sigma_1^{1/4}(\Sigma) \epsilon} \\
  &\ge \sqrt{\sigma_{\min}^2\left( \frac{W_1(t) + W_1^\prime(t)}{2} \right) - 16 f_1 d^{9/8} \sigma_1^{1/4}(\Sigma) \epsilon} \ge \frac{1}{2^{3/8}} \sigma_1^{1/4}(\Sigma) . 
 \end{aligned}
\end{equation}

Similarly, $\min_{j,k}(\sigma_k(W_j(t))) \ge \frac{1}{2^{3/8}} \sigma_1^{1/4}(\Sigma)$. 

Then following the derivations in Lemma \ref{stage 2, skew-hermitian error, GD} and \ref{stage 2, main term, GD}, 

\begin{equation}
 \begin{aligned}
  \left \| W_1(t+1) - W_1^\prime(t+1) \right \|_F^2 &\le \left(1-2\eta \sigma_1(\Sigma)\sigma_{\min}(W_2)^2 \right) \left \| W_1(t) - W_1^\prime(t) \right \|_F^2 + 2^{-17} \eta f_1^{-8} f_2^{-3} d^{-25/2} \epsilon^4 \sigma_1(\Sigma) \\ &\le \left(1-\eta \sigma_1^{3/2}(\Sigma) \right) \left \| W_1(t) - W_1^\prime(t) \right \|_F^2 + 2^{-17} \eta f_1^{-8} f_2^{-3} d^{-25/2} \epsilon^4 \sigma_1(\Sigma) \le 3 f_1 d \epsilon \\ 
  \frac{1}{\sigma_{\min}\left(W_1(t+1) + W_1^\prime(t+1) \right)^2} &\le \frac{1}{\sigma_{\min}(t)^2} + \frac{3}{8} \eta \sigma_1(\Sigma) - \frac{1}{32} \eta \sigma_{\min}(t)^{4} < \frac{1}{\left(2^{3/4} \sigma_1^{1/4}(\Sigma)\right)^2} . 
 \end{aligned}
\end{equation}

Then by Theorem \ref{L ori non-increasing, GD} and \ref{regularization term, convergence bound, GD}, 

\begin{equation}
 \begin{aligned}
  \mathcal{L}_{\rm ori}(t+1) &\le \mathcal{L}_{\rm ori}(t) - 2^{3/4}\eta \sigma_{1}^{3/2}(\Sigma) \mathcal{L}_{\rm ori}(t) \\
  &+ \eta^2 O\left( \max_j \|W_j(t)\|_{op}^8 \left(\max_j \|W_j(t)\|_{op}^4 + \sqrt{\mathcal{L}_{\rm ori}(t)} \right) \mathcal{L}_{\rm ori}(t) + a \max_j \|W_j(t)\|_{op}^4 \sqrt{\mathcal{L}_{\rm ori}(t)} \mathcal{L}_{\rm reg}(t)  \right) \\
  &+ \eta^4 O\left( \max_j \|W_j(t)\|_{op}^{16} \mathcal{L}_{\rm ori}(t)^2 + a^2 \max_j \|W_j(t)\|_{op}^8 \mathcal{L}_{\rm reg}(t)^2 \right) \\ 
  &\le \left(1-\eta \sigma_{1}^{3/2}(\Sigma)\right) \mathcal{L}_{\rm ori}(t) , 
 \end{aligned}
\end{equation}

\begin{equation}
 \begin{aligned}
  \mathcal{L}_{\rm reg}(t+1) &\le \left(1 - \frac{1}{3} \eta a d^{-1/4} \sigma_1^{1/2}(\Sigma) \right) \cdot \mathcal{L}_{\rm reg}(t) + \eta^2 O\left(a^2 M^4 \mathcal{L}_{\rm reg}(t) + \sqrt{a \mathcal{L}_{\rm reg}(t)} M^6 \mathcal{L}_{\rm ori}(t) \right) \\ &+ \eta^4 O\left(a M^{12} \mathcal{L}_{\rm ori}(t)^2 + a^3 M^4 \mathcal{L}_{\rm reg}(t)^2 \right) \\ 
  &\le \left(1 - \frac{1}{4} \eta a d^{-1/4} \sigma_1^{1/2}(\Sigma) \right) \cdot \mathcal{L}_{\rm reg}(t) \le \left(1 - \eta a d^{-1/4} \sigma_1^{3/2}(\Sigma) \right) \cdot \mathcal{L}_{\rm reg}(t) . 
 \end{aligned}
\end{equation}

This completes the proof. 

\end{proof}

By Combining the three-stage results, the global convergence guarantee of Theorem \ref{Total convergence bound, gd} is proved.

%% file: subfiles/appendix/appendix_9_numerical_simulations.tex
\section{Numerical Simulations}\label{section: Numerical Simulations}

Through out this section, we consider numerical simulations under four-layer matrix factorization on square matrices with dimension of $5$. 

\subsection{Saddle avoidance dynamics under balance initialization}

This section presents numerical simulations of the saddle avoidance stage under balanced initialization. In this experiment, $\epsilon=0.05$, $\eta = 0.1$, $\Sigma_w(0) = \epsilon \cdot \operatorname{diag}(1, 0.8, 0.6, 0.5, 0.9)$. 

We set the target matrix to $\Sigma = I$ in Figure \ref{fig: balanced init, log singular values, identity target} and to $\Sigma = \operatorname{diag}(2.00, 1.55, 1.10, 0.65, 0.20)$ in Figure \ref{fig: balanced init, log singular values, non-identity target}. Each pair of solid and dashed lines of the same color represents the logarithms of the $k$-th singular value of $\Sigma_W$ and that of $\frac{1}{2}(U+V)\Sigma_W$, respectively.

These figures clearly exhibit the following properties: 

\begin{itemize}
    \item $\sigma_k\left(\frac{1}{2}(U+V)\Sigma_W\right)$ provides a tight lower bound for $\sigma_k\left(\Sigma_W\right)$, verifying the conclusion of Lemma \ref{bound of eigenvalues under perturbation}. 

    \item The eigen-gap of the target matrix introduces non-smoothness and non-monotonicity into the original lower bound for singular values of the product matrix, leading to segmented rather than global smoothness and monotonicity. This explains why the dynamics are easier to analyze when the target matrix is the identity. 

    \item The $1/2$ probability of converging to a saddle point under real balanced initialization is a general phenomenon, even if the target matrix is not identity. However, in the setting of Figure \ref{fig: balanced init, log singular values, non-identity target}, initializations with $\det(U^\top V) = 1$ fail to converge, which contrasts with the identity target case.
    
\end{itemize}

\begin{figure*}[htbp]
\centering
\begin{subfigure}{0.49\linewidth}
    \includegraphics[width=\linewidth]{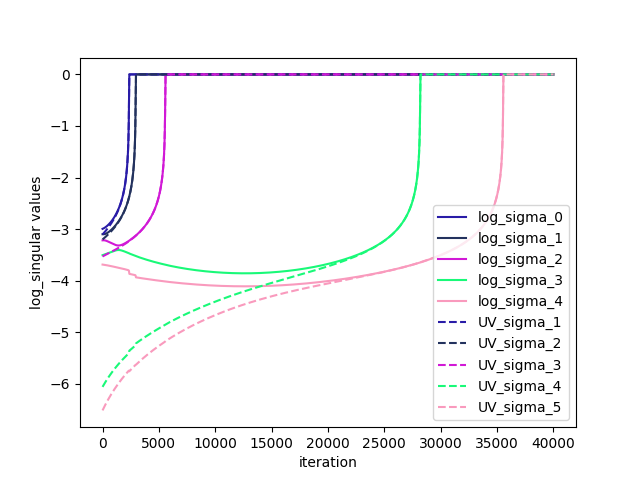} 
\end{subfigure}
\begin{subfigure}{0.49\linewidth}
    \includegraphics[width=\linewidth]{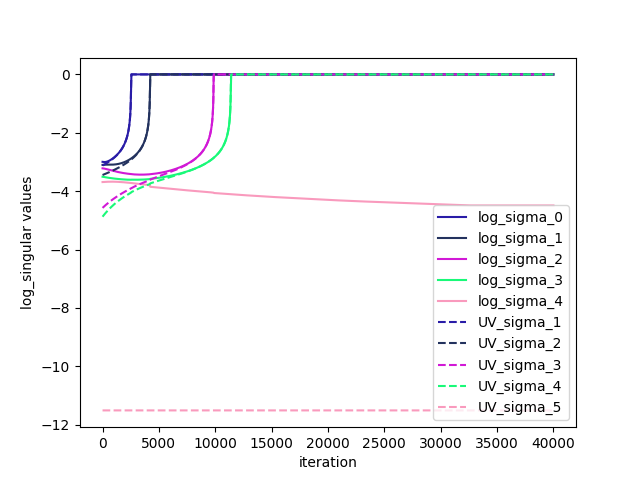} 
\end{subfigure}
\begin{subfigure}{0.49\linewidth}
    \includegraphics[width=\linewidth]{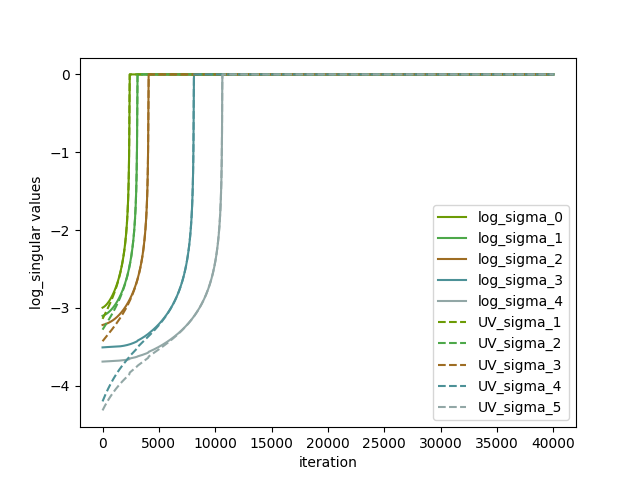} 
\end{subfigure}
\caption{Dynamics of singular values (log scale) for an identity target matrix. From left to right, up to down: real initialization with $\det(U^\top V) = 1$, $\det(U^\top V) = -1$, and complex initialization. }
\label{fig: balanced init, log singular values, identity target}
\end{figure*}

\begin{figure*}[htbp]
\centering
\begin{subfigure}{0.49\linewidth}
    \includegraphics[width=\linewidth]{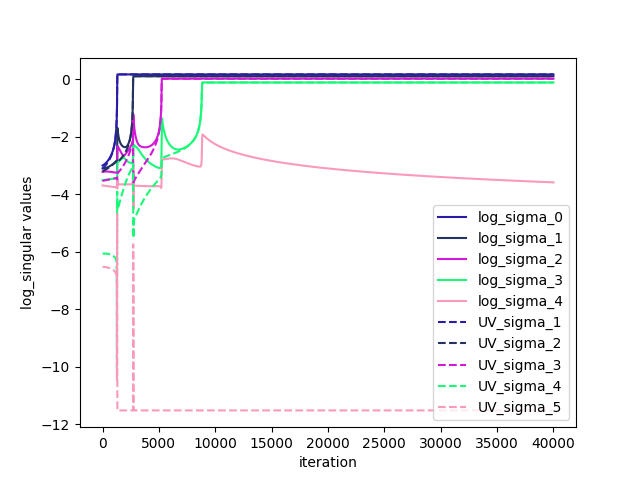} 
\end{subfigure}
\begin{subfigure}{0.49\linewidth}
    \includegraphics[width=\linewidth]{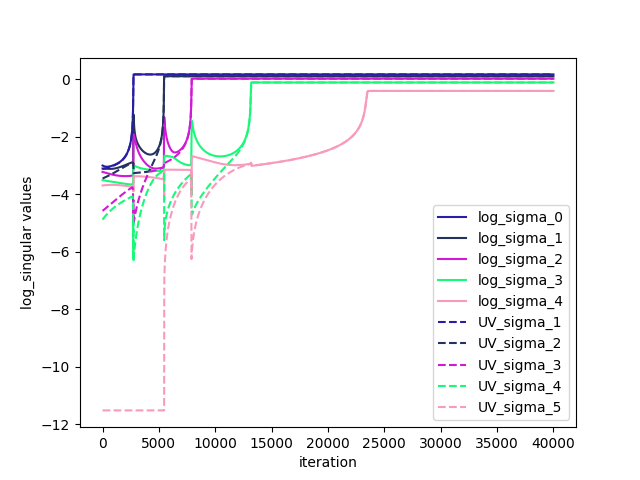} 
\end{subfigure}
\begin{subfigure}{0.49\linewidth}
    \includegraphics[width=\linewidth]{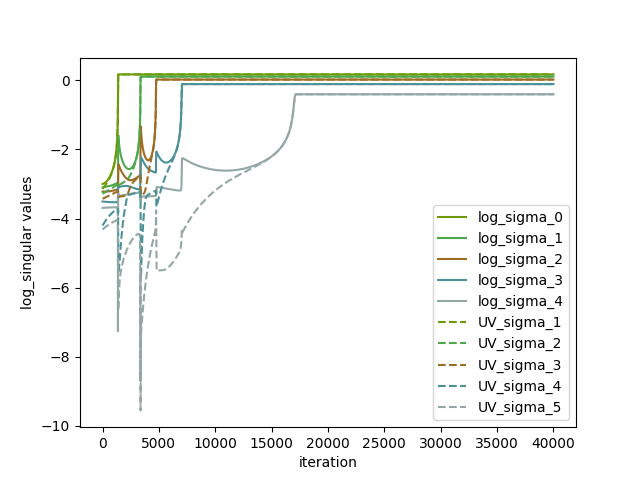} 
\end{subfigure}
\caption{Dynamics of singular values (log scale) for a non-identity target matrix. From left to right, up to down: real initialization with $\det(U^\top V) = 1$, $\det(U^\top V) = -1$, and complex initialization. }
\label{fig: balanced init, log singular values, non-identity target}
\end{figure*}

\subsection{Alignment dynamics under balance regularization term}

This section exhibits the dynamics of weight matrices under regularization term. The original square loss $\mathcal{L}_{\rm ori}$ is omitted. Here $a=1$, $\epsilon=1$, $\eta=0.001$. 

Figure \ref{fig: random init, log singular values, max-min} illustrates the conclusion of Theorem \ref{maximum and minimum singular values are irrelevant of the regularization term} and \ref{maximum and minimum singular values are irrelevant of the regularization term, GD}. Clearly the maximum among all the singular values are non-increasing while the minimum is non-decreasing. 

\begin{figure*}[htbp]
\centering
\begin{subfigure}{0.49\linewidth}
    \includegraphics[width=\linewidth]{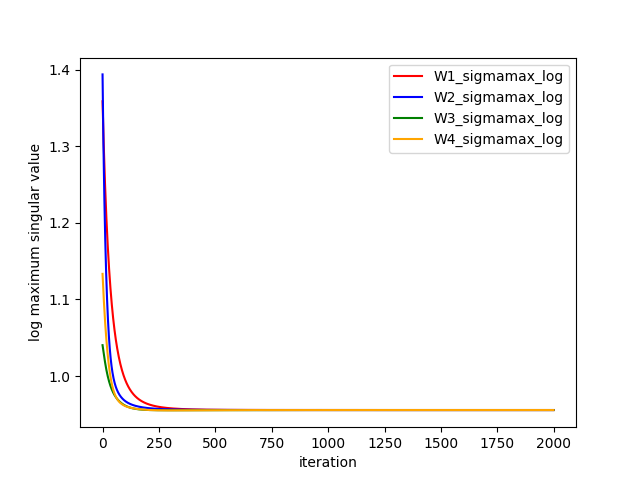} 
\end{subfigure}
\begin{subfigure}{0.49\linewidth}
    \includegraphics[width=\linewidth]{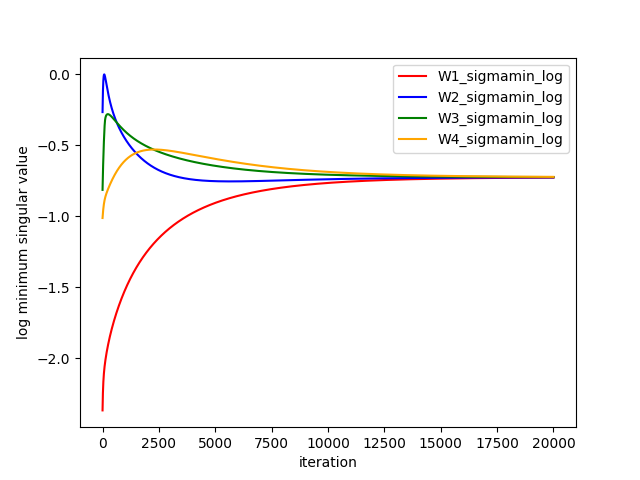} 
\end{subfigure}
\caption{Dynamics of extreme singular values (log scale) for four weight matrices. }
\label{fig: random init, log singular values, max-min}
\end{figure*}

Figure \ref{fig: random init, log singular values, main term} illustrates the dynamics of main term $\sigma_{\min}(W_1 + W_2^{-1} W_3^H W_4^H)$. For real initialization with $\det(W(0))<0$, $\sigma_{\min}(W_1 + W_2^{-1} W_3^H W_4^H)$ decays to 0 at a linear rate, while for $\det(W(0))>0$ and complex initialization it stays at a small value after some oscillation. 

\begin{figure*}[htbp]
\centering
\begin{subfigure}{0.49\linewidth}
    \includegraphics[width=\linewidth]{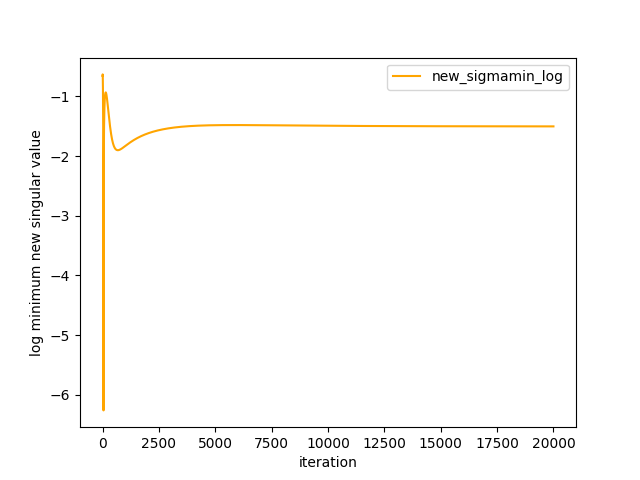} 
\end{subfigure}
\begin{subfigure}{0.49\linewidth}
    \includegraphics[width=\linewidth]{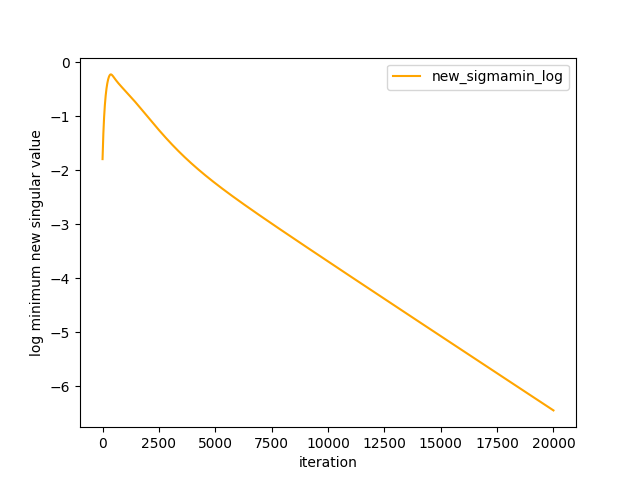} 
\end{subfigure}
\begin{subfigure}{0.49\linewidth}
    \includegraphics[width=\linewidth]{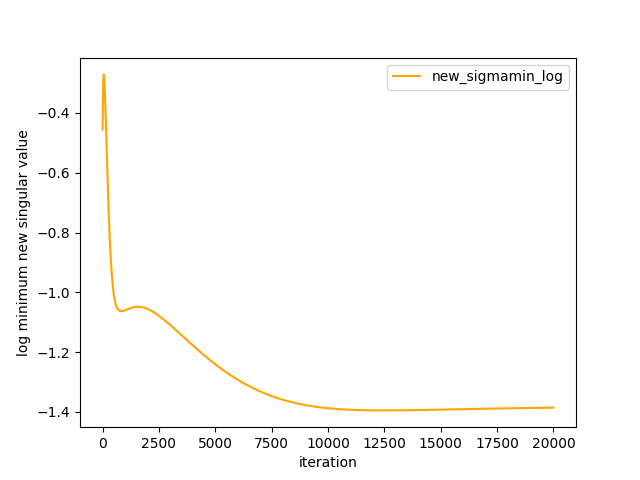} 
\end{subfigure}
\caption{Dynamics of the minimum singular value of hermitian main term $W_1 + W_2^{-1} W_3^H W_4^H$ (log scale). From left to right, up to down: real initialization with $\det(W) > 0$, $\det(W) < 0$, and complex initialization. }
\label{fig: random init, log singular values, main term}
\end{figure*}

%% file: subfiles/appendix/appendix999.LLM_usage_declaration.tex
\newpage

\section{LLM usage declaration}

In the preparation of this paper, large language models (LLMs) served only as an auxiliary tool for enhancing writing clarity, checking grammar, and assisting in the drafting and debugging of simulation code. These tasks were performed under the authors' complete oversight. The central scientific ideas, theoretical results, and research contributions are entirely the work of the authors. 

%% file: iclr2026_conference.bib
@misc{jain2017globalconvergencenonconvexgradient,
      title={Global Convergence of Non-Convex Gradient Descent for Computing Matrix Squareroot}, 
      author={Prateek Jain and Chi Jin and Sham M. Kakade and Praneeth Netrapalli},
      year={2017},
      eprint={1507.05854},
      archivePrefix={arXiv},
      primaryClass={math.NA},
      url={https://arxiv.org/abs/1507.05854}, 
}

@misc{li2019algorithmicregularizationoverparameterizedmatrix,
      title={Algorithmic Regularization in Over-parameterized Matrix Sensing and Neural Networks with Quadratic Activations}, 
      author={Yuanzhi Li and Tengyu Ma and Hongyang Zhang},
      year={2019},
      eprint={1712.09203},
      archivePrefix={arXiv},
      primaryClass={cs.LG},
      url={https://arxiv.org/abs/1712.09203}, 
}

@article{Chen_2019,
   title={Gradient descent with random initialization: fast global convergence for nonconvex phase retrieval},
   volume={176},
   ISSN={1436-4646},
   url={http://dx.doi.org/10.1007/s10107-019-01363-6},
   DOI={10.1007/s10107-019-01363-6},
   number={1–2},
   journal={Mathematical Programming},
   publisher={Springer Science and Business Media LLC},
   author={Chen, Yuxin and Chi, Yuejie and Fan, Jianqing and Ma, Cong},
   year={2019},
   month=feb, pages={5–37} }

@misc{lee2016gradientdescentconvergesminimizers,
      title={Gradient Descent Converges to Minimizers}, 
      author={Jason D. Lee and Max Simchowitz and Michael I. Jordan and Benjamin Recht},
      year={2016},
      eprint={1602.04915},
      archivePrefix={arXiv},
      primaryClass={stat.ML},
      url={https://arxiv.org/abs/1602.04915}, 
}

@misc{tu2016lowranksolutionslinearmatrix,
      title={Low-rank Solutions of Linear Matrix Equations via Procrustes Flow}, 
      author={Stephen Tu and Ross Boczar and Max Simchowitz and Mahdi Soltanolkotabi and Benjamin Recht},
      year={2016},
      eprint={1507.03566},
      archivePrefix={arXiv},
      primaryClass={math.OC},
      url={https://arxiv.org/abs/1507.03566}, 
}

@misc{ge2017spuriouslocalminimanonconvex,
      title={No Spurious Local Minima in Nonconvex Low Rank Problems: A Unified Geometric Analysis}, 
      author={Rong Ge and Chi Jin and Yi Zheng},
      year={2017},
      eprint={1704.00708},
      archivePrefix={arXiv},
      primaryClass={cs.LG},
      url={https://arxiv.org/abs/1704.00708}, 
}

@misc{du2018algorithmicregularizationlearningdeep,
      title={Algorithmic Regularization in Learning Deep Homogeneous Models: Layers are Automatically Balanced}, 
      author={Simon S. Du and Wei Hu and Jason D. Lee},
      year={2018},
      eprint={1806.00900},
      archivePrefix={arXiv},
      primaryClass={cs.LG},
      url={https://arxiv.org/abs/1806.00900}, 
}

@misc{ye2021globalconvergencegradientdescent,
      title={Global Convergence of Gradient Descent for Asymmetric Low-Rank Matrix Factorization}, 
      author={Tian Ye and Simon S. Du},
      year={2021},
      eprint={2106.14289},
      archivePrefix={arXiv},
      primaryClass={math.OC},
      url={https://arxiv.org/abs/2106.14289}, 
}

@misc{bartlett2018gradientdescentidentityinitialization,
      title={Gradient descent with identity initialization efficiently learns positive definite linear transformations by deep residual networks}, 
      author={Peter L. Bartlett and David P. Helmbold and Philip M. Long},
      year={2018},
      eprint={1802.06093},
      archivePrefix={arXiv},
      primaryClass={cs.LG},
      url={https://arxiv.org/abs/1802.06093}, 
}

@misc{arora2019convergenceanalysisgradientdescent,
      title={A Convergence Analysis of Gradient Descent for Deep Linear Neural Networks}, 
      author={Sanjeev Arora and Nadav Cohen and Noah Golowich and Wei Hu},
      year={2019},
      eprint={1810.02281},
      archivePrefix={arXiv},
      primaryClass={cs.LG},
      url={https://arxiv.org/abs/1810.02281}, 
}

@misc{ji2019gradientdescentalignslayers,
      title={Gradient descent aligns the layers of deep linear networks}, 
      author={Ziwei Ji and Matus Telgarsky},
      year={2019},
      eprint={1810.02032},
      archivePrefix={arXiv},
      primaryClass={cs.LG},
      url={https://arxiv.org/abs/1810.02032}, 
}

@misc{arora2019implicitregularizationdeepmatrix,
      title={Implicit Regularization in Deep Matrix Factorization}, 
      author={Sanjeev Arora and Nadav Cohen and Wei Hu and Yuping Luo},
      year={2019},
      eprint={1905.13655},
      archivePrefix={arXiv},
      primaryClass={cs.LG},
      url={https://arxiv.org/abs/1905.13655}, 
}

@misc{kawaguchi2016deeplearningpoorlocal,
      title={Deep Learning without Poor Local Minima}, 
      author={Kenji Kawaguchi},
      year={2016},
      eprint={1605.07110},
      archivePrefix={arXiv},
      primaryClass={stat.ML},
      url={https://arxiv.org/abs/1605.07110}, 
}

@misc{guberman2016complexvaluedconvolutionalneural,
      title={On Complex Valued Convolutional Neural Networks}, 
      author={Nitzan Guberman},
      year={2016},
      eprint={1602.09046},
      archivePrefix={arXiv},
      primaryClass={cs.NE},
      url={https://arxiv.org/abs/1602.09046}, 
}

@article{Bunse1991/92,
author = {Bunse-Gerstner, A. and Byers, R. and Mehrmann, V. and Nichols, N.\ K.},
journal = {Numerische Mathematik},
number = {1},
pages = {1-40},
title = {Numerical computation of an analytic singular value decomposition of a matrix valued function.},
url = {http://eudml.org/doc/133582},
volume = {60},
year = {1991/92},
}

@techreport{demoor1989analyticProperties,
  author      = {De Moor, B. and Boyd, S.},
  title       = {Analytic properties of singular values and vectors},
  institution = {ESAT-SISTA, Department of Electrical Engineering, KU Leuven},
  type        = {Technical Report},
  number      = {1989-28},
  year        = {1989},
  url         = {http://ftp.esat.kuleuven.be/pub/sista/ida/reports/89-28.pdf},
}

@article{dyson1962threefold,
  author  = {Dyson, Freeman J.},
  title   = {The Threefold Way: Algebraic Structure of Symmetry Groups and Ensembles in Quantum Mechanics},
  journal = {Journal of Mathematical Physics},
  volume  = {3},
  number  = {6},
  pages   = {1199--1215},
  year    = {1962},
}

@article{girko1985distributionOrthogonal,
  author       = {Girko, V. L.},
  title        = {Distribution of eigenvalues and eigenvectors of orthogonal random matrices},
  journal      = {Ukrainian Mathematical Journal},
  volume       = {37},
  pages        = {457--463},
  year         = {1985},
  month        = {September},
  doi          = {10.1007/BF01061167},
}

@misc{tao2009randommatricesdistributionsmallest,
      title={Random matrices: The distribution of the smallest singular values}, 
      author={Terence Tao and Van Vu},
      year={2009},
      eprint={0903.0614},
      archivePrefix={arXiv},
      primaryClass={math.PR},
      url={https://arxiv.org/abs/0903.0614}, 
}

@book{taotopics,
  title={Topics in Random Matrix Theory},
  author={Tao, T.},
  isbn={9780821885079},
  series={Graduate studies in mathematics},
  url={https://books.google.com/books?id=Hjq_JHLNPT0C},
  publisher={American Mathematical Soc.}
}

@inbook{Vershynin_2018, 
  place={Cambridge}, 
  series={Cambridge Series in Statistical and Probabilistic Mathematics}, 
  title={Random Matrices}, 
  booktitle={High-Dimensional Probability: An Introduction with Applications in Data Science}, 
  publisher={Cambridge University Press}, 
  author={Vershynin, Roman}, 
  year={2018}, 
  pages={70–97}, 
  collection={Cambridge Series in Statistical and Probabilistic Mathematics}
}

@book{bhatia1996matrix,
  title={Matrix Analysis},
  author={Bhatia, R.},
  isbn={9780387948461},
  lccn={96032217},
  series={Graduate Texts in Mathematics},
  url={https://books.google.co.uk/books?id=F4hRy1F1M6QC},
  year={1996},
  publisher={Springer New York}
}

@book{forrester2010log,
  title={Log-Gases and Random Matrices (LMS-34)},
  author={Forrester, P.J.},
  isbn={9781400835416},
  lccn={2009053314},
  series={London Mathematical Society Monographs},
  url={https://books.google.com/books?id=C7z3NgOlb1gC},
  year={2010},
  publisher={Princeton University Press}
}

@book{szegő1939orthogonal,
  title={Orthogonal Polynomials},
  author={Szeg{\H{o}}, G.},
  lccn={39033497},
  series={American Mathematical Society colloquium publications},
  url={https://books.google.com/books?id=w755xgEACAAJ},
  year={1939},
  publisher={American mathematical society}
}

@article{jiang2023algorithmic,
  title={Algorithmic regularization in model-free overparametrized asymmetric matrix factorization},
  author={Jiang, Liwei and Chen, Yudong and Ding, Lijun},
  journal={SIAM Journal on Mathematics of Data Science},
  volume={5},
  number={3},
  pages={723--744},
  year={2023},
  publisher={SIAM}
}

@article{chou2024gradient,
  title={Gradient descent for deep matrix factorization: Dynamics and implicit bias towards low rank},
  author={Chou, Hung-Hsu and Gieshoff, Carsten and Maly, Johannes and Rauhut, Holger},
  journal={Applied and Computational Harmonic Analysis},
  volume={68},
  pages={101595},
  year={2024},
  publisher={Elsevier}
}

@article{wang2023implicit,
  title={Implicit bias of SGD in $ L\_ $\{$2$\}$ $-regularized linear DNNs: One-way jumps from high to low rank},
  author={Wang, Zihan and Jacot, Arthur},
  journal={arXiv preprint arXiv:2305.16038},
  year={2023}
}

@article{hardt2016identity,
  title={Identity matters in deep learning},
  author={Hardt, Moritz and Ma, Tengyu},
  journal={arXiv preprint arXiv:1611.04231},
  year={2016}
}

@article{zheng2016convergence,
  title={Convergence analysis for rectangular matrix completion using Burer-Monteiro factorization and gradient descent},
  author={Zheng, Qinqing and Lafferty, John},
  journal={arXiv preprint arXiv:1605.07051},
  year={2016}
}

@inproceedings{park2017non,
  title={Non-square matrix sensing without spurious local minima via the Burer-Monteiro approach},
  author={Park, Dohyung and Kyrillidis, Anastasios and Carmanis, Constantine and Sanghavi, Sujay},
  booktitle={Artificial Intelligence and Statistics},
  pages={65--74},
  year={2017},
  organization={PMLR}
}

@inproceedings{du2019width,
  title={Width provably matters in optimization for deep linear neural networks},
  author={Du, Simon and Hu, Wei},
  booktitle={International Conference on Machine Learning},
  pages={1655--1664},
  year={2019},
  organization={PMLR}
}

@article{chizat2024infinite,
  title={Infinite-width limit of deep linear neural networks},
  author={Chizat, L{\'e}na{\"\i}c and Colombo, Maria and Fern{\'a}ndez-Real, Xavier and Figalli, Alessio},
  journal={Communications on Pure and Applied Mathematics},
  volume={77},
  number={10},
  pages={3958--4007},
  year={2024},
  publisher={Wiley Online Library}
}

@inproceedings{min2023convergence,
  title={On the convergence of gradient flow on multi-layer linear models},
  author={Min, Hancheng and Vidal, Ren{\'e} and Mallada, Enrique},
  booktitle={International Conference on Machine Learning},
  pages={24850--24887},
  year={2023},
  organization={PMLR}
}

@inproceedings{min2021explicit,
  title={On the explicit role of initialization on the convergence and implicit bias of overparametrized linear networks},
  author={Min, Hancheng and Tarmoun, Salma and Vidal, Ren{\'e} and Mallada, Enrique},
  booktitle={International Conference on Machine Learning},
  pages={7760--7768},
  year={2021},
  organization={PMLR}
}

@article{xiong2023over,
  title={How over-parameterization slows down gradient descent in matrix sensing: The curses of symmetry and initialization},
  author={Xiong, Nuoya and Ding, Lijun and Du, Simon S},
  journal={arXiv preprint arXiv:2310.01769},
  year={2023}
}

@inproceedings{tarmoun2021understanding,
  title={Understanding the dynamics of gradient flow in overparameterized linear models},
  author={Tarmoun, Salma and Franca, Guilherme and Haeffele, Benjamin D and Vidal, Rene},
  booktitle={International Conference on Machine Learning},
  pages={10153--10161},
  year={2021},
  organization={PMLR}
}
